\documentclass[a4paper,10pt]{amsart}

\usepackage{amssymb}
\usepackage{amsmath}
\usepackage{amsfonts}
\usepackage{mathrsfs}
\usepackage[all]{xy}
\usepackage{hyperref}

\usepackage{graphicx}
\graphicspath{ {images/} }

\usepackage{verbatim}
\usepackage{mathrsfs} 
\usepackage{tikz}
\usetikzlibrary{matrix,arrows,shadings}
\pdfpageattr{/Group <</S /Transparency /I true /CS /DeviceRGB>>} 

\theoremstyle{plain}
\newtheorem{theorem}{Theorem}[section]
\newtheorem{lemme}[theorem]{Lemma}
\newtheorem{lemma}[theorem]{Lemma}
\newtheorem{proposition}[theorem]{Proposition}
\newtheorem{corollaire}[theorem]{Corollary}
\newtheorem{corollary}[theorem]{Corollary}
\newtheorem*{theorem*}{Theorem}

\theoremstyle{definition}
\newtheorem{definition}[theorem]{Definition}

\newtheorem{example}[theorem]{Example}

\newtheorem*{lemA}{Lemma A}
\newtheorem*{thmB}{Theorem B}
\newtheorem*{thmC}{Theorem C}

\theoremstyle{remark}
\newtheorem{remarque}[theorem]{Remark}
\newtheorem{remark}[theorem]{Remark}

\def\ZZ{{\mathbb Z}}
\def\NN{{\mathbb N}}
\def\QQ{{\mathbb Q}}

\def\D{{\mathscr{D}}}
\def\EE{{\mathscr{E}}}
\def\TT{{\mathbb{T}}}
\def\G{{\mathbb{G}}}
\def\dim{{\rm dim}}
\def\CP{{\rm CPDiv}}
\def\tail{{\Sigma_{\star}}}
\def\face{{\rm face}}
\def\supp{{\rm supp}}
\def\hs{{\mathscr{Q}}}
\def\P{{\mathbb{P}}}
\def\AA{{\mathbb{A}}}
\def\spec{{\rm Spec}}
\def\Hom{{\rm Hom}}
\def\aa{{\mathbf{A}}}
\def\ss{{\mathscr{S}}}
\def\F{{\mathscr{F}}}
\def\sc{{\mathrm{Mod}}}
\def\gsc{{\mathrm{Mod}_{G}}}
\def\Vc{{\mathcal{V}}}
\def\Vh{{\mathscr{Q}_{\Sigma}}}
\def\loc{{\mathrm{Loc}}}
\def\lc{{\star}}
\def\SL{{\mathrm{SL}}}
\def\Ray{{\mathrm{Ray}}}
\def\Vert{{\mathrm{Vert}}}
\def\Div{{\mathrm{Div}}}
\def\Cl{{\mathrm{Cl}}}
\def\div{{\mathrm{div}}}
\def\TT{{\mathbb{T}}}
\def\XX{{\mathscr{X}}}
\def\UU{{\mathscr{U}}}

\def\YY{S}
\def\cont{{c}}
\def\va{{s}}
\def\geo{\varsigma }

\def\tail{{\mathrm{Tail}}}
\def\Aut{{\mathrm{Aut}}}
\def\face{{\rm face}}
\def\fsc{{\mathrm{Mod}_{F}}}

\title[Classification of $G$-varieties]
{On the classification of normal $G$-varieties\\  with spherical orbits}
\author{Kevin Langlois}
\date{}

\address{Mathematisches Institut, Heinrich Heine Universit\"at, 40225 D\"usseldorf, Germany.}
\email{langlois.kevin18@gmail.com}

\begin{document}

\begin{abstract}
In this article, we investigate the geometry of reductive group actions on
algebraic varieties.  
Given a connected reductive group $G$, we elaborate on a geometric and combinatorial approach based on Luna-Vust theory to describe every normal $G$-variety with spherical orbits. This description encompasses the classical case of spherical varieties and the theory of $\TT$-varieties recently introduced by Altmann, Hausen, and S\"uss.  
\\

\textsc{R\'esum\'e}. Dans cet article, nous \'etudions la g\'eom\'etrie des op\'erations de groupes r\'eductifs dans les vari\'et\'es alg\'ebriques. \'Etant donn\'e un groupe alg\'ebrique r\'eductif connexe $G$, nous \'elaborons une approche g\'eom\'etrique et combinatoire bas\'ee sur la th\'eorie de Luna-Vust pour d\'ecrire toute $G$-vari\'et\'e normale avec orbites sph\'eriques. Cette description comprend le cas classique des vari\'et\'es sph\'eriques et la th\'eorie des $\mathbb{T}$-vari\'et\'es introduite r\'ecemment par Altmann, Hausen et S\"uss.  

\end{abstract}  
\thanks{\em 2010 Mathematics Subject Classification \rm 14L30, 14M27, 14M25,  	13A18. 
\\ \em Key Words and Phrases: \rm action of algebraic groups, Luna-Vust theory, homogeneous spaces, valuation theory.}

\maketitle

\tableofcontents
\section*{Introduction}
Throughout this article, we consider algebraic varieties and algebraic groups over an algebraically closed field $k$ of characteristic $0$.
\\

{\bf Goal.} Toric varieties are known to provide applications to test conjectures
and general theories. They naturally come with a combinatorial description encoding their geometric properties in terms of simple objects of convex geometry (e.g.
fans, monoids, polytopes etc.). The approach that is developed here aims to establish a similar dictionary for a larger class of algebraic varieties. More precisely, considering a connected reductive group $G$, the goal of the paper is to study the classification of normal $G$-varieties with spherical orbits. We
propose a geometric and combinatorial construction of these $G$-varieties which also generalizes the classical examples of spherical varieties (case of an open $G$-orbit) \cite{Kno91}
and of normal varieties with a torus action (case where the acting group $G$ is a torus) \cite{AH06, AHS08}, see Theorems B, C. 
Moreover, for such reductive group actions new results are obtained, see  Theorems \ref{t-divisor1}, \ref{t-canonical}.
\\

{\bf Context.}
Before stating our results, 
we start by recalling the definitions and some basic facts on reductive group actions. Let us fix a Borel subgroup $B\subseteq G$ and a $G$-variety $X$. The \emph{complexity}
 (cf \cite{Vin86}) of the $G$-action on $X$, denoted by $c(X)$, is defined as the transcendence degree over $k$ of the field of $B$-invariant functions $k(X)^{B}$. This number does not depend on the choice of the Borel subgroup and corresponds by a result of Rosenlicht (see \cite{Ros63}, \cite[Satz 2.2]{Spr89}) to the codimension of a $B$-orbit of $X$ in general position.
 The complexity has a remarkable property, namely, if $Z\subseteq X$ is an irreducible $G$-stable closed subvariety, then $c(Z)\leq c(X)$ \cite[Theorem 5.7]{Tim11}. 
A \emph{spherical $(G$-$)$variety} is a normal $G$-variety of complexity $0$.

We will say that the $G$-variety $X$ has a \emph{stabilizer in general position} (or has a (unique) general orbit) if there exist a $G$-stable dense open subset $X_{1}\subseteq X$ and a closed subgroup $H\subseteq G$ such that for any $x \in X_{1}$ the isotropy group $G_{x}$ is conjugate to $H$. In other words, this means that
each $G$-orbit of $X_{1}$ is $G$-isomorphic to the homogeneous space $G/H$.

 Moreover, the $G$-variety $X$ is said to have \emph{trivial equivariant birational type}  if there exist a variety $S$, a homogeneous $G$-space $G/H$
and a $G$-equivariant birational map $X\dashrightarrow S\times G/H$, where $G$ acts on the product $S\times G/H$ by the trivial action on the first factor, and with the usual action on the second one. By a result of Richardson (see \cite{Ric72})
any smooth affine $G$-variety has a stabilizer in general position. In particular, this applies to the case of finite dimensional rational representations of $G$. 
In addition, this is also true for \emph{$G$-varieties with spherical orbits} which are, by definition, the $G$-varieties whose all $G$-orbits are spherical. More precisely, the authors in \cite{AB05} have shown that for any $G$-scheme $\mathcal{X}$ of finite type over $k$, there exist a finite number of conjugacy classes of isotropy groups of $\mathcal{X}$ giving rise to spherical orbits (see \cite[Theorem 3.1]{AB05}). 

The classification of algebraic varieties in algebraic geometry has an equivariant analogue, namely one can distinguish two types of classification: 
\begin{itemize}
\item[(1)] One determines a natural representative for each $G$-equivariant birational class.  
\item[(2)] Given a $G$-variety $S$, one studies (or classifies) the $G$-isomorphism classes of $G$-varieties $X$ which are $G$-equivariantly birational to $S$.
\end{itemize}
Note that in general, we restrict ourselves to the case where the $G$-varieties are normal; this is the viewpoint that we will adopt. In this case, following the notation of (2), we will say that $X$ is a \emph{$G$-model} of $S$.

Several general approaches were given to study this classification problem. For the type (1), it can be reformulated in terms of the relative Galois cohomology using the space of quasi-sections of a $G$-variety (see \cite[Paragraph 2.5]{PV89}).  A description for the type (2) was obtained by Luna-Vust in the setting of embeddings of homogeneous spaces (cf. \cite{LV83}). It turns out to be effective in the case where the acting connected linear algebraic group is reductive and the complexity is $\leq 1$. A generalization for reductive group actions can be found in \cite{Kno93b}, \cite[Section 1]{Tim97}. The techniques for their classifications are based on commutative algebra, especially on the theory of valuations.   

Reductive group actions in small complexity still attract great interest. The class of spherical varieties comprises several important examples appearing in algebraic geometry and representation theory. It includes the classes of toric varieties (cf \cite{Ful93}), flag varieties, horospherical varieties \cite{PV72, Pau81, Pas08}, embeddings of symmetric spaces (cf \cite{Sat60, Vus90}), determinantal varieties,
wonderful compactifications (cf \cite{CP83, Lun96}), and many others. We refer to \cite{Kno91} for a description of spherical varieties in terms of colored fans and \cite{Tim97} for a generalization to the case of complexity one. Together they gave  a complete classification for the type (2) in the case where the complexity is $\leq 1$. 

The description of a spherical variety $X$ in \cite{Kno91} uses the geometric structure of its open orbit and especially the (equivariant) birational invariants attached to it, namely their lattice, colors and $G$-valuations (see Section \ref{s:spherical-sub} for a reminder). These invariants form the \emph{colored equipment}
of the spherical variety $X$.

In \cite{Lun01} Luna proposed a description of the spherical homogeneous spaces and their colored equipment in terms of the root system of $G$ and gave a complete classification when $G$ is of type A. While Losev showed that the correspondence
between a $G$-isomorphism class of a spherical homogeneous $G$-space and its colored equipment is injective (see \cite{Los09}), the conjectural general case following Luna is reformulated from a conjecture based on the classification of wonderful varieties \cite[Section 2]{Lun01}. This problem was recently solved by the efforts of several authors (see for instance \cite{Akh83, Was96,Cup,BP14,BP15,BP16}) and completes the classification of type (1) for spherical varieties. Note that there exist alternative solutions for this problem; see for instance \cite{Avd11, Avd16} for a classification in the special case of spherical subgroups contained in a Borel subgroup.

In 2006, Altmann and Hausen have developed a new theory for describing torus actions on normal affine varieties in the setting of arbitrary complexity \cite{AH06}. Their idea involves the geometry of line bundles on a normal variety $\Gamma$ (this later playing the role of a certain quotient for the torus action) and the combinatorics coming from toric geometry. This description specializes to the known cases when the acting torus $\TT$ is of dimension one (see \cite{Dol75, Pin77, Dem88, FZ03}) and intersects with the description for complexity-one reductive group actions
(see \cite{KKMS73, Tim97, Tim08, Lan15}). 

It is well known (see for instance \cite[Section 2]{FZ03}) that one can construct a normal surface $X$ with a $\G_{m}$-action by considering the affine cone of a smooth projective algebraic curve $\Gamma$ (modulo an appropriate action of a finite group). In this case, the algebra of regular functions $k[X]$ is described by a $\QQ$-divisor $D$ having positive degree on $\Gamma$ via the equality 
$$k[X] = \bigoplus_{m\geq 0}H^{0}(\Gamma, \mathcal{O}_{\Gamma}(\lfloor mD\rfloor)).$$
For instance (see \cite[Example 3.6]{Dem88}), if $D$ is the divisor $\frac{1}{2}[0] - \frac{1}{3}[1] - \frac{1}{7}[\infty]$ over $\P^{1}$, then we recover the hypersurface $x^{2} + y^{3} + z^{7} = 0$ in $\AA^{3}$.  

The formalism of \emph{polyhedral divisors} introduced in \cite{AH06} is a generalization of this phenomenon for the multigraded case where we consider instead of a $\QQ$-divisor a piecewise linear map
$$\sigma^{\vee}\rightarrow {\rm CaDiv}_{\QQ}(\Gamma),\, m\mapsto \D(m)=\sum_{Y\subseteq \Gamma}\min_{v\in\D_{Y}}\langle m ,v\rangle\cdot Y$$
for describing the algebra of functions of an affine normal $\TT$-variety. The set     
$\sigma^{\vee}$ is a polyhedral cone living in the vector space $M_{\QQ}$ generated by the character lattice of the torus $\TT$ and ${\rm CaDiv}_{\QQ}(\Gamma)$ is the vector space of Cartier $\QQ$-divisors on the normal variety $\Gamma$. The polyhedron $\D_{Y}\subseteq N_{\QQ} = {\rm Hom}(M_{\QQ}, \QQ)$
is referred to as the coefficient at the prime divisor $Y\subseteq \Gamma$ of the polyhedral divisor $\D$, see the reminder in Section \ref{s-chi}.

The generalization to the setting of normal $\TT$-varieties presented in \cite{AHS08} involves considering a certain finite family $\EE$ of polyhedral divisors $\{\D^{i},\,i\in I\}$, called the \emph{divisorial fan}, and defined on a common normal variety $\Gamma$ (see \cite[Definition 5.2]{AHS08} for a precise definition). Here the coefficient family $\{\D^{i}_{Y},\,i\in I\}$ forms a polyhedral subdivision. In particular, this notion collapses to the notion of the defining fan of a toric variety when the complexity of the torus action is $0$.
\\

{\bf Main results.} In this article, we generalize the combinatorial description in \cite{AH06, AHS08} for torus actions and the one of spherical varieties coming from the Luna-Vust theory \cite{Kno91} to the setting of normal $G$-varieties with spherical orbits. We first collect some classical results about the equivariant birational type of such varieties. One knows for instance by \cite[Theorem 2.13]{CKPR11} that a $G$-variety having a stabilizer in general position has trivial equivariant birational type after making an \'etale base change on a $G$-stable dense open subset (see also the reminder in \ref{t-CKPR}). 

A morphism $\pi:Z_{1}\rightarrow Z_{2}$ between two varieties is called a \emph{Galois covering}
if it is dominant finite and the field extension $k(Z_{1})/\pi^{\star}k(Z_{2})$ is Galois. By using tools from the Luna-Vust theory, one can obtain the following intermediate result as a consequence of \cite[Theorem 2.13]{CKPR11} (see Corollary \ref{t-Galois} for a proof). 
\begin{lemA}\label{t-zero}
Let $X$ be a normal $G$-variety with stabilizer $H$ in general position. Then there exist a normal $G$-variety $\tilde{X}$ with general stabilizer $H$ having trivial equivariant birational type and a $G$-equivariant Galois covering $\tilde{X}\rightarrow X$.
\end{lemA}
In other words, this means that $\tilde{X}$ admits a generically free $G$-equivariant action of a finite group $F$, the quotient $\tilde{X}/F$ exists
and is identified with $X$. Note that a version of Lemma A was originally shown by Arzhantsev (see \cite[Section 3, Proposition 3]{Arz97a}) for certain affine $G$-varieties of complexity one with spherical orbits. Lemma A also gives a concrete picture for the classification of type (1) of normal $G$-varieties with spherical orbits. Indeed, it reduces the classification to the determination of birational models of the rational quotient of the total space $\tilde{X}$, the description of the general spherical orbit in terms of the Luna theory (cf. \cite{Lun01}), and the description of a certain $G$-equivariant finite group action which is completely determined by a Galois cohomology class (see Corollary \ref{c:cohoclass} for more details). This will allow us to deal first with the trivial equivariant birational case and then to go back to the general case via the finite group action. 

We now explain how to combinatorially describe a normal $G$-variety with spherical orbits. Our main motivation is provided by the case of a toric variety $V$ defined by a fan $\EE_{V}$. In this setting $G= B = \TT$ is an algebraic torus
and elements of $\EE_{V}$ are strictly convex polyhedral cones $\sigma\subseteq N_{\QQ}$ living in the vector space generated by the lattice $N$ of one-parameter subgroups of $\TT$. Each element of $\sigma\in\EE_{V}$ determines a $\TT$-stable dense open affine subset $V_{\sigma}\subseteq V$, and vice-versa. Moreover, one can recover $\sigma$ as the set of the discrete (geometric) valuations $v:k(V)^{\star}\rightarrow \QQ$ centered in the generic point of a $\TT$-stable irreducible closed subvariety of $V_{\sigma}$. From this viewpoint, we have the equality 
$k[V_{\sigma}] = k[\TT]\cap \bigcap_{v\in\sigma}\mathcal{O}_{v},$
where $\mathcal{O}_{v}$ is the valuation ring associated with $v$. The Luna-Vust theory aims to exploit this observation in the general setting of reductive group actions. 
Note that the analogous notion for the $V_{\sigma}$'s in the context of an arbitrary $G$-variety is the notion of simplicity. Let us recall that a \emph{$B$-chart} of a $G$-variety is a $B$-stable dense affine open subset. The $G$-variety is said to be \emph{simple\footnote{In case of spherical varieties, the notion of simple $G$-varieties coincides with the usual one, namely to require to have a unique closed $G$-orbit (see Theorem 3.1 and results in Section 2 of \cite{Kno91}).}}
if it has a $B$-chart intersecting any $G$-orbit. According to a result of Sumihiro (see \cite[Theorem 1]{Sum74}, \cite[Theorem 1.3]{Kno91}), any normal $G$-variety is covered by $G$-translates of $B$-charts, or equivalently, admits a finite open covering of simple $G$-varieties. 

Let $X$ be a normal $G$-variety with spherical orbits having trivial equivariant birational type and consider a $B$-chart $X_{0}$.
Our first main result (see Theorem B below) is a description of $X_{0}$ in terms of a pair $(\D,\F)$ called a \emph{colored polyhedral divisor} (see Definition \ref{d-poly}). The symbol $\D$ denotes a polyhedral divisor 
defined on a certain birational model $\Gamma$ of the rational quotient of $X$ by $G$. More precisely, the polyhedral divisor $\D$ describes the algebra of $U$-invariants $k[X_{0}]^{U}$ graded by the $B$-eigencharacters of $k(X)$,
where $U$ is the unipotent radical of the Borel subgroup $B$. The finite set $\F$ consists of \emph{colors} (i.e., prime $B$-divisors of $X$ that are not $G$-stable) that intersect the open subset $X_{0}$. 

Similarly to the toric case, one can define the $k$-algebra $k[X_{0}]$ by considering valuations of $k(X)$ that are centered in the generic point of a $G$-stable closed irreducible subvariety
 or in a color of $X$, namely that $k[X_{0}]$ is equal to a ring intersection 
$$(k(\Gamma)\otimes_k k[\Omega_{0}])\cap \bigcap_{D\in\F}\mathcal{O}_{v_D}\cap \bigcap_{v\in C(\D)\cap \Vh}\mathcal{O}_{v}\subseteq k(X)$$
depending on the combinatorial datum $(\D,\F)$. Here $\Omega_{0}$ stands for the open $B$-orbit of the general orbit of $X$ and $\Vh$ denotes the set of $G$-valuations of $k(X)$. We refer to Section \ref{s-chi} for further details. To sum up, we have the following statement.     
\begin{thmB}\label{t-un}
Any $B$-chart of a normal $G$-variety $X$ with spherical orbits having trivial equivariant birational type arises from a colored polyhedral divisor and vice-versa.  
\end{thmB}

In the general situation (i.e. $X$ has non-trivial equivariant birational type), we have a $G$-equivariant Galois covering $\gamma: \tilde{X}\rightarrow X$ (see Lemma A). Since the general orbit of $\tilde{X}$ is spherical and therefore its $G$-equivariant automorphism group is a diagonalizable group (see \cite[Section 5.2]{BP87}, \cite[Theorem 6.1]{Kno91}), the Galois 
group $F$ of $\gamma$ can be chosen to be abelian finite. Also, we may assume that $\gamma$ is trivial if and only if its Galois cohomology class is trivial. In this case, $\gamma$ is referred as a \emph{splitting} of $X$ (see Definition \ref{def-split}). 

The geometric and combinatorial approach that we present in this article is to consider a finite set of colored polyhedral divisors $\EE$ that we will call a \emph{colored divisorial fan}. This set encodes the geometry of a $G$-stable open covering by simple $G$-varieties of $\tilde{X}$ and the $F$-action on $\tilde{X}$ (given by the splitting $\gamma$) in order to determine $X$ as the quotient $\tilde{X}/F$. 
In the case where $G$ is a torus $\TT$, the splitting $\gamma$ is the identity map and $\EE$ corresponds exactly to the divisorial fan introduced in \cite{AHS08} for describing normal $\TT$-varieties. Note that we have exactly the same face relations (compare \cite[Definition 5.1]{AHS08} and Theorem \ref{p:glue-pp}). 

In analogy with the toric case, 
the colored divisorial fan $\EE$ will consist of a finite set of colored polyhedral divisors on a common smooth projective variety $\Gamma$ stable under intersection, and such that for all $(\D,\F), (\D',\F')\in \EE$ the corresponding natural maps between $B$-charts
$$X_{0}(\D,\F)\leftarrow X_{0}(\D\cap \D',\F\cap \F')\rightarrow X_{0}(\D',\F')$$
are $B$-equivariant open immersions (see Definition \ref{d:divfan} for more information).  

Moreover, we ask that $\EE$ satisfies the additional condition (see Condition (iii) of loc. cit.) corresponding to the separateness and that any element is $F$-stable for the natural $F$-action on the valuation set (see Definition \ref{d: F-divfan}). To mention this latter property, we will say that $\EE$ is a colored divisorial fan 
defined on the triple $(\Gamma, \ss, \gamma),$ where here $\ss$ denotes the Luna invariant attached to the general $G$-orbit $\Omega$. Considering $\gamma$ as a rational map, it can be defined as the quotient map of a generically free $G$-equivariant $F$-action
on the space $\Gamma\times \Omega$. Thus the datum $\gamma$ encodes a priori the $G$-equivariant birational information. Our classification result yields a 'toric picture' of the classification of type (2) for normal $G$-varieties with spherical orbits that can be stated as follows (see Theorem \ref{t-claf}). 
\begin{thmC}\label{t-deux} 
Colored divisorial fans are the geometric and combinatorial realizations of normal $G$-varieties with spherical orbits. 
For any normal $G$-variety $X$ with spherical orbits there exist a splitting $\gamma:\tilde{X}\rightarrow X$ and a 
 colored divisorial fan $\EE= \EE_{X}$ on a triple $(\Gamma, \ss, \gamma)$ attached to it. 
The $(G\times F)$-variety $\tilde{X}$ is $(G\times F)$-birational to $\Gamma\times \Omega$ and is covered by the $G$-translates of $B$-charts $X_{0}(\D,\F)$ for any colored polyhedral divisor $(\D,\F)$ running through $\EE$. 
Moreover, if we let $$X(\D,\F):= G\cdot X_{0}(\D,\F), \text{ then we have }
X(\D,\F)\cap X(\D',\F') = X(\D\cap \D', \F\cap \F') \text{ in }\tilde{X}$$
for all $(\D,\F), (\D',\F')\in\EE$. 
Each $G$-stable open subvariety $X(\D,\F)$ is $F$-stable. 

Conversely, every colored divisorial fan defines a normal $G$-variety with spherical orbits.
\end{thmC}
By extending the construction of Knop for simple spherical varieties (see the proof of \cite[Theorem 3.1]{Kno91}) to our setting (see Theorem \ref{t-explicit}), one can explicitly define  each simple $G$-variety $X(\D,\F)$ as a locally closed 
$G$-stable subvariety in the projectivization $\P(V)$ of a finite dimensional rational $G$-module $V$ by choosing generators of the multigraded algebra associated with the polyhedral divisor $\D$. Hence the abstract $G$-variety $\tilde{X}$ associated with $\EE$ can be thought as the gluing of those $G$-varieties $X(\D,\F)$ so that 
their intersections agree with the intersections of colored polyhedral divisors. Moreover, in Theorem \ref{t-claf} we characterize the completeness property for normal $G$-varieties with spherical orbits in terms of the colored divisorial fan in the same way as in \cite[Section 7]{AHS08}. 

Following the philosophy in \cite{AHS08}, applications of our construction are expected where the combinatorics of toric varieties have proved their usefulness. We refer the reader to \cite{AIPSV12, Per14} for surveys on the geometry of spherical varieties and $\TT$-varieties. Especially in Theorem \ref{t-divisor1}, we explicitly describe the divisor class group of a normal $G$-variety $X$ with spherical orbits in terms of its colored divisorial fan (see \cite[Theorem 4.22]{FZ03}, \cite[Corollary 3.15]{PS11}, \cite[Corollary 2.12]{LT16} for former cases). Using 
the Riemann-Hurwitz formula for finite coverings of algebraic varieties, we give an explicit description of the Weil divisor $\sharp F\cdot K_{X}$ (see Theorem \ref{t-canonical}), where $K_{X}$ is a canonical divisor of $X$. The calculation is expressed in terms of the ramification indices of the quotient map by $F$. In particular, this formula is a first step toward the classification of Fano varieties in this setting as it was studied in some particular cases in \cite{Bat94, Pas08, Pas10, Sus14, GH17}. We believe that the combinatorial description developed in this article can be useful for describing the deformation theory of spherical varieties as in \cite{AB05, AB06}. We also refer to \cite{Arz97b, Arz02a, Arz02b} for other results on the geometry of normal $G$-varieties with spherical orbits.
\\

{\bf Content of the article.} Let us give a brief summary of the contents of each section. In Section \ref{sec-prel}, we introduce the notation for reductive group actions by following the viewpoint of the Luna-Vust theory. Section \ref{s-combi} establishes Theorem B and Theorem C in the case where the $G$-equivariant birational type is trivial. We also introduce combinatorial tools such as the concept of colored polyhedral divisor. In Section \ref{sec-3class}, we investigate the equivariant birational classification of a normal $G$-variety with spherical orbits and in particular we prove Theorem C. Finally, in Sections \ref{sec-Weil} and \ref{s-canclass} we provide some applications. We determine the divisor class group of a normal $G$-variety with spherical orbits in \ref{t-divisor1} and obtain information on the canonical class in the last section.        
\\

{\em Acknowledgments.}
We are extremely grateful to the referee for her (or his) careful remarks that improved a lot the presentation in a previous version. We warmly thank Ivan Arzhantsev and Roman Avdeev for useful discussions, and David Bradley-Williams for many corrections. We also warmly thank Hendrik S\"uss for e-mail exchanges and for allowing us to use his pictures of divisorial fans.
This research was supported by the Max Planck Institute for Mathematics of Bonn and
by the Heinrich Heine University of D\"usseldorf. This research was also conducted in the framework of the research training group
\emph{GRK 2240: Algebro-geometric Methods in Algebra, Arithmetic and Topology},
which is funded by the DFG.
\\

{\bf Notation.} By a \emph{variety} (resp. a \emph{linear algebraic group}) we mean an integral separated scheme (resp. an affine group scheme) of finite type over $k$. All subgroups of a linear algebraic group are assumed to be closed. Let $X$ be an integral scheme of finite type over $k$. We denote by $k[X] = \Gamma(X, \mathcal{O}_{X})$ the algebra of regular global functions. Moreover, $k(X)$ is the
field of rational functions on $X$ which is by definition the residue field of the generic point of $X$. 

The letter $G$ stands for a connected reductive linear algebraic group. Changing $G$ by a finite covering, we may restrict to the case where $G$ is \emph{simply-connected}. This latter condition means that $G$ is a direct product $C\times G^{ss}$, where $C$ is an algebraic torus
and $G^{ss}$ is a simply-connected semi-simple linear algebraic group. We will consider a maximal torus $T$ in a Borel subgroup $B$ of $G$. 
We denote by $U$ the unipotent radical of $B$ so that $B = TU$ and by $\Delta$ the set of simple roots of $G$ with respect to the pair $(B,T)$.

We will use \cite{Tim11} as our reference for the geometry
of homogeneous spaces. 

\begin{remark}
The restriction to the case where the base field is characteristic $0$  is mainly due to the assumptions in the result of Alexeev and Brion (see \cite[Theorem 3.1]{AB05}) used in our paper. It would be interesting to develop the theory of $G$-varieties with spherical orbits over a field of positive characteristic.
\end{remark}
\section{Preliminaries}\label{sec-prel}
In this section, we briefly recall some basic notions on reductive group actions that we will use in this paper.  In Section \ref{sec-valsec}, we define the notions of $G$-valuations, colors and scheme of geometric localities of a given $G$-variety $\XX$. We also mention the definitions of a $B$-chart and of an embedding of a homogeneous space. Section \ref{s:spherical-sub} is devoted to the classification of \emph{spherical subgroups} of $G$, that is closed subgroups $H\subseteq G$ such that $G/H$ is a spherical homogeneous $G$-space.

\subsection{Valuations and colors}\label{sec-valsec} 
Many constructions that we will encounter deal with valuations and
colors. From the viewpoint of the Luna-Vust theory \cite[Chapter 12]{Tim11}, they constitute the basic material 
for describing normal algebraic varieties with a reductive group action.
\\

Let $X$ be an integral scheme of finite type over $k$. 
A \emph{discrete valuation} of $k(X)$ is a group homomorphism 
$$v : (k(X)^{\star}, \times)\rightarrow (\QQ, +)$$
with kernel containing the subgroup $k^{\star}$ and image $a\ZZ$
for some $a\in\QQ$, such that $v(f_{1} + f_{2})$  is greater or equal to  $$\min\{v(f_{1}), v(f_{2})\}\text{ for all }f_{1}, f_{2}\in k(X)^{\star}\text{ satisfying }f_{1} + f_{2}\neq 0.$$ For a possibly non-closed point $\zeta\in X$ (resp. a valuation $v$ on $k(X)$), the associated local ring
is denoted by $\mathcal{O}_{\zeta, X}$ (resp. $\mathcal{O}_{v}$). 
The valuation $v$
of $k(X)$ has a \emph{center} in $X$ if there exists a schematic point $\xi\in X$ such that the valuation ring $\mathcal{O}_{v}$ dominates $\mathcal{O}_{\xi, X}$, i.e., 
$\mathcal{O}_{\xi, X}\subseteq \mathcal{O}_{v}$ and $\mathfrak{m}_{\xi}\subseteq \mathfrak{m}_{v}$ for the corresponding maximal ideals.
\\

Assume that $X$ is a normal variety.
Then every prime divisor $D$ on the variety $X$ determines
a discrete valuation $v_{D}$ on $k(X)$ (called the \emph{vanishing
order} along $D$) so that $\mathcal{O}_{v_{D}} = \mathcal{O}_{\xi, X}$, where $\xi$ is the generic point of $D$. The \emph{geometric valuations}
are those of the forms $\alpha \cdot v_{D}$ for any possible choice 
of normal varieties $X'$ such that $k(X) = k(X')$, prime divisors $D\subseteq X'$
and scalars $\alpha\in\QQ_{>0}$. Let $L\subseteq G$ be a subgroup and suppose that $G$ acts on $X$.
The valuation $v$ on the field $k(X)$ is said to be \emph{$L$-invariant} (or simply called an \emph{$L$-valuation}) if it is geometric and  if further the equalities $v(g\cdot f) = v(f)$ hold for all $f\in k(X)^{\star}$  and $g\in L$.
It is moreover called \emph{central} if its restriction to the subfield of $B$-invariants $k(X)^{B}$ is trivial.

Let us introduce some special subvarieties of the $G$-variety $X$ which will play an important role later on.
An \emph{$L$-cycle} (or \emph{$L$-germ} if we rather look at the corresponding generic point) of $X$ is an $L$-stable irreducible closed subvariety of $X$, an \emph{$L$-divisor} is an $L$-cycle of codimension one, and a \emph{color} is a $B$-divisor which is not
$G$-stable. In particular, every $L$-divisor on $X$ defines an $L$-valuation on $k(X)$.
\\

We now introduce the scheme of geometric localities \cite[Section 1]{LV83}. Let $\XX$ be a variety. A (geometric) \emph{locality}
of $k(\XX)$ is a local ring associated with a prime ideal of a finitely generated normal subalgebra $A\subseteq k(\XX)$ having field of fractions $k(\XX)$.
The set of localities $\sc(\XX)$ of $k(\XX)$ is naturally endowed with a structure of scheme over $k$ where the possible affine schemes $\spec\, A$ are considered as open subsets. Note that the scheme $\sc(\XX)$ is in general not separated over $k$.
A ($G$-)\emph{model of $\XX$} is a normal ($G$-)variety equipped with a ($G$-equivariant) birational map $$X \dasharrow \XX,$$ yielding an identification of ($G$-)algebras over $k$ between $k(X)$
and $k(\XX)$. 
We denote by $\gsc(\XX)$ the \emph{$G$-scheme of
geometric localities of $\XX$} (also called the universal $G$-model).
As a set, $\gsc(\XX)$ consists of localities
$\mathcal{O}_{\xi, X}\subseteq k(\XX)$, where $\xi$ is a schematic point of a $G$-model $X$ of $\XX$. Note that the natural birational $G$-action on $\sc(\XX)$ is not regular in general. Actually, $\gsc(\XX)$ corresponds to the maximal open $G$-stable subset of $\sc(\XX)$ in which the $G$-action on $\sc(\XX)$ is regular. 
Summing up, any ($G$-)model of $\XX$
can be thought as a ($G$-stable) separated open subscheme of finite type over $k$ of $\sc_{(G)}(\XX)$, and vice-versa. Notice that the generic point of a color of a $G$-model of $\XX$ meets any other $G$-model (compare with \cite[Remark 13.3]{Tim11}). Hence it is reasonable to speak about \emph{colors} of $\gsc(\XX)$.
\\

Local information on the normal $G$-variety $X$ can be read off from their \emph{$B$-charts}, that is,
their affine $B$-stable dense open subsets. We will consider $B$-charts of the $G$-scheme $\sc_{G}(\XX)$ which will be by definition the $B$-charts of any $G$-model of $\XX$.
The $G$-variety $X$ will be called \emph{simple} if it possesses 
a $B$-chart intersecting every $G$-orbit. This terminology makes sense, since according to the Sumihiro theorem (see \cite[Theorem 1]{Sum74}, \cite[Theorem 1.3]{Kno91}), every normal $G$-variety is a finite open union of simple $G$-varieties. 
\\

The next proposition is a well-known fact that we will constantly use in the present paper.
\begin{proposition}\cite[Section 2, Proposition 1]{Arz97c}
Let $\mathcal{X}$ be a $G$-variety. Then the following are equivalent.
\begin{itemize}
\item[(i)]  For a general point $x\in \mathcal{X}$, the $G$-orbit $G\cdot x$ is spherical.
\item[(ii)] Any $G$-orbit of $\mathcal{X}$ is spherical.
\end{itemize}
\end{proposition}

We finish this section by recalling the notion of embeddings of homogeneous spaces.
\begin{definition}\label{def-embbbb}
Let $H\subseteq G$ be a closed subgroup. If the mention of $H$ is clear from the context, then an \emph{embedding} of the homogeneous $G$-space $\Omega = G/H$ is a pair $(X,x)$ with the following properties. The letter $X$ denotes a normal $G$-variety and $x$ is a point of $X$ such that the stabilizer $G_{x}$ at $x$ is $H$ and the $G$-orbit $G\cdot x$ is open. To simplify the notation, in the sequel, we will
not indicate the base point $x$ of $(X,x)$ and rather write $X$ for the embedding $(X,x)$.
\end{definition}
\subsection{Spherical subgroups}
\label{s:spherical-sub}
In \cite{Lun01} Luna attached to any spherical subgroup of $G$ a discrete invariant depending on the root system of $G$: the prominent combinatorial
object occurring is a homogeneous spherical datum. 
It was shown that these objects classify 
the conjugacy classes of spherical subgroups of $G$ \cite{Los09, BP16}. The aim
of this subsection is to give a brief overview of this description. 
We recall that a subset $C$ of a finite dimensional $\QQ$-vector space $E$ is a \emph{polyhedral cone}
if there exist vectors $v_{1}, \ldots , v_{d}\in E$ such that $C = \QQ_{\geq 0}v_{1}+\ldots +\QQ_{\geq 0}v_{d}.$  
\\

Let $\Omega =G/H$ be a spherical homogeneous $G$-space. We define the lattice $M$
as the set of $B$-eigencharacters of the $B$-algebra $k(\Omega)$. Denote by
$N = \Hom(M,\ZZ)$ the dual lattice and by $M_{\QQ} = \QQ\otimes_\ZZ M$,
$N_{\QQ} = \QQ\otimes_\ZZ N$ the associated dual vector spaces. The set of $B$-divisors of $\Omega$ consists of the irreducible components of the complement of the open $B$-orbit in $\Omega$. It forms the set of colors $\F_{\Omega}$ of $\Omega$.  

Any valuation $v$ on $k(\Omega)$ and any $B$-eigenfunction $f\in k(\Omega)$ gives a pairing  
$\langle \varrho(v), \chi_f\rangle = v(f),$
where $\chi_f$ is the $B$-weight associated with $f$. This expression is well-defined since $f$ is uniquely determined by $\chi_f$ up to the multiplication
of a nonzero constant. 

It is known by \cite[Proposition 7.4]{LV83} that the map $\varrho$ is injective on the set of $G$-valuations $\Vc$  of $k(\Omega)$. We will again denote by $\Vc$
the image $\varrho(\Vc)$. The subset $\Vc$ is a full dimensional cosimplicial polyhedral cone in $N_\QQ$ and admits therefore
a presentation
$$ \Vc = \bigcap_{\gamma\in \Sigma}\{v\in N_{\QQ}\, |\,\langle \gamma, v\rangle\leq 0\},$$
where $\Sigma$ is a finite set of linear independent primitive lattice vectors of $M$ (compare with \cite{BP87}). 
The set $\Sigma$ is called the \emph{set of spherical roots} of $\Omega$.
By \cite[Theorem 1]{Los09}, the triple $(M, \Sigma, \F_{\Omega})$ determines uniquely the spherical homogeneous space $\Omega$ up to a $G$-isomorphism.  
\\

Let $\alpha\in \Delta$ be a simple root of $(B,T)$ and let $\F_{\Omega}(\alpha)$ be the set of colors $D\subseteq G/H$ such that the corresponding minimal parabolic subgroup $P_{\alpha}$
moves $D$, i.e., $P_{\alpha}\cdot D\neq D$. The parabolic subgroup associated with the subset
$$\Delta^{p} = \{ \alpha\in \Delta\, |\, \F_{\Omega}(\alpha) = \emptyset\}$$
is the subgroup of $G$ preserving the open $B$-orbit of $\Omega$. 
Moreover, denoting by $D_1,\ldots, D_s$ the distinct colors such that
$$\{\alpha\in \Delta\, |\, P_{\alpha}\cdot D_{i}\neq D_{i}\}\cap \Sigma \neq \emptyset$$
for any $i$, we define the family $\aa$ as $(D_i)_{1\leq i\leq s}$.  
\\

Summarizing, to every spherical homogeneous $G$-space $\Omega= G/H$ we may attach 
$$\ss_\Omega = (\Delta^p_{\Omega}, \Sigma_{\Omega}, \aa_{\Omega}, M_{\Omega}) = (\Delta^p, \Sigma, \aa, M).$$ 
It was shown that any datum $\ss_\Omega$ satisfies the combinatorial conditions  of a \emph{homogeneous spherical datum}; we refer to \cite[\S 2]{Lun01} for the list of axioms. Note that from $\ss_{\Omega}$ one can recover the set of colors of $\Omega$ (see \cite[\S 2.3]{Lun01}).
In addition, we have the following important result of classification,
see \cite{Lun01, Los09, BP16}. 
\begin{theorem}
The correspondence 
$$\Omega\mapsto \ss_\Omega  = (\Delta^p_{\Omega}, \Sigma_{\Omega}, \aa_{\Omega}, M_{\Omega})$$
is a well-defined map
from the class of spherical homogeneous $G$-spaces to the set of homogeneous 
spherical data. It induces an injective map on the
set of $G$-isomorphism classes of spherical homogeneous $G$-spaces. Every homogeneous 
spherical datum of $G$ is geometrically realizable by a spherical homogeneous $G$-space via the map $\Omega\mapsto \ss_\Omega$. 
\end{theorem} 
The next two examples deal with reductive spherical subgroups of $\SL_{2}$.
\begin{example}\label{e-sl2}
We consider the natural diagonal action of $\SL_{2}$ on $\P^{1}\times \P^{1}$.
Let $T$ be the subgroup of diagonal matrices and let $B$ be the subgroup of upper triangular matrices. Then the spherical homogeneous space $\SL_{2}/T$ is identified with 
$\P^{1}\times \P^{1}\setminus {\rm diag}(\P^{1})$. It contains the $B$-chart $$X_{0}:= \{([x_{0}:1], [y_{0}:1])\,|x_{0}\neq y_{0}\}.$$
We have $M = \ZZ\alpha$, where $\alpha$ is the character $ \left( \begin{array}{ccc}
a & 0  \\
0 & a^{-1}  \end{array} \right)\mapsto a$ and $k[\chi^{-\alpha}, \chi^{\alpha}]$ is the subalgebra of $k(\SL_{2}/T)$ generated by $k(\SL_{2}/T)^{(B)},$ where $\chi^{\alpha} = (x_{0} - y_{0})^{-1}$. 
Moreover, $\SL_{2}/T$ has two colors obtained via the intersection with the two subsets
$D_{+} = \P^{1}\times \{[0:1]\}$ and $D_{-} = \{[0:1]\}\times \P^{1}$. Note that $\langle \varrho(D_{\pm}), \alpha\rangle = 1$. Finally, the homogeneous 
spherical datum $\ss = (\Delta^p, \Sigma, \aa, M)$ of $\SL_{2}/T$ is given by
$\Delta^{p}= \emptyset$, $\Sigma = \{\alpha\}$, $\aa =\{D_{-}, D_{+}\}$, and $M= \ZZ\alpha$.
\end{example}
\begin{example}\label{expl-sl2n}
We regard the projective plane $\P^{2}$ as the projectivization of $S^{2}V$, where $V$ is a two-dimensional vector space over $k$ and $S^{2}V$ is the space of binary forms 
$S^{2}V = k\, v^{2}\oplus k\,vw\oplus k\,w^{2}.$ Let
$$\tau:\P^{1}\times \P^{1}\setminus {\rm diag}(\P^{1})\rightarrow \P^{2}\setminus E$$
be the map 
whose restriction to $X_{0}\subseteq \P^{1}\times \P^{1}\setminus {\rm diag}(\P^{1})$ (see the notation of Example \ref{e-sl2}) is $$([x_{0}:1], [y_{0}:1])\mapsto [(x_{0}v + w)(y_{0}v + w)].$$ Then this map
is identified with the natural projection $\SL_{2}/T\rightarrow \SL_{2}/H$, where
$H$ is the normalizer of $T$ in $\SL_{2}$. The subset $E$ equal to $\{[(x_{0}v + y_{0}w)^{2}] \,|\, [x_{0}: y_{0}]\in\P^{1}\}$ is the space of degenerate binary forms. The morphism $\tau$ is a covering involution which sends the union $D_{-}\cup D_{+}$ of the colors of $\SL_{2}/T$ onto the unique color $D$ of $\SL_{2}/H$. Finally, the homogeneous 
spherical datum $\ss = (\Delta^p, \Sigma, \aa, M)$ of $\SL_{2}/H$ is given by
$\Delta^{p}= \emptyset$, $\Sigma = \{2\alpha\}$, $\aa =\emptyset$, and $M= 2\ZZ\alpha$.
\end{example}

\begin{example}{\em Horospherical homogeneous spaces.}
A closed subgroup $H\subseteq G$ is said to be \emph{horospherical}
if $H$ contains a maximal unipotent subgroup of $G$. Then in this case,
the normalizer $P= N_{G}(H)$ is a parabolic subgroup corresponding to a set of
simple roots $\Delta^{p}$ and $P/H$ is an algebraic torus $\TT$ with character lattice $M$. Note that $G/H$ is spherical and is equal to the parabolic induction $G\times^{P}\TT$. Hence the homogeneous spherical datum $\ss$ of $H$ is $(\Delta^{p},\emptyset, \emptyset, M)$ (see \cite[Proposition 2.4]{Pas08}).  
\end{example}

\section{Combinatorics}\label{s-combi}
Let $\YY$ be an arbitrary algebraic variety and denote by $\Omega = G/H$  a
spherical homogeneous $G$-space.
We consider the $G$-variety $\XX:=\YY\times\Omega$, where the action is defined by letting $G$ act trivially on the variety $\YY$ and by left translations on the homogeneous space $\Omega$. We will denote by
$\ss = (\Delta^p, \Sigma, \aa, M)$
the homogeneous spherical datum
corresponding to $\Omega$. We recall that $N$ is the dual lattice of $M$.
In the current section, we classify normal $G$-varieties for the case of $G$-models of $\XX$.

Our approach splits into several sections that we now summarize. In Section \ref{s-chi},
we introduce the notion of colored polyhedral divisors inspired by the works of Altmann and Hausen for torus actions. This will latter be used to classify $B$-charts of any $G$-model 
of $\XX$. To have local information on the $G$-action on $\gsc(\XX)$,  we deal with the local structure theorem in Section \ref{s-sing}. While we make preparations in Section \ref{s-teclem}, in Section \ref{sec-localiz} we describe equivariant 
open immersions of $B$-charts using the language of colored polyhedral divisors. This allows us in Section \ref{sec-divfan} to construct via a gluing process
any $G$-model of $\XX$ in terms of a divisorial fan which is roughly speaking a fan of colored polyhedral divisors. Finally, the last section is devoted to providing
an explicit description of a simple $G$-model of $\XX$ by means of an embedding into a projective space.
\begin{remark}\label{remark-color}
There exists a natural bijection $\psi$ between the set of colors of $\Omega$
and $\gsc(\XX)$. The map $\psi$ sends a color
$D\subseteq \Omega$ to the closure of $S\times D$ in $\gsc(\XX)$. The injectivity of 
$\psi$ is clear and the surjectivity comes from \cite[Lemma 3]{FMSS95}.
In particular, the number of colors of $\gsc(\XX)$ is finite.
\end{remark}
\subsection{Colored polyhedral divisors}\label{s-p-divisor}
\label{s-chi}
Let us introduce the geometric environment where the $G$-valuations of $k(\XX)$ are represented \cite[Section 20.1]{Tim11}. More precisely, we want to construct an injective map from the set of $G$-valuations of $k(\XX)$ to a set $\hs$
(called later on hyperspace) depending only on $S$ and $N$.
\\

We call \emph{hyperspace} associated to the pair $(S, N)$ the quotient $\hs$ of $\geo(\YY)\times N_{\QQ}\times \QQ_{\geq 0}$ by the equivalence
relation $\equiv$ given by
\begin{equation*} \label{mod}
(\va,p,\ell)\equiv (\va', p', \ell') \text{ if and only if } \va = \va',\ p=p',\ \ell = \ell' \text{ or } p = p',\ \ell = \ell' = 0,
\end{equation*}
where $\geo(\YY)$ is the set of geometric valuations of $k(\YY)$ considered up to proportionality.
The equivalence class of $(\va,p,\ell)$ will be denoted by $[\va,p,\ell]$.
The reader may think that $\hs$ is the geometric object obtained as the union of the copies of upper spaces $\{s\}\times N_{\QQ}\times\QQ_{\geq 0}$ (where $s$ runs over $\geo(\YY)$)
in which the boundaries $\{s\} \times N_{\QQ}\times\{0\}$ are glued together into a common part. 
The element $[\va,p,0]$ does not depend on $\va$ and we will write it by the symbol $[\cdot,p,0]$. In particular, we have a natural inclusion $N_{\QQ}\subseteq \hs$
where $N_{\QQ}$ is identified with the image of the map
$$N_{\QQ}\rightarrow \hs,\,\,p \mapsto [\cdot,p,0].$$
Let us consider the short exact sequence of abelian groups   
$$0 \rightarrow k(\YY)^{\star}\rightarrow k(\XX)^{(B)}\rightarrow M\rightarrow 0,$$
where $k(\XX)^{(B)}$ is the multiplicative group of the rational $B$-eigenfunctions on 
$\XX$. The arrow $k(\XX)^{(B)}\rightarrow M$ is the map sending a $B$-eigenfunction to its
$B$-weight. Let $A_{M}\subseteq k(\XX)$ be the subalgebra generated by $k(\XX)^{(B)}$. Then
choosing a lifting $M\rightarrow k(\XX)^{(B)}\cap k(\Omega)$, $m\mapsto \chi^{m}$ we obtain the decomposition
$$A_{M} = \bigoplus_{m\in M}k(\YY)\otimes \chi^{m}.$$
Every $G$-valuation $v$ of $k(\XX)$ is uniquely determined by its values on $A_{M}$ (see \cite[Corollary 19.13]{Tim11})
and therefore \cite[Proposition 20.7]{Tim11} defines an element $$[\va,a,b]\in \hs \text { via }v(f\otimes\chi^{m}) = b\cdot \va(f) + \langle m, a \rangle,$$
where $f\in k(\YY)^{\star}$ and $m\in M$. We denote by $\Vh$ the set of $G$-valuations identified with a subset
of $\hs$. Note that if $S$ is $0$-dimensional, then $\hs = N_{\QQ}$ and  $\Vh$ is nothing but the valuation cone $\Vc$ of $\Omega$.
\\

The next proposition determines $\Vh$ in terms of the homogeneous spherical datum $\ss$.
This result seems to be well known to the experts. For the convenience of the reader we include here a short proof.
\begin{proposition}\label{prop-vc}
Let $\Vc$ be  the valuation cone of the spherical homogeneous space $\Omega$ (which is completely determined by the set of spherical roots $\Sigma$). Then we have the equality
$$\Vh = \{[\va,a,b]\in\hs\,|\, \va\in \geo(\YY), (a,b)\in\Vc\times\QQ_{\geq 0}\}.$$  
\end{proposition}
\begin{proof}
Let $v = [\va,a,b]\in\Vh$. The restriction of $v$ to the subfield $k(\Omega)\subseteq k(\XX)$
is a $G$-valuation (see \cite[Corollary 16.9]{Tim11}). This implies that $a\in \Vc$. Conversely, let us show that
$$\{[\va,a,b]\in\hs\,|\, \va\in \geo(\YY), (a,b)\in\Vc\times\QQ_{\geq 0}\}\subseteq \Vh.$$ 
Let us fix elements $a\in \Vc$ and $\va\in \geo(\YY)$.
We first consider an embedding $X_{a}$ of the homogeneous space $\Omega$. It has
the property that the complement of the open orbit is an orbit $D_{a}$ of codimension one with vanishing order equal to $a$. Such
embedding always exists (see \cite[3.3, 7.5, 8.10]{LV83}). Since $\va$ is a geometric valuation, there exist a model $\Gamma$ of $\YY$ and a prime divisor $Y\subseteq \Gamma$  such that $\va$ is the vanishing order of $Y$. By considering the vanishing orders of $\Gamma\times D_{a}$
and $Y\times X_{a}$ in the $G$-model $\Gamma\times X_{a}$ of $\XX$, we obtain that $[\va,a,0], [\va,0,1]\in\Vh$.
We conclude by using \cite[Theorem 20.3]{Tim11}.
\end{proof}
Let $\sigma\subseteq N_{\QQ}$ be a strictly convex polyhedral cone (that
is a polyhedral cone containing no line) and let $\Gamma$ be a model of $\YY$. A \emph{polytope} of $N_{\QQ}$ is the convex hull of a non-empty finite subset of $N_{\QQ}$. The concept of polyhedral divisor was invented by Altmann and Hausen in \cite{AH06}. We recollect this notion in the next paragraph.

\begin{definition}
A \emph{$\sigma$-polyhedral divisor} on $\Gamma$ is a formal sum
$$\D = \sum_{Y\subseteq \Gamma}\D_{Y}\cdot Y,$$ 
where $\D_{Y}$ is empty or a $\sigma$-polyhedron (i.e., $\D_{Y}\subseteq N_{\QQ}$ is a Minkowski sum of $\sigma$ and a polytope), $Y\subseteq \Gamma$ runs over the set of prime divisors of $\Gamma$, and $\D_{Y} = \sigma$ for all but
finitely many prime divisors $Y\subseteq \Gamma$. 
The complement in $\Gamma$ of the union of prime divisors $Y\subseteq \Gamma$  such that $\D_{Y} = \emptyset$ and the cone $\sigma$ are called respectively the \emph{locus} and the \emph{tail} of $\D$; we will denote them by $\loc(\D)$ and $\tail(\D)$ if the notation is not explicitly specified.
\end{definition}
Let $\D$ be a $\sigma$-polyhedral divisor on $\Gamma$.
Let us denote by $\sigma^\vee$ the dual cone defined
as the subset 
$$\sigma^\vee = \{m\in M_\QQ\,|\, \forall v\in \sigma,\,\langle m, v\rangle \geq 0\}
$$
which consists of the linear forms of $M_{\QQ}$ that are non-negative on $\sigma$.
In practice, it is convenient to see the polyhedral divisor $\D$ as a piecewise linear function on $\sigma^{\vee}$; we call \emph{evaluation} at the vector $m\in \sigma^\vee$ the $\QQ$-divisor
$$\D(m) = \sum_{Y\subseteq \loc(\D)}\min_{v\in\D_{Y}} \langle m, v\rangle\cdot Y.$$
Consider the sheaf $\mathcal{O}_{\loc(\D)}(\D(m))$ where on each open subset $V\subseteq \loc(\D)$ we have 
$$\mathcal{O}_{\loc(\D)}(\D(m))(V) = \{f\in k(S)^{\star}\,|\, (\div(f)+\D(m))_{|V}\geq 0\}\cup\{0\}.$$
Here $\div(f) = \sum_{Y\subseteq\loc(\D)}v_{Y}(f)\cdot Y$ is the principal divisor associated to $f$. The expression
$$ \mathcal{A} = \bigoplus_{m\in \sigma^\vee\cap M} \mathcal{O}_{\loc(\D)}(\D(m)))\otimes\chi^m$$
naturally defines a sheaf of $M$-graded $\mathcal{O}_{\loc(\D)}$-algebras where the multiplication on the homogeneous elements
is induced by the multiplication on $k(\YY)$. The algebra of global sections $A(\Gamma, \D):= \Gamma(\loc(\D), \mathcal{A})$ is the \emph{associated algebra of $\D$}. Note that $A(\Gamma, \D)$ is an $M$-graded subalgebra of $A_{M}$.
\\

The properness is a technical condition on the polyhedral divisor $\D$ which ensures that the algebra $A(\Gamma, \D)$ is of finite type over $k$ and that its field
of fractions is equal to that of $A_{M}$ \cite[Theorem 3.1]{AH06}. Note that this condition on $\D$ is needed for the finite generation condition even in the case where $\Gamma$
is an algebraic curve, see the counterexample of Knop in \cite[Remark 16.22]{Tim11}, \cite{Kno93a}. It is defined as follows.

\begin{definition}\label{def-p-divisors}
For a $\QQ$-divisor $D$ on an algebraic variety $Z$, the symbol ${\rm supp}(D)$ will denote the \emph{support} of $D$ that is 
the union of the prime divisors of $Z$ corresponding to nonzero coefficients of $D$.

The $\sigma$-polyhedral divisor $\D$ is said to be \emph{proper} (cf. \cite[Definition 2.7]{AH06}) if it satisfies the following additional properties.
\begin{itemize}
\item The locus $\loc(\D)$ is a \emph{semi-projective variety}, i.e., it is projective over an affine variety. 
\item For all $m\in\sigma^{\vee}$, the $\QQ$-divisor $\D(m)$ is a \emph{semi-ample} Cartier $\QQ$-divisor, i.e., a multiple of $\D(m)$ corresponds to  a basepoint-free line bundle.
\item For all $m\in M_{\QQ}$ in the relative interior of $\sigma^{\vee}$, the 
$\QQ$-divisor $\D(m)$ is \emph{big}, i.e., there exist a positive integer $d$
and a global section $f\in H^{0}(\loc(\D), \mathcal{O}_{\loc(\D)}(\D(dm)))$ such
that 
$$\loc(\D)\setminus {\rm supp}(\div(f) + \D(dm))$$
is affine.
\end{itemize}
\end{definition}
\begin{remark}
The bigness condition enunciated in Definition \ref{def-p-divisors} generalizes
the classical one in the projective case (see \cite[Lemma 2.60]{KM98}).  Recall that  a line bundle $\mathscr{L}$
on a normal projective algebraic variety $Z$ is big if for a sufficiently large positive integer $d$,
the rational mapping of the total system of divisors $Z\dashrightarrow \P(\Gamma(Z, \mathscr{L}^{\otimes d}))$ is birational on its image. 
\end{remark}

Note that any polyhedral divisor with Cartier evaluations and affine locus is proper. The following result determines which multigraded algebras are described by  proper polyhedral divisors. They geometrically correspond  to algebras of regular functions on normal affine varieties with an effective torus action.

\begin{theorem}\cite[Theorem 3.4]{AH06}\label{pp-divisor}
Let $\sigma\subseteq N_{\QQ}$ be a strictly convex polyhedral cone.
Let $A$ be a normal $M$-graded subalgebra of $A_{M}$ of finite type over $k$ with  field of fractions equal to that of $A_{M}$. For $m\in M$ denote by $A_{m}$ the $m$-th graded piece of $A$ and assume that the cone generated by
$\{m\in M\,|\,A_{m}\neq \{0\}\}$ is $\sigma^{\vee}$.
Then there exist an open semi-projective subvariety 
$\Gamma\subseteq \sc(\YY)$ and a proper $\sigma$-polyhedral divisor $\D$ on $\Gamma$ such that $A = A(\Gamma, \D)$.   
\end{theorem}

Let us give some examples of torus actions obtained from polyhedral divisors.
\begin{example} Assume that $N_{\QQ} = \QQ$. Consider the proper polyhedral divisor $\D = \sum_  {y\in \P^{1}}\D_{y}\cdot [y]$ over the projective line $\P^{1}$ with the property that 
$$\D_{y} = \QQ_{\geq 0}\,\,\text{ if }y \in \P^{1}\setminus\{ 0,\infty\} \,\,\text{ and }\,\,
\D_{0} = \D_{\infty} = \left\{\frac{1}{2}\right\} + \QQ_{\geq0}.$$
Here $0$ and $\infty$ are respectively the origin and the point at infinity with
respect to a local parameter $t$ generating the function field $k(\P^{1})$. For $m\in \ZZ_{\geq 0}$, let $A_{m}: = H^{0}(\P^{1}, \mathcal{O}_{\P^{1}}(\D(m)))$. 
Then a direct computation shows that $A_{0} = k$, $A_{1}\otimes \chi^{1} = k\otimes \chi^{1}$ and for $m\geq 2,$
$$A_{m}\otimes \chi^{m} = \left(k\oplus \bigoplus_{i=1}^{\lfloor m/2\rfloor}kt^{i}\oplus \bigoplus_{i=1}^{\lfloor m/2\rfloor}k\frac{1}{t^{i}}\right)\otimes \chi^{m}.$$
Therefore we get that 
$$A:= A(\P^{1}, \D) = \bigoplus_{m\geq 0}A_{m}\otimes \chi^{m}  = k\left[\chi^{1}, t\otimes \chi^{2}, \frac{1}{t}\otimes\chi^{2}\right].$$
The spectrum $\spec\, A$ is identified with the $\G_{m}$-surface $\mathbb{V}(x^{4}-yz)\subseteq \AA^{3}$ and the action is given by the formula 
$\lambda\cdot (x,y,z) = (\lambda x, \lambda^{2} y, \lambda^{2} z)$ for any  $\lambda\in \G_{m}.$
\end{example}
\begin{example}
We again assume that $N_{\QQ} = \QQ$. Let $n$ be a positive integer. Consider the polyhedral divisor $\D = \sum_  {y\in \AA^{1}}\D_{y}\cdot [y]$ over the affine line $\AA^{1}$
with the conditions 
$$\D_{y} = \{0\}\,\,\text{ if }y \in \AA^{1}\setminus\{ 0,1\},\,\,
\D_{0} = \left\{\frac{1}{n}\right\}\,\, \text{ and }\,\, \D_{1} = \left[0, \frac{1}{n}\right].$$
Then a system of homogeneous generators of the $\ZZ$-graded algebra $A:= A(\AA^{1},\D)$
is given by
$$x := t^{-1}\otimes \chi^{n},\,y:= t(t+1)\otimes \chi^{-n},\,\, z:= \chi^{1}.$$
Using these coordinates, the spectrum $\spec\, A$ is $\G_{m}$-isomorphic to the smooth $\G_{m}$-surface $$W_{n} = \{x^{2}y = x + z^{n}\}\subseteq \AA^{3}.$$ 
The action is given by $\lambda\cdot (x,y,z) = (\lambda^{n}x,\lambda^{-n}y, \lambda z)$ for any  $\lambda\in \G_{m}.$
\end{example}

Let us introduce various objects attached to the $\sigma$-polyhedral divisor $\D$.
The \emph{Cayley cones} of $\D$ are the cones $C_{Y}(\D)\subseteq N_\QQ\oplus \QQ$ generated by the union of $(\sigma,0)$ and $(\D_Y, 1)$, where $Y \subseteq \Gamma$ is a prime divisor. The \emph{hypercone} associated with $\D$ is the subset 
$$C(\D) = \{[v_{Y}, a, b]\in \hs\,|\, Y\subseteq\Gamma \text{ prime divisor},\, (a,b)\in C_Y(\D)\}.$$
We define in an obvious way the relative interior of $C(\D)$.
\\

Here is the definition of a colored polyhedral divisor.
\begin{definition}\label{d-poly}
Considering the set of colors $\F_{\Omega}$ of $\Omega$, a pair $(\sigma, \F)$, where $\F$ is a subset of $\F_{\Omega}$, is a \emph{colored cone}\footnote{To have more flexibility, we take a different viewpoint by modifying slightly the definition of colored cone in \cite[Section 3]{Kno91}. In our definition, we do not impose that the relative
interior of $\sigma$ intersects $\Vc$. The reason is that in \cite[Section 3]{Kno91} the author deals only with \emph{minimal} $B$-charts of spherical varieties in order to have a perfect dictionary between colored cones and simple spherical embeddings (see \cite[Theorem 3.1]{Kno91}). We refer to \cite[Section 15.1 and Remark 14.3]{Tim11} for more information.}  (cf. \cite[Section 3]{Kno91}) if $\varrho(\F)$ does not contain $0$ and $\sigma$ is a strictly convex polyhedral cone generated by the union of $\varrho(\F)$ and a finite subset of $\Vc$. Colored cones are combinatorial objects related to the classification of simple spherical varieties, see for instance
\cite[Section 15.1]{Tim11}. 

The pair $(\D,\F)$ is a \emph{colored $\sigma$-polyhedral divisor} on $\Gamma$ if the following hold.
\begin{itemize}
\item $(\sigma, \F)$ is a colored cone.
\item $\D$ is a proper $\sigma$-polyhedral divisor.
\item The hypercone $C(\D)$ is generated by the union of $C(\D)\cap \Vh$
and $\varrho(\F)$.  
\end{itemize}  
\end{definition}
We denote by $\CP(\Gamma, \ss)$ the set of colored polyhedral divisors with
respect to the normal semi-projective variety $\Gamma\subseteq \sc(\YY)$ and the homogeneous spherical datum $\ss$. Note that colored polyhedral divisors have been used in \cite{LT16, LT17, Lan17, LPR19} for studying the geometry of complexity-one horospherical varieties.
\begin{remark}
Recall that a \emph{vertex} of a polyhedron $Q$ in $N_{\QQ}$ is a $0$-dimensional face
of $Q$. In Definition \ref{d-poly}, one observes that the vertices of $\D_{Y}$
for any prime divisor $Y\subseteq \Gamma$ belong to $\Vc$. This is due to Proposition
\ref{prop-vc} and the fact that $\varrho(\F)$ is contained in $\sigma$ seen as a subset of the hyperspace $\hs$.
\end{remark}

Let us take subsets $\UU\subseteq\Vh$ and $\F\subseteq \F_{\Omega}$, where $\F_{\Omega}$ is seen as the set of colors of $\gsc(\XX)$ (see Remark \ref{remark-color}). With these two data, one can construct a $B$-stable subalgebra of $k(\XX)$. Recall that $\mathcal{O}_v$ denotes the local ring associated with the valuation $v$. We then let 
$$R(\UU, \F) = (k(\Gamma)\otimes_k k[\Omega_{0}])\cap \bigcap_{D\in\F}\mathcal{O}_{v_D}\cap \bigcap_{v\in \UU}\mathcal{O}_{v}\subseteq k(\XX),$$
where $\Omega_{0}$ is the open $B$-orbit of $\Omega$.
The next lemma follows from an adaptation of the results in \cite[Section 8]{LV83}. It gives conditions for the affine scheme $X_{0}= \spec\, R(\UU, \F)$ to be a $B$-chart of $\gsc(\XX)$. 
\begin{lemme}\label{l-chartt}\cite[Theorem 13.8]{Tim11} Denote by $E$ the set $\UU\sqcup \F$ where $\F$ is considered as a set of valuations of $k(\XX)$. The affine scheme $X_{0} = \spec\, R(\UU, \F)$ is a $B$-chart of $\gsc(\XX)$ if and only if the following conditions 
are satisfied.
\begin{itemize}
\item[(i)]  For any finite subset $E_{0}\subseteq E$, there exists  a homogeneous element $\xi\in A_{M}$ such that for all $v\in E$ and $w\in E_{0}$
we have $v(\xi) \geq 0$ and $w(\xi) > 0$. 
\item[(ii)] The subalgebra $R(\UU, \F)^{U} = k[R(\UU, \F)^{(B)}]$ is of finite type over $k$, where $U\subseteq G$ is the unipotent radical of $B$. 
\end{itemize}
Moreover, any $B$-chart of $\gsc(\XX)$ arises in this way.
\end{lemme}
Actually, Conditions $(i), (ii)$ in the preceding lemma imply that the normal algebra $R(\UU, \F)$ has field of fractions $k(\XX)$ and is of finite type over $k$, respectively, so that the affine scheme $X_{0}$ is an open subset of
$\sc(\XX)$. Following the argument of \cite[Section 8]{LV83}, it was shown \cite[Theorem 13.8]{Tim11} that $R(\UU, \F)$ is stable under the natural action  of the Lie algebra of $G$ (see \cite[Proposition 12.3]{Tim11}).
\\

We will use the next two technical lemmata from the Luna-Vust theory.
\begin{lemme}\label{l:tecval1}\cite[Lemma 19.12]{Tim11}
For any nonzero element $f\in k(\Gamma)\otimes_{k}k[\Omega_{0}]$ and  any $G$-valuation $v\in\Vh$, there exists a $B$-eigenfunction $\tilde{f}\in A_{M}$ such that
the following hold.
\begin{itemize}
\item[(i)] $v(\tilde{f}) = v(f)$ and $w(\tilde{f})\geq w(f)$ for all
$w\in \Vh$.
\item[(ii)] $v_{D}(\tilde{f}) \geq v_{D}(f)$ for all $D\in\F_{\Omega}$.
\end{itemize}
\end{lemme}

\begin{lemme}\label{l:tecval2}\cite[Theorem 14.2]{Tim11}
Let $X_{0}$ be a $B$-chart of $\gsc(\XX)$ described by a pair $(\UU, \F)$
as in Lemma \ref{l-chartt}. Let $D\in\F$ be a color, let $O\subseteq \gsc(\XX)$ be  a $G$-cycle 
intersecting $X_{0}$, and let $v$ be a $G$-valuation centered in the generic point of $O$. Then $O$ is contained in $D$ if and only if for any $B$-eigenfunction $f$ in $R(\UU, \F)$ such that $v(f) = 0$,
we have $v_{D}(f) = 0$.
\end{lemme}

The next result gives a combinatorial picture for the correspondence between
$G$-valuations and colors, and the $B$-charts of $\XX$ via the language of polyhedral divisors.
\begin{theorem}\label{t-ppcol}
Let $(\D,\F)\in \CP(\Gamma, \ss)$ be a colored polyhedral divisor on an open semi-projective subvariety $\Gamma\subseteq\sc(\YY)$. Then the affine scheme $X_{0}(\D,\F) = \spec\, R(C(\D)\cap \Vh ,\F)$ is a $B$-chart of $\gsc(\XX)$. Any $B$-chart of $\gsc(\XX)$ arises from a colored polyhedral divisor as above. Therefore
any simple $G$-model of $\XX$ is of the form 
$$X(\D,\F):= G\cdot X_0(\D,\F)\subseteq \gsc(\XX).$$
The algebra of $U$-invariants $k[X_{0}(\D,\F)]^{U}$ is identified with $A(\Gamma, \D)$. 
\end{theorem}
\begin{proof}[Proof of Theorem \ref{t-ppcol}]
Let us fix a colored polyhedral divisor $(\D,\F)\in \CP(\Gamma, \ss).$ Associated to it,
we let $E := C(\D)\cap \Vh \sqcup \F$, where both sets $C(\D)\cap \Vh$
and $\F$ are seen as sets of valuations of the function field $k(\XX)$.
We start by proving that $X_{0}(\D,\F)$ satisfies
Conditions $(i)$ and $(ii)$ of Lemma \ref{l-chartt}. We naturally have the following equivalences 
$$\xi = f\otimes \chi^{m} \in R(C(\D)\cap \Vh ,\F)^{U}
\Leftrightarrow \forall v\in C(\D)\cap \Vh, \forall D\in \F,\, v(\xi)\geq 0 \text{ and } v_{D}(\xi)\geq 0$$
$$\Leftrightarrow \forall Y\subseteq \Gamma, \forall a\in\D_{Y}, \forall p\in \sigma,\, v_{Y}(f) + \langle m , a \rangle \geq 0 \text{ and } \langle m , p \rangle \geq 0,$$
for any homogeneous element $\xi$ of $A_{M}$. Hence $R(C(\D)\cap \Vh ,\F)^{U} = A(\Gamma, \D)$ and by \cite[Theorem 3.1]{AH06}, Condition $(ii)$ is verified.
Let $E_{0}\subseteq E$ be a finite set. Denote by $E_{1}$ and $E_{2}$ the subsets of $E_{0}$ of central and non-central valuations, respectively.
The elements of $E_{1}$ can be represented as elements in $\sigma:= \tail(\D)$,
in the sense that
$$E_{1} \subseteq \{[\cdot, p, 0]\in \hs\,|\, p\in\sigma\}.$$
Thus, we may find $m\in M$ such that for any $v\in E_{1}$ we have
$v(1\otimes \chi^m)>0$. Moreover, let $f\in k(\YY)^{\star}$ such that $v(f\otimes \chi^{m})>0$ for all $v\in E_{2}$. As $\D$ is a proper polyhedral divisor, the field of fractions of $A(\Gamma, \D)$ is equal to that of $A_{M}$, and so there exist homogeneous elements $\xi, \xi'$ of $A(\Gamma, \D)$ such that $f\otimes \chi^{m} =  \xi/\xi'$. Hence $v(\xi)\geq 0$  for all $v\in E$ and 
$$w(\xi) \geq w(\xi)  - w(\xi')  =  w(f\otimes \chi^{m})>0$$
for all $w\in E_{0}$, yielding Condition $(i)$ of Lemma \ref{l-chartt}. This shows that $X_{0}(\D,\F)$ is a $B$-chart of $\gsc(\XX)$.

Conversely, let $X_{0}$ be a $B$-chart of $\gsc(\XX)$. Since $k[X_{0}]$
is a Krull ring, we may write
$$k[X_{0}] = \bigcap_{D\in D(\XX)\sqcup\F}\mathcal{O}_{v_{D}}\cap \bigcap_{v\in \UU}\mathcal{O}_{v},$$
where $D(\XX)$ stands for the set of prime divisors of $\XX$ that are not $B$-stable (see \cite[Section 1.4]{Tim97}, \cite[Section 13]{Tim11}). Note that this set does not depend on the choice of $\XX$ and
$$ k(\Gamma)\otimes_k k[\Omega_{0}] = \bigcap_{D\in D(\XX)}\mathcal{O}_{v_{D}}
\text{ so that } k[X_{0}] = R(\UU, \F).$$
Given such a $B$-chart $X_{0}$, the dependence of the sets $\UU$ and $\F$ with respect to $X_{0}$ can be explained as follows. The set $\UU$ corresponds to $G$-valuations
that are centered in the generic point of an irreducible closed subvariety of $X_{0}$.
The set $\F$ is the set of colors of $\gsc(\XX)$ that intersect $X_{0}$. 
Moreover, we assume that the pair $(\UU, \F)$ satisfies Conditions $(i)$ and $(ii)$
of Lemma \ref{l-chartt}. Let $C$ be the linear span of $(\UU, \F)$ in $\hs$.
It is defined as the subset 
$$C =\{[\va, a, b]\in \hs\,|\, b\cdot \va(f) + \langle m, a\rangle\geq 0 \text{ for all } f\otimes \chi^{m}\in R(\UU, \F)^{(B)}\}.$$  
Let us show that  $k[X_{0}] = R(C\cap \Vh, \F)$. Since $\UU$ is a set of $G$-valuations of $k(\XX)$, we have
$\UU\subseteq C\cap \Vh$ and therefore $$R(C\cap \Vh, \F)\subseteq k[X_{0}].$$ Let $\zeta\in k[X_{0}]$ and take $v\in C\cap \Vh$. By Lemma \ref{l:tecval1},
there exists $\xi\in A_{M}^{(B)}$ such that 
$$w(\xi)\geq w(\zeta)\geq 0  \text{ and }v_{D}(\xi)\geq v_{D}(\zeta)\geq 0  \text{ for all }w\in \UU, D\in\F$$
and $v(\zeta) = v(\xi)$. Remarking that these latter conditions imply that $$\xi \in R(C\cap \Vh, \F)^{U} = k[X_{0}]^{U}$$ we get $v(\zeta)\geq 0$. Hence $\zeta \in R(C\cap \Vh, \F)$,
yielding the desired equality  $k[X_{0}] = R(C\cap \Vh, \F)$.

In the sequel, we may assume that $\UU = C\cap \Vh$. Now using Theorem \ref{pp-divisor}, there exist an open semi-projective subvariety $\Gamma\subseteq \sc(\YY)$ and a proper polyhedral divisor $\D$ on $\Gamma$ such that $k[X_{0}]^{U} = A(\Gamma, \D)$. 

We claim that $k[X_{0}] = R(C(\D)\cap \Vh, \F)$.
Indeed, by definition of the set $C$, we have $$C(\D)\cap \Vh\subseteq C\cap \Vh\text{ and thus }k[X_{0}]\subseteq R(C(\D)\cap \Vh, \F).$$ Let us assume that there exists
$\zeta\in R(C(\D)\cap \Vh, \F)$ such that $\zeta\not\in k[X_{0}]$. A contradiction is expected. From this assumption there is a $G$-valuation $v\in \UU$ 
such that $v(\zeta)<0$. By Lemma \ref{l:tecval1}, one can find a $B$-eigenfunction $\xi\in A_{M}^{(B)}$ such that $$(1)\,\,v(\xi) = v(\zeta)<0,\text{ and }(2)\,\, w(\xi)\geq w(\zeta)\geq 0$$  for all $w\in C(\D)\cap \Vh\sqcup \F$.
So Condition $(2)$ implies that $\xi\in A(\Gamma, \D)$. Since $k[X_{0}]^{U} = A(\Gamma, \D)$ the function $\xi$ verifies 
$w(\xi)\geq 0$ for all $w\in C\cap \Vh$ which contradicts $(1)$. This shows the equality $k[X_{0}] = R(C(\D)\cap \Vh, \F)$. 

We need
to prove that the resulting pair $(\D, \F)$ is a colored polyhedral divisor.
First, we may assume that $\Gamma$ is smooth. Indeed, let $\psi: \Gamma'\rightarrow\Gamma$ be a projective desingularization and consider the pull back polyhedral divisor $\psi^{\star}(\D)$, defined as $\psi^{\star}(\D)(m) = \psi^{\star}(\D(m))$ for any $m\in \sigma^{\vee}\cap M$. Since $\psi$ is a projective fibration, we have the equality
$$H^{0}(\Gamma, \mathcal{O}_{\Gamma}(\D(m))) = H^{0}(\Gamma', \psi^{\star}\mathcal{O}_{\Gamma}(\D(m))) = H^{0}(\Gamma',\mathcal{O}_{\Gamma'}(\psi^{\star}\D(m)))$$ 
for any $m\in\sigma^{\vee}\cap M$ such that $\D(m)$ is an integral Cartier divisor.
Hence by normality, we have $A(\Gamma,\D) = A(\Gamma',\psi^{\star}\D).$
Note that  $\psi^{\star}(\D)$ is proper (see \cite[Example 8.4(i)]{AH06}). In the sequel, we assume $\Gamma = \Gamma'$.

In the next step, we deal with the polyhedral divisor $\D'$ over $\Gamma$ corresponding to the hypercone generated by $C(\D)\cap \Vh$ and $\varrho(\F)$. 
We wish to have that $\D = \D'$ (and therefore
having $C(\D)$ generated by $C(\D)\cap \Vh$ and $\varrho(\F)$). 
Let $Y_{1}\subseteq \loc(\D)$ be a (smooth) dense affine open subset. Then by construction
$$A(Y_{1}, \D_{|Y_{1}}) = R(C(\D_{|Y_{1}})\cap \Vh, \F)^{U} = A(Y_{1}, \D_{|Y_{1}}').$$
Therefore we may assume that $\Gamma$ is smooth and affine. 
Doing this reduction, the equality $A(\Gamma, \D) = A(\Gamma, \D')$ implies that  $\D = \D'$ (see \cite[Lemmata 6.4 (iii), 9.1]{AH06}).

It remains to check that the pair $(\sigma,\F)$ is a colored cone, i.e. that $\varrho(\F)$ does not contain $0$. But this latter is a consequence of Condition (i) of Lemma \ref{l-chartt}. This finishes the proof of the theorem.
\end{proof}
As consequence of the proof of Theorem \ref{t-ppcol}, we get the following known result which says that colors and $U$-invariants determine uniquely a $B$-chart.
\begin{corollary}\label{c-equal}
Let $(\D,\F), (\D',\F')\in\CP(\Gamma, \ss)$ be two colored polyhedral divisors. Then the equalities
$$k[X_{0}(\D,\F)]^{U} = k[X_{0}(\D',\F')]^{U}\text{ and  } \F = \F'$$ hold if and
only if $X_{0}(\D,\F) = X_{0}(\D', \F')$.
\end{corollary}
\begin{proof}
The 'if' part is a consequence of the proof of Theorem \ref{t-ppcol} and the 'only if'
part of \cite[Proposition 13.7 (1)]{Tim11}.
\end{proof}
The next proposition states that a $B$-chart of $\gsc(\XX)$ associated with a
colored polyhedral divisor $(\D,\F)$ does not change if we modify the locus 
of $\D$ by a birational projective morphism. 
\begin{proposition}
\label{p:pullbackpp}
Let $\Gamma, \Gamma'$ be two semi-projective models of $S$ with a projective birational morphism
$\psi: \Gamma'\rightarrow \Gamma.$ Consider a colored polyhedral divisor 
$(\D,\F)\in\CP(\Gamma,\ss)$ and denote by $\psi^{\star}(\D)$ the polyhedral divisor defined by the equality $\psi^{\star}(\D)(m) = \psi^{\star}(\D(m))$ for 
any $m\in\tail(\D)^{\vee}$. Then $(\psi^{\star}(\D),\F)\in\CP(\Gamma',\ss)$ and we have 
$$X_{0}(\D,\F) =  X_{0}(\psi^{\star}(\D),\F)\text{ in } \gsc(\XX).$$
\end{proposition}
\begin{proof} 
This follows from the argument of the proof of Theorem \ref{t-ppcol}.
\end{proof}

\subsection{The local structure theorem} 
\label{s-sing}
The local structure theorem is an important result which asserts that a normal variety 
with an action of a connected reductive group can be locally expressed 
 as the product of an affine space and an affine variety with an action of a Levi subgroup (see \cite{BLV86}, \cite[Section 4]{Tim11} for more information). In this section, we investigate this result for simple normal $G$-varieties with spherical orbits having trivial birational type. Using Theorem \ref{t-ppcol}, we will translate it into the language of colored polyhedral divisors, see Theorem \ref{p-loc}. 
As usual, the symbol $\ss = \ss_{\Omega}= (\Delta^p, \Sigma, \aa, M)$ denotes the homogeneous
spherical datum of the spherical homogeneous $G$-space $\Omega$.
\\

\label{s-zeta}
We start by recalling the notion of localization of homogeneous spherical data
(cf. \cite{Lun97}, \cite[Section 3.2]{Lun01}). We first consider a subset
 $\Delta_{a}\subseteq \Delta$ of the set of simple roots of $(B,T)$. Then the \emph{localization} of $\ss$ with respect to $\Delta_{a}$ is the
datum 
$$\ss_{a}  = (\Delta^p_{a}, \Sigma_{a}, \aa_{a}, M),$$
\begin{align*}
		\text{where } \begin{cases}
						\Sigma_{a}  = \Sigma\cap {\rm vect}_{\QQ}(\Delta_{a}), \\
						\Delta^p_{a} =  \Delta^p\cap \Delta_{a},\\
						\aa_{a} = \{ D\in \aa\,|\,\zeta(D)\cap \Delta_{a}\neq \emptyset\}. 
					\end{cases}
\end{align*}
The symbol $\zeta(D)$ denotes the set of roots $\alpha\in \Delta$ such that the associated 
minimal parabolic subgroup $P_{\alpha}$ moves $D$. It is a homogeneous spherical datum for 
the Levi subgroup of $G$ corresponding to $\Delta_{a}$. 
Let us fix a set of colors $\F$ of $\Omega$. 
The case which we will encounter is when
$$\Delta_{a} = \Delta_{\lc}= \Delta_{\lc, \F}:=\Delta\setminus \bigcup_{D \not\in \F}\zeta(D).$$
We denote by $\ss_{\lc} = (\Delta^p_{\lc}, \Sigma_{\lc}, \aa_{\lc}, M)$
the corresponding homogeneous spherical datum.
\\

Let us explain the meaning of this operation in terms of spherical homogeneous spaces. For the parabolic subgroup $P$ associated with $\Delta_{\lc}$ and a Levi decomposition $P = G_{\star} \ltimes P_u$, the datum $\ss_{\lc}$ corresponds to the spherical homogeneous $G_{\star}$-space $\Omega_{\lc}$ satisfying the following property (see for instance \cite[Proposition 3.2]{Gag15}). Considering a colored cone $(\sigma, \F)$ and the simple spherical embedding $W$ of $\Omega$ attached to it, we may define the $B$-chart
$$W_{0} = W\setminus\bigcup_{D\in \F_{\Omega}\setminus \F} \bar{D}.$$
By the local structure theorem (see \cite[\S 4.2]{Tim11}), we have a decomposition $P_u\times W_{\lc}\simeq W_0$, where
$W_{\lc}$ is a $G_{\lc}$-stable closed subvariety of $W_0$ which is spherical for the acting group $G_{\star}$. 
The open $G_{\star}$-orbit in $W_{\lc}$ is $G_{\star}$-isomorphic to $\Omega_{\lc}$. In the sequel, we will denote by the letter $\F_{\lc}$ the set of colors of $\Omega_{\lc}$. Note that the spherical $G_{\lc}$-variety $W_{\lc}$ is described by the colored cone $(\sigma,\F_{\lc})$.     
\\

In order to study the local structure for the simple $G$-variety $X = X(\D,\F)$
associated with the colored polyhedral divisor $(\D,\F)\in \CP(\Gamma, \ss),$ we need to introduce the intermediate affine $G_{\star}$-variety $X_{\lc}$. Its definition is stated in the next paragraph. We keep the same notation as before for the spherical homogeneous $G_{\star}$-space $\Omega_{\lc}$.
\\

\label{r-module}
We now consider the multiplicity-free rational $G_{\star}$-module
$k[\Omega_{\lc}] = \bigoplus_{\lambda\in\Lambda}V_\lambda.$
Here $\Lambda\subseteq 	M$ is a satured semigroup of dominant weights, and $V_\lambda$ is the simple $G_{\star}$-submodule associated with $\lambda$. 
For all $\lambda, \mu\in \Lambda$ the subset $V_{\lambda}\cdot V_{\mu}$ is a direct sum of simple submodules $V_{\nu}$. The possible differences $\lambda + \mu - \nu$
belong to the semigroup generated by $\Sigma_{\lc}$. 
In particular, the dual of the cone generated by the possible differences $\nu-\lambda - \mu$ (belonging to $-\Sigma_{\lc}$) is exactly the valuation cone $\Vc_{\lc}$ of $\Omega_{\lc}$ (see \cite[\S 1.2, Proposition]{Bri91}). More precisely,
let us define the partial order $\leq_{\Lambda}$ on $\Lambda$ by letting $\mu\leq_{\Lambda}\lambda$
if $\lambda - \mu$ is a non-negative integral linear combination of elements of $\Sigma_{\lc}$. Thus,
we have the $G_{\lc}$-module inclusion
$$V_{\lambda}\cdot V_{\mu}\subseteq \bigoplus_{\nu\leq_{\Lambda} \lambda +\mu}V_{\nu}.$$
\begin{lemme}
\label{l-loc}
Consider a colored polyhedral divisor $(\D, \F)\in \CP(\Gamma, \ss).$ Then the subset 
$$A_{\lc}(\Gamma, \D):=  \bigoplus_{\lambda\in \tail(\D)^\vee\cap M} H^0(\loc(\D), \mathcal{O}_{\loc(\D)}(\D(\lambda)))\otimes_{k} V_{\lambda}\subseteq k(\Gamma)\otimes_{k}k(\Omega_{\lc})$$
is a $G_{\star}$-stable subalgebra. Moreover, the $G_{\star}$-scheme $X_{\lc} = \spec\, A_{\lc}(\Gamma, \D)$ identifies with the $G_{\star}$-model of $\XX_{\lc} = \Gamma\times \Omega_{\lc}$ corresponding to the colored polyhedral divisor
$(\D, \F_{\lc})$. 
\end{lemme}
\begin{proof}
For the first claim, we only need to show that if $\lambda,\mu,\nu\in\Lambda$ satisfy
$\nu\leq_{\Lambda} \lambda +\mu$, then $\D(\nu)\geq \D(\lambda+\mu)$, i.e., $\D(\nu) - \D(\lambda+ \mu)$
is an effective $\QQ$-divisor. Let $V(\D_{Y})$ be the set of vertices of $\D_{Y}$. We recall that for such $\lambda,\mu,\nu$ we have $\nu -\lambda - \mu \in\Vc_{\lc}^{\vee},$
where $\Vc_{\lc}$ is the valuation cone of $\Omega_{\lc}$. Since $\Sigma_{\lc}\subseteq \Sigma$, we have $-\Sigma_{\lc}\subseteq -\Sigma$ and therefore by duality the inclusions $V(\D_{Y})\subseteq \Vc\subseteq \Vc_{\lc}$
for any prime divisor $Y\subseteq \Gamma$. 
Hence we obtain that
$$ \min_{v\in\D_{Y}}\langle \nu, v\rangle - \min_{v\in\D_{Y}}\langle \lambda + \mu,v\rangle \geq     \min_{v\in\D_{Y}}\langle \nu - \lambda - \mu , v\rangle \geq 0,$$
yielding the first claim. By properness of $\D$, the scheme $X_{\lc}$ is a $G_{\star}$-model of $\XX_{\lc}$ (see \cite[Theorem D5]{Tim11}). Finally, the colored polyhedral divisor $(\D,\F_{\star})$ describes
$X_{\lc}$ since it is a $B$-chart (see \cite[Corollary 13.10]{Tim11} and Theorem \ref{t-ppcol}).
\end{proof}
The next result determines the local structure of a simple $G$-model of $\XX$ in terms of its colored polyhedral divisor.
\begin{theorem}
\label{p-loc}
Let $P\subseteq G$ be the parabolic subgroup associated with the set of simple roots $\Delta_{\lc}$. 
The local structure for the simple $G$-variety $X = X(\D,\F)$ can be expressed as follows. Consider
 the $B$-chart $X_{0} = X_{0}(\D,\F)\subseteq X$ attached to the colored polyhedral divisor $(\D,\F)$. Then there exist a Levi decomposition $P = P_{u}\ltimes G_{\star}$
and a closed $G_{\star}$-stable subvariety $X_{\lc}\subseteq X_{0}$ such that the map
$$\pi_{\lc}:P\times^{G_{\lc}}X_{\lc} = P_{u}\times X_{\lc}\rightarrow X_{0}, \,\,\,(u,x)\mapsto u\cdot x$$ is a $P$-isomorphism.
The variety $X_{\lc}$ is the variety defined in Lemma \ref{l-loc} and it corresponds to the colored 
polyhedral divisor $(\D, \F_{\lc})$. 
\end{theorem}

\begin{proof}
Consider the parabolic subgroup $P_{1} = \{g\in G\,|\, g\cdot X_{0}\subseteq X_{0}\}$
that stabilizes the $B$-chart $X_{0}$.
By construction of the $B$-chart $X_{0}$ (see the comment after \cite[Corollary 13.9]{Tim11}), we have the equality 
$$X_{0} = X(\D,\F)\setminus \bigcup_{D\in\F_{\Omega}\setminus \F}\bar{D},$$
where we identify colors of $\Omega$ with colors of $X(\D,\F)$. Indeed, $X(\D,\F)$ has a $G$-stable dense open subset $G$-isomorphic to $\Gamma_{0}\times\Omega$, where 
$\Gamma_{0}\subseteq \Gamma$ is an open subset. To any color $D\in\F_{\Omega}$
the closure $\bar{D}$ of $\Gamma_{0}\times D$ in $X(\D,\F)$ defines a color and
each of them is obtained in this way. Hence $P_{1}$ coincides with the parabolic subgroup $P$ preserving $\F_{\Omega}\setminus \F$.

Consequently, by \cite[\S 5, Lemma 2]{Tim00}, there exists a closed $G_{\star}$-stable subvariety $Z\subseteq X_{0}$ such that the map
$$\pi_{\lc}: P\times^{G_{\lc}}Z = P_{u}\times Z\rightarrow X_{0},\,\, (u,x)\mapsto u\cdot x$$
is a $P$-isomorphism. Let us show that $Z$ is $G_{\star}$-isomorphic to $X_{\lc}$. 
Denoting by $U_{\lc}$ (resp. $B_{\lc}$) the maximal unipotent subgroup $G_{\star}\cap U$ (resp. the Borel subgroup $G_{\star}\cap B$), we obtain that
$$k[X_0]^{U} = A(\Gamma, \D) = k[Z]^{U_{\lc}}\text{ and } k(X_{0})^{B} = k(\Gamma)  = k(Z)^{B_{\lc}}.$$ Moreover, by identifying the complement of the union of the colors of $\F_{\Omega}\setminus \F$ in $\Omega$ with the product $P_{u}\times \Omega_{\lc}$
and using the map $\pi_{\lc}$, it follows that the $G_{\star}$-variety
$Z$ has a $G_{\star}$-stable dense open subset $G_{\star}$-isomorphic to $\Gamma_{0}\times \Omega_{\lc}$. We conclude that $Z$ is $G_{\star}$-isomorphic to $X_{\lc}$. This finishes the proof of the theorem.   
\end{proof}
\subsection{Some technical lemmata}\label{s-teclem}
In this 
section, we collect some technical results needed for our classification problem.
The following classical lemma is a consequence of \cite[Section 5, Theorem 3]{Sum74} and \cite[Theorems 12.11, 12.13]{Tim11}. It is an equivariant version of the valuative criterion of properness and separateness.
\begin{lemma}
\label{l:testval}
If $X$ is an integral scheme of finite type over $k$ with a $G$-action, then 
$X$ is separated over $k$ (resp. proper over $k$) if and only if any $G$-valuation $\nu$ of $k(X)$
 has at most one center (resp. exactly one center) in $X$.
\end{lemma}
Let us recall some notation originally introduced in \cite[Section 7]{AHS08}. 
Let $(\D,\F)\in \CP(\Gamma, \ss)$ be a colored polyhedral divisor, where $\Gamma$ is a smooth semi-projective variety, and consider $s$ a discrete valuation on the function field $k(\Gamma)$ with center a schematic point $\xi\in \Gamma$. Since $\Gamma$ is smooth, any Weil divisor is locally described by a hypersurface and so for any prime $Y\subseteq \Gamma$ we denote by $f_{Y}$ the local equation of $Y$ near the schematic point $\xi$.
The symbol $s(\D)$ will stand for the polyhedron
$$s(\D) = \sum_{Y\subseteq \Gamma}s(f_{Y})\cdot \D_{Y}\subseteq N_{\QQ}.$$
If $s$ is trivial, then we make the convention that $s(\D)$ is equal to the tail of $\D$. 
Let $m\in\tail(\D)^{\vee}$ such that $\D(m)$ is an integral Cartier divisor and let $f$ be 
a local equation of $\D(m)$ near $\xi$. The polyhedron $s(\D)$ is constructed in a such way that $\min_{v\in s(\D)}\langle m, v \rangle = s(f)$.
 
\begin{lemma}
\label{l:valcenter}
Let $(\D,\F)\in \CP(\Gamma, \ss)$ be a colored polyhedral divisor on a smooth 
semi-projective variety $\Gamma$ and let $\nu = [s, p, \ell]\in \hs$
be a $G$-valuation on $k(\XX)$. Denote by $X= X(\D,\F)$ the simple $G$-model
of $\XX$ associated with $(\D,\F)$. Then $\nu$ has a center in $X$ if and only if the valuation $s$ has center a schematic point of $\Gamma$ and 
	\begin{align*}
		 \begin{cases}
						p/\ell\in s(\D) & \text{if $\ell\neq 0$,} \\
						p\in \tail(\D) & \text{if $\ell=0$.}
					\end{cases}
	\end{align*}
\end{lemma}

\begin{proof}
Let $X_{0}  =  X_{0}(\D,\F)$.
We start by using a similar argument as in the proof of \cite[Lemma 7.7]{AHS08}.
Assume that $\nu$ has a center $\zeta\in X$. Since $\zeta$ is the generic point 
of a $G$-cycle $X_{1}$, we have $X_{0}\cap X_{1}\neq \emptyset$ and this implies that $k[X_{0}]\subseteq \mathcal{O}_{\nu}$. Hence restricting the valuation $\nu$ to $k(\Gamma_{0})$, where $ \Gamma_{0}:= \spec\, k[\Gamma]$, and considering the projective morphism $\Gamma\rightarrow \Gamma_{0}$, we conclude by the valuative criterion of properness that $s = \nu_{|k(\Gamma)}$ has a center in $\Gamma$. Moreover, the inclusion $A(\Gamma, \D)= k[X_{0}]^{U}\subseteq \mathcal{O}_{\nu}$ and \cite[Lemma 7.7]{AHS08}
imply that $p/\ell\in s(\D)$ if $\ell\neq 0$ and $p\in \tail(\D)$ otherwise. 

Let us show the converse. By loc. cit. we directly have that $A(\Gamma,\D)\subseteq \mathcal{O}_{\nu}$. Now using \cite[Lemma 19.12]{Tim11},
for any $f\in k[X_{0}]$ there exists a $B$-eigenfunction $\alpha\in k[X_{0}]^{(B)}\subseteq A(\Gamma,\D)$ such that $\nu(f) = \nu(\alpha)\geq 0$.
Hence $\nu$ has a center in $X_{0}$ and therefore in $X$. This completes the proof of the lemma.
\end{proof}
The next result is a direct consequence of Lemma \ref{l:valcenter}.
\begin{lemma}\label{l-prop-lc}
Let $(\D,\F)\in \CP(\Gamma, \ss)$ be a colored polyhedral divisor. Consider
the sheaf of $\mathcal{O}_{\loc(\D)}$-algebras
$$\mathcal{A}_{\lc} = \bigoplus_{\lambda\in\sigma^{\vee}\cap M}\mathcal{O}_{\loc(\D)}(\D(\lambda))\otimes V_{\lambda}.$$
Here $\Omega_{\lc}$ is the homogeneous space obtained from $\ss$ by localization with respect to $\F$ (see Section \ref{s-zeta}) and $k[\Omega_{\lc}] = \bigoplus_{\lambda\in\Lambda}V_{\lambda}$
is the decomposition in irreducible representations. Then the natural morphism
$$q:\spec_{\loc(\D)}\,\mathcal{A}_{\lc}\rightarrow \spec\,\Gamma(\loc(\D), \mathcal{A}_{\lc})$$
is proper and birational.
\end{lemma}
\begin{proof}
The fact that the map $q$ is birational comes from the fact that the polyhedral divisor $\D$ is proper. We note that 
$\Gamma(\loc(\D), \mathcal{A}_{\lc})$ is equal to the intersection of the $\mathcal{A}_{\star}(V)$'s, where $V$ runs over the open dense subsets of $\loc(\D)$.
So we conclude that $q$ is proper by applying \cite[Theorem 12.13]{Tim11}
and Lemma \ref{l:valcenter}.  
\end{proof}

We now describe the invariant cycles 
of simple $G$-models of $\XX$ via the local structure theorem (see Theorem \ref{p-loc}).
\begin{lemme}\label{l-L-orbits}
Let $(\D,\F)\in\CP(\Gamma, \ss)$ be a colored polyhedral divisor. Denote by
$X_{0} = X_{0}(\D,\F)$ the associated $B$-chart and let $X = G\cdot X_{0}(\D, \F)$.  Let $Z$  
be a $G$-cycle in $X$ and let $\nu$ be the $G$-valuation centered in its generic point. Then
the following assertions hold.
\begin{itemize}
\item[$(i)$] The intersection $Z\cap X_{0}$ is identified via the map $\pi_{\lc}$ of Theorem 
\ref{p-loc} with the product $P_{u}\times Z_{\lc}$, where $Z_{\lc}$ is the 
$G_{\star}$-cycle in $X_{\lc}$ with corresponding vanishing ideal 
$$I(Z_{\lc}) = \{f\in k[X_{\lc}]\setminus\{0\}\,|\,\nu(f)>0\}\cup\{0\}.$$

\item[$(ii)$] The correspondence $Z\mapsto Z_{\lc}$ defines an injective map from
the set of $G$-cycles in $X$ to the set of $G_{\star}$-cycles in $X_{\lc}$ and preserves the inclusion order.  
\end{itemize} 
\end{lemme}
\begin{proof} Self evident.
\end{proof}
\begin{remarque}
The map $Z\mapsto Z_{\lc}$ is not a bijection in general, see the comment in \cite[Remark 15.19]{Tim11}. 
\end{remarque}

\subsection{Localization of colored polyhedral divisors}\label{sec-localiz}
In this subsection, we study and classify certain open immersions between $B$-charts of $\gsc(\XX)$. As in the case of
normal varieties with a torus action (see \cite[Sections 3, 4]{AHS08}), we characterize these immersions in the language of (colored) polyhedral divisors. 
In the sequel, we will denote by  $(\D,\F)\in \CP(\Gamma, \ss)$  a colored polyhedral divisor with tail the polyhedral cone $\sigma\subseteq N_{\QQ}$. 
\begin{definition}\label{d-loca}
Let $m\in\sigma^{\vee}\cap M$ and let 
$f\in H^{0}(\loc(\D), \mathcal{O}_{\loc(\D)}(\D(m)))$ be a nonzero global section. The \emph{localization} of $(\D,\F)$ with  respect to the homogeneous element $f\otimes\chi^{m}$ is the colored polyhedral divisor $(\D,\F)_{f}$ (or $(\D,\F)^{m}_{f}$ if we want to specify the degree of $f$) defined as follows. First of all, for any $\sigma$-polyhedron $C\subseteq N_\QQ$ we denote
$$\face(C, m) = \{v\in C\,|\, \langle m, v\rangle \leq \langle m, v'\rangle\text{ for all }v'\in C\}\subseteq N_{\QQ};$$
it is a $\sigma\cap m^{\perp}$-polyhedron. One can also denote by
$\Gamma_{f}$ the subset
$$\Gamma\setminus \supp (\div(f) + \D(m)).$$
Then considering
$$\D_{f} = \sum_{Y\subseteq \Gamma_{f}}\face(\D_{Y}, m)\cdot Y\text{ and } \F^{m} = \{D\in \F\,|\, \varrho(D)\in m^{\perp}\},$$
the colored polyhedral divisor  $(\D,\F)_{f}$ is the pair $(\D_{f}, \F^{m})$.
Note that the properness of $\D_{f}$ is a consequence of \cite[Proposition 3.3]{AHS08}.
Moreover, the symbol $\D_{y}:= \sum_{Y\subseteq \Gamma,\,y\in Y}\D_{Y}\subseteq N_{\QQ}$ will stand for the \emph{fiber polyhedron} over the point $y\in \Gamma$.
\end{definition}
As expected, the localization of a colored polyhedral divisor geometrically translates into the usual localization of the corresponding $B$-eigenfunction.
\begin{lemma}
\label{l:localiz}
Let $X_{0} = X_{0}(\D,\F)$ be a $B$-chart associated to the colored polyhedral divisor $(\D,\F)$. Let $\xi = f\otimes \chi^{m}$ be a homogeneous element of $A(\loc(\D),\D)$. Then $(\D,\F)_{f}$ is a colored polyhedral divisor describing 
the $B$-chart $(X_{0})_{\xi}:= X_{0}\setminus \supp(\div(\xi))$. 
\end{lemma}
\begin{proof}
By \cite[Proposition 3.3]{AHS08} we know that $k[X_{0}]_{\xi}^{U} = k[X_{0}((\D,\F)_{f})]^{U}.$ Hence letting $f_{1}/\xi^{r}$ be in $k[X_{0}]_{\xi}$ with $f_{1}\in k[X_{0}]\setminus \{0\}$, we have $v(f_{1}/\xi^{r})\geq 0$
for any $G$-valuation $v$ in $C(\D_{f})$ (by using Lemma \ref{l:tecval1}). Since $\F^{m}\subseteq \F$, we also have $$v_{D}(f_{1}/\xi^{r}) = v_{D}(f_{1}) - r\langle m, \varrho(D)\rangle  = v_{D}(f_{1}) \geq 0$$
for any $D\in \F^{m}$. This implies that $k[X_{0}]_{\xi}\subseteq k[X_{0}((\D,\F)_{f})]$.

For the opposite inclusion, we remark that the $B$-chart $(X_{0})_{\xi}$ can be described by a colored polyhedral divisor $(\D_{f},\F')$ over $\Gamma_{f}$ and
it remains to establish the equality $\F'=\F^{m}$. By the preceding step, we have $\F'\subseteq \F$ and therefore
$$D\in \F'\iff \forall r\in \ZZ,\, v_{D}(\xi^r) \geq 0\iff \langle m, \varrho(D)\rangle = 0\iff D\in \F^{m},$$
finishing the proof of the lemma.
\end{proof} 
The following lemma shows that open immersions of $B$-charts can be expressed in terms 
of localizations of colored polyhedral divisors. 
\begin{lemme}
\label{l:reduce} 
Let $X_{0}', X_{0}$ be two $B$-charts of $\gsc(\XX)$. Denote by $(\D', \F'), (\D, \F)$ the colored polyhedral divisors
describing $X_{0}'$ and $X_{0}$ respectively and assume that $\F'\subseteq \F$. 
\begin{itemize}
\item[$\rm (i)$]  Then we have an inclusion $X_{0}'\subseteq X_{0}$ if and only if there exists a finite family of elements $\alpha_{1} =f_{1}\otimes \chi^{m_{1}}, \dots, \alpha_{r} = f_{r}\otimes \chi^{m_{r}}$ of $k[X_{0}]^{U}$ such that 
$$G\cdot X_{0}' = G\cdot ((X_{0})_{\alpha_{1}}\cup\ldots \cup (X_{0})_{\alpha_{r}}) \text{ and }$$
$$(X_{0})_{\alpha_{i}} = (X_{0}')_{\alpha_{i}}\text{ for all }i\in\{1, \ldots, r\}.$$ 

\item[$\rm (ii)$] Let $\Gamma$ and $\sigma\subseteq N_{\QQ}$ be respectively
the (smooth) locus and the tail of $\D$ and assume 
that $X_{0}'\subseteq X_{0}$. Denote by 
$X_{1}$ the relative spectrum over $\Gamma$ of the sheaf
$$\mathcal{A} = \bigoplus_{m\in\sigma^{\vee}\cap M}\mathcal{O}_{\Gamma}(\D(m)).$$
Consider the contraction map $r: X_{1}\rightarrow X_{0}/\!\!/U$ and 
the quotient morphism $\pi: X_{1}\rightarrow \Gamma$. Then the subvariety $$\Gamma_{1}':= \pi(r^{-1}(X_{0}'/\!\!/U))\subseteq \Gamma$$
is open and semi-projective. Moreover, let
$$\D^{1} =\sum_{Y\subseteq \Gamma_{1}'}\D_{Y}^{1}\cdot Y, \text{ where }
\D^{1}_{Y} = \bigcup _{Y\cap (\Gamma_1')_{f_{i}} \neq \emptyset}\face(\D_{Y}, m_{i}),$$
and write $\F^{1}$ for the set $\bigcup_{i= 1}^{r}\F^{m_{i}}$. Then the pair $(\D^{1},\F^{1})$ is a colored polyhedral divisor
describing $X_{0}'$.
\end{itemize}
\end{lemme}
\begin{proof} 
(i) Let us show the direct implication. Since $X_{0}'$ is an affine dense open subset of $X_{0}$, the complement $D= X_{0}\setminus X_{0}'$ is pure of codimension $1$. So the closure of the irreducible components of $D$ in $X := G\cdot X_{0}$ are described by $G$-valuations and colors. Let $\zeta$ be the generic point of a $G$-cycle of $X(\D', \F')$ or a color intersecting $X_{0}'$. Then using Lemma \ref{l:tecval1}, there exists a homogeneous function $a_{\zeta}\in k[X_{0}]^{U}$ such that $a_{\zeta}$ vanishes on $D$ and is nonzero at $\zeta$.
Let $X_{2}:= \bigcup_{\zeta} (X_{0})_{a_{\zeta}}$, where $(X_{0})_{a_{\zeta}}$ stands for the localization of $X_{0}$ with respect to $a_{\zeta}$. Therefore $X(\D',\F') = G\cdot X_{2}$ and since we only consider finitely many $(X_{0})_{a_{\zeta}}$ for defining $X_{2}$, we conclude. For the converse, our assumption gives $X(\D',\F')\subseteq X(\D, \F)$. As $\F'\subseteq \F$ it follows that 
$$X_{0}' = X(\D', \F')\setminus \bigcup_{D\in\F_{\Omega}\setminus \F'}\bar{D}\subseteq X(\D,\F)\setminus \bigcup_{D\in \F_{\Omega}\setminus \F}\bar{D} = X_{0},$$
yielding Assertion (i).

(ii) The fact that $\Gamma_{1}'$ is a semi-projective open subset of $\Gamma$
and that $(\D^{1}, \F^{1})$ is a colored polyhedral divisor is a consequence
of the proof of \cite[Proposition 4.3]{AHS08}. By the proof of Assertion (i), $X_{0}'$ has an open subset of the form 
$$\Omega_{1}:= (X_{0})_{\alpha_{1}}\cup\ldots \cup (X_{0})_{\alpha_{r}}$$
intersecting any of its color and $G$-cycle. From the description
of $k[X_{0}']$ in terms of valuation rings, we observe that $k[\Omega_{1}] = k[X_{0}']$. 
Moreover for an arbitrary $B$-eigenfunction $\xi$ of $k(\XX)$, we have 
$$\xi \in k[\Omega_{1}]^{U}\iff \xi \in \bigcap_{i=1}^{r}k[(X_{0})_{\alpha_{i}}]^{U}
\iff \xi \in \bigcap_{i=1}^{r} A(\Gamma_{f_{i}}', \D_{f_{i}})\iff \xi \in A(\Gamma_{1}', \D^{1}), $$
and $\F' =  \F_{1}$, whence the result.
\end{proof} 
With the same notation as in Lemma \ref{l:reduce}, we will write $\bigcup_{i\in I}(\D, \F)_{\alpha_{i}}$ for the colored polyhedral divisor $(\D^{1}, \F^{1})$. 
\begin{corollaire}
\label{c:tec-loc}
Let $X_{0}$ be a $B$-chart of $\gsc(\XX)$ built from a colored poyhedral divisor $(\D,\F)\in \CP(\Gamma, \ss)$ and consider two colored polyhedral divisors
$$\left(\D^{1}= \sum_{Y\subseteq \Gamma}\D_{Y}^{1}\cdot Y, \F^{1} \right)= \bigcup_{i\in I}(\D,\F)_{\alpha_{i}}$$ 
$$\text{ and }\left(\D^{2}= \sum_{Y\subseteq \Gamma}\D_{Y}^{2}\cdot Y, \F^{2} \right)= \bigcup_{j\in J}(\D,\F)_{\beta_{j}}$$
that describe $B$-equivariant open subsets 
$$X_{0}^{1} := \bigcup_{i\in I}(X_{0})_{\alpha_{i}}\text{ and } X_{0}^{2} :=\bigcup_{j\in J}(X_{0})_{\beta_{j}}$$ of $X_{0}$ as in Lemma \ref{l:reduce}.
Then the intersection can be computed as
$$\left(\sum_{Y\subseteq \Gamma}\D_{Y}^{1}\cap \D_{Y}^{2}\cdot Y, \F^{1}\cap \F^{2}\right) = \left(\bigcup_{(i,j)\in I\times J}\D_{f_{i}g_{j}},\bigcup_{(i,j)\in I\times J}\F^{w_{i}+y_{j}}\right)$$
and corresponds to the $B$-chart $X_{0}^{1}\cap X_{0}^{2}$, where 
$\alpha_{i} = f_{i}\otimes\chi^{w_{i}}$ and $\beta_{j} = g_{j}\otimes\chi^{y_{j}}$. 
\end{corollaire}
\begin{proof}
This immediately follows from the equivalences 
$$\xi\in k[X_{0}^{1}\cap X_{0}^{2}]^{U}\iff \xi \in \bigcap_{(i,j)\in I\times J}k[(X_{0})_{\alpha_{i}}\cap (X_{0})_{\beta_{j}}]^{U}\iff \xi \in \bigcap_{(i,j)\in I\times J}k[(X_{0})_{\alpha_{i}\beta_{j}}]^{U}$$
$$\iff \xi \in \bigcap_{(i,j)\in I\times J}A(\loc(\D_{f_{i}g_{j}}), \D_{f_{i}g_{j}})\iff \xi \in A\left(\loc\left(\bigcup_{(i,j)\in I\times J}\D_{f_{i}g_{j}}\right),\bigcup_{(i,j)\in I\times J}\D_{f_{i}g_{j}}   \right)$$
and the equality $\F^{1}\cap \F^{2}= \bigcup_{(i,j)\in I\times J}\F^{w_{i}+y_{j}}$
by remarking that the polyhedral divisor $\D_{f_{i} g_{j}}$ is obtained by the component-wise intersection between $\D_{f_{i}}$ and $\D_{g_{j}}$.
\end{proof}

The next two lemmata are preparations for our next main result, namely Theorem \ref{p:glue-pp}.
\begin{lemma}
\label{l:torusorbit12}
Let $(\D, \F)\in\CP(\Gamma, \ss)$ be a colored polyhedral divisor with smooth
locus $\Gamma$.
For any $G$-orbit $O$ of a simple model $X= X(\D,\F)$ of $\XX$, there exist a closed point $y\in\Gamma$, a $G$-valuation $v = [\mu, p, \ell]$,  where $\ell\neq 0$, $p/\ell\in \D_{y}$ and $y$ is the center of $\mu$, such that $v$ is centered in the generic point 
of $\bar{O}$. 
\end{lemma}
\begin{proof}
Denote by $X_{0}$ the $B$-chart associated with $(\D,\F)$ and let $P$ be the parabolic subgroup stabilizing $X_{0}$. 
Using Theorem \ref{p-loc}, there exists a $P$-isomorphism 
$X_{0}\simeq P_{u}\times X_{\lc}$. We refer to Lemma \ref{l-loc} for the definition of $X_{\lc}$.
Let $\hat{X}_{\lc} = \spec\,\mathcal{A}_{\lc}$ be the relative spectrum over $\Gamma$ of the sheaf  
$$\mathcal{A}_{\lc} = \bigoplus_{\lambda\in\sigma^{\vee}\cap M} \mathcal{O}_{\Gamma}(\D(\lambda))\otimes_{k}V_{\lambda},$$
where $\sigma$ is the tail of $\D$. 
Then as in Lemma \ref{l-prop-lc},
we have the natural proper birational $P$-equivariant morphism $q:\hat{X_{0}}\rightarrow  X_{0},$
where $\hat{X_{0}} = P_{u}\times \hat{X_{\lc}}$. By Lemma \ref{l-L-orbits}, there exists a $P$-orbit $\hat{O}$ inside a fiber $(\pi)^{-1}(y)$ such that $q(\hat{O})$ is dense in the closure of $X_{0}\cap O$, where $\pi: \hat{X_{0}}\rightarrow \Gamma$ is the quotient map. So without loss of generality, we may assume that $\Gamma$ is a smooth affine variety.

Now since $\Gamma$ is smooth and affine, the polyhedral divisor $\D^{\circ}$ over $\Gamma$ defined by the equality $C(\D^{\circ}) = C(\D)\cap \Vh$ is proper. Moreover,
by combining Lemma \ref{l:valcenter} and \cite[Theorem 12.13]{Tim11}, we get a natural $G$-equivariant projective morphism
$X(\D^{\circ}, \emptyset)\rightarrow X(\D, \F)$. Hence we may assume that $\D = \D^{\circ}$ and $\F = \emptyset$. In this new situation, $X_{\star}$ becomes $\spec\, A(\Gamma,\D)$ and we conclude by using the description of torus orbits in \cite[Section 7]{AH06}.
\end{proof}
\begin{lemma}
\label{p:glue-pp1}
Let $(\D,\F), (\D',\F')\in\CP(\Gamma, \ss)$ be two colored polyhedral divisors with tails $\sigma, \sigma'$ living in $N_{\QQ}$ and smooth loci $\Gamma_{1}, \Gamma_{1}'$, respectively. Assume that $\Gamma_{1}'\subseteq \Gamma_{1}$, $C(\D')\subseteq C(\D)$, and $\F'\subseteq\F$. Then
the induced $B$-equivariant dominant morphism $X_{0}(\D',\F')\rightarrow X_{0}(\D,\F)$ is an open immersion if and only if for any geometric valuation $\mu$ centered in a schematic point $\zeta$ of $\Gamma_{1}'$ there exist $m\in \sigma^{\vee}\cap M$ and a nonzero section $f\in H^{0}( \Gamma_{1},\mathcal{O}_{\Gamma_{1}}(\D(m)))$ such that $(\F')^m = \F^m$, 
$$\zeta\in (\Gamma_{1})_{f}\subseteq \Gamma_{1}',\, \mu(\D')\cap \Vc = \face(\mu(\D), m)\cap\Vc,\text{ and }
\face(\nu(\D), m)\cap \Vc = \face(\nu(\D'), m)\cap\Vc$$
for any  other valuation $\nu$ centered in $(\Gamma_{1})_{f}\,\, (\star).$ Here $\Vc$ denotes the valuation cone of $\Omega$.
\end{lemma}
\begin{proof} 
$\Rightarrow:$ Assume that we have a $B$-equivariant immersion $X_{0}(\D',\F')\rightarrow X_{0}(\D,\F).$ Let $\mu$ be a geometric valuation centered in a generic
point of $\Gamma_{1}'$. Consider a $G$-valuation $v =[\mu, p, \ell]$ such that $\ell\neq 0$ and $p/\ell$ belongs to the relative interior of $\mu(\D')\cap \Vc$. Then by Lemma \ref{l:valcenter}, $v$ is centered in the generic point of a $G$-cycle $Z\subseteq X(\D',\F')$, and any other choice as before of $(p,l)$ gives the same $G$-cycle. Using Lemma \ref{l:reduce} (i), there exists $m\in \sigma^{\vee}\cap M$ and a nonzero $f\in H^{0}( \Gamma_{1},\mathcal{O}_{\Gamma_{1}}(\D(m)))$ such that the localization $X_{0}(\D, \F)_{\alpha}$ intersects $Z$ and $X_{0}(\D, \F)_{\alpha} = X_{0}(\D', \F')_{\alpha}$,
where $\alpha = f\otimes \chi^{m}$. From Lemma \ref{l:valcenter} we get that
$\mu$ is centered in $(\Gamma_{1})_{f}$ and the inclusion
$\mu(\D')\cap \Vc\subseteq \face(\mu(\D), m)\cap\Vc$, the reverse inclusion is performed in the same way. Condition $(\star)$ follows from the equality $X_{0}(\D, \F)_{\alpha} = X_{0}(\D', \F')_{\alpha}$.

 $\Leftarrow:$ This is the combinatorial translation of the characterization of Lemma \ref{l:reduce} (i).
\end{proof}
The next theorem gives a combinatorial description of $B$-equivariant open immersions of $B$-charts. 
\begin{theorem}
\label{p:glue-pp}
Let $(\D,\F), (\D',\F')\in\CP(\Gamma, \ss)$ be two colored polyhedral divisors with tails $\sigma, \sigma'$ living in $N_{\QQ}$ and smooth loci $\Gamma_{1}, \Gamma_{1}'$, respectively. Assume that $\Gamma_{1}'\subseteq \Gamma_{1}$, $C(\D')\subseteq C(\D)$, and $\F'\subseteq\F$. Then
the induced $B$-equivariant dominant morphism $X_{0}(\D',\F')\rightarrow X_{0}(\D,\F)$ is an open immersion if and only if for any $y\in \Gamma_{1}'$ there exist $m\in \sigma^{\vee}\cap M$ and a nonzero section $f\in H^{0}( \Gamma_{1},\mathcal{O}_{\Gamma_{1}}(\D(m)))$ such that $(\F')^m = \F^m$, 
$$y\in (\Gamma_{1})_{f}\subseteq \Gamma_{1}',\, \D_{y}'\cap \Vc = \face(\D_{y}, m)\cap\Vc, \text{ and }
\face(\D_{z}, m)\cap \Vc = \face(\D_{z}', m)\cap\Vc$$
for any $z\in(\Gamma_{1})_{f}$.
\end{theorem}
\begin{proof}
According to Lemma \ref{l:torusorbit12}, we may restrict the characterization of Lemma \ref{p:glue-pp1}
to the case of fiber polyhedra.  
\end{proof}
\subsection{Colored divisorial fans}\label{sec-divfan}
We now describe the $G$-models of $\XX$, that is, all normal $G$-varieties that are $G$-birational to $\XX = \YY\times \Omega$ in terms of geometric and combinatorial objects which we will call colored divisorial fans. In the special case of torus actions on normal varieties, this class of objects restricts to
those of divisorial fans introduced in \cite[Definition 5.2]{AHS08} and encompasses the defining fans of toric varieties. 

Without loss of generality, we will assume that all polyhedral divisors are defined on a smooth semi-projective variety $\Gamma$ (where $\Gamma$ is a model of $S$) by taking a pull back by a projective resolution of singularities (compare with Proposition \ref{p:pullbackpp})
if necessary. We recall that $\ss$ denotes the homogeneous spherical datum of the spherical homogeneous $G$-space $\Omega$.
\\

Here we give the definition of a colored divisorial fan.
\begin{definition}
\label{d:divfan}
A \emph{colored divisorial fan} associated with the pair $(\Gamma, \ss)$ is a finite set 
$$\EE = \{(\D^i, \F^i)\in \CP(\Gamma, \ss)\,|\,i\in I\}$$
of colored polyhedral divisors satisfying the following properties.
\begin{itemize}
\item[(i)] The intersections $(\D^i\cap \D^j, \F^i\cap \F^j)$, for all $i,j\in I$,
belong to $\EE$,
where 
$$\D^i\cap \D^j := \sum_{Y\subseteq\Gamma}\D^i_Y\cap \D^j_Y\cdot Y.$$
\item[(ii)] For all $i,j\in I$,
the natural maps 
$$X_{0}(\D^{i},\F^{i})\leftarrow X_{0}(\D^{i}\cap \D^{j},\F^{i}\cap \F^{j})\rightarrow X_{0}(\D^{j},\F^{j})$$
are open immersions (see Theorem \ref{p:glue-pp} for a geometric and combinatorial description).
\item[(iii)] For any geometric valuation $\mu$ on the function field $k(\Gamma) = k(S)$ we have 
$$\mu(\D^{i})\cap \mu(\D^{j})\cap\Vc = \mu(\D^{i}\cap \D^{j})\cap\Vc$$ for all $i,j\in I$, where $\Vc$ is the valuation cone of the spherical homogeneous space $\Omega$ (see Section \ref{s-teclem} for the definition of $\mu(\D^{i})$).
\end{itemize}
\end{definition}
We define the \emph{locus} of $\EE$ as
$$\loc(\EE) := \bigcup_{(\D,\F)\in \EE}\loc(\D)\subseteq \Gamma.$$
 
\begin{example} \label{ex-sphere}
A spherical $G$-variety can be seen as a $G$-model (or an embedding) of its open $G$-orbit $\Omega$. Classically, such $G$-varieties are described by colored fans (see \cite[Section 3]{Kno91}). Let $(\sigma_{1}, \F_{1})$, $(\sigma_{2}, \F_{2})$ be two colored cones
of $\Omega$. We say that  $(\sigma_{1}, \F_{1})$ is \emph{essential} if its 
relative interior meets the valuation cone $\Vc$. Moreover, $(\sigma_{1}, \F_{1})$ is a \emph{face} of $(\sigma_{2}, \F_{2})$ if $\sigma_{1}$ is a face of $\sigma_{2}$ whose its relative interior intersects $\Vc$ and $\F_{1} = \varrho^{-1}(\sigma_{1})\cap \F_{2}$.
A \emph{colored fan} $\mathbb{F}$ is a finite set of essential colored cones of $\Omega$, stable under the face relation, and such that for any $v\in \Vc$ there exists at most one colored cone $(\sigma, \F)\in \mathbb{F}$ with $v$ in the relative interior of $\sigma$.  Each essential colored cone exactly corresponds to a simple $G$-model of $\Omega$ and the open immersions of $B$-charts are translated into their face relations. One may recover it from Theorem \ref{p:glue-pp}. Colored divisorial fans are vast generalizations of colored fans, where $S$ is no longer assumed to be $0$-dimensional.  
\end{example}
\begin{example}
\begin{figure}[t]
\includegraphics[width=8cm, height=1.5cm]{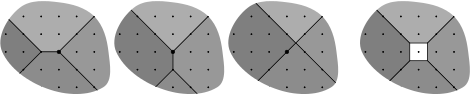}
\centering
\caption{}
\end{figure}
\emph{S\"uss pictures and torus actions.} Figure 1 illustrates a usual divisorial fan over
the projective line. The three first polyhedral subdivisions represent
the non-trivial polyhedral coefficients.
The fourth polyhedral subdivision consists of their Minkowski sums. Moreover, 
there are exactly four distinct polyhedral divisors with complete locus 
and maximal tail cone. According to the classification of S\"uss in \cite{Sus14} of Fano threefolds
admitting a faithful $2$-torus action, this divisorial fan defines
a smooth quadric threefold of $\P^{4}$ with Picard number $1$.
\end{example}
We now enunciate the main result of this section which classifies $G$-models of $\XX$. 
\begin{theorem}
\label{t:class1}
Let $\Gamma$ be a smooth projective model of the variety $S$ and
let $\ss$ be the homogeneous spherical datum of the spherical homogeneous $G$-space $\Omega$.
Denote by $\EE$ a colored divisorial fan on $(\Gamma, \ss)$. Then the open subscheme
$$X(\EE) := \bigcup_{(\D,\F)\in \EE} X(\D,\F)\subseteq \gsc(\XX)$$
is a $G$-model of $\XX =  S\times\Omega$ in which the open subsets $X(\D\cap \D',\F\cap \F')$
are identified with the intersections $$X(\D,\F)\cap X(\D',\F')\text{ for all }(\D,\F), (\D',\F')\in \EE.$$ Conversely, any $G$-model of $\XX$ arises in this way. 
\end{theorem}
\begin{proof}
Let $\EE = \{(\D^{i},\F^{i})\}_{i\in I}$ be a colored divisorial on $(\Gamma, \ss)$. We start by proving the first claim, namely that the 
open subscheme $X(\EE)$ in $\gsc(\XX)$ is obtained by gluing charts $X(\D,\F)$ for $(\D,\F)\in \EE$, where their intersections are given by the intersections of colored 
polyhedral divisors.
We first observe that we have the commutative diagram 
$$
\xymatrix{
X_{0}^{j}  \ar[r]
&X_{0}(\D^{i}, \F^{i}) 
&X_{0}^{\ell}\ar[l]\\
&X_{0}(\D^{i}\cap \D^{j}, \F^{i}\cap \F^{j})\cap X_{0}(\D^{i}\cap \D^{\ell}, \F^{i}\cap \F^{\ell}) \ar[u]\ar[lu]\ar[ru]}
$$
Here $X_{0}^{j}$ and $X_{0}^{\ell}$ are defined by the equalities
$$X_{0}^{e} = X_{0}(\D^{i}\cap \D^{e}, \F^{i}\cap \F^{e})$$
for $e = j,\ell$.
Note that the horizontal maps are the open 
immersions given by the definition of a divisorial fan.
Moreover, the other maps are clearly
open immersions.

The natural morphisms
$$X_{0}(\D^{i},\F^{i})\leftarrow X_{0}(\D^{i}\cap \D^{j},\F^{i}\cap \F^{j})\rightarrow X_{0}(\D^{j},\F^{j})$$
induce the open immersions
$$X(\D^{i},\F^{i})\leftarrow X(\D^{i}\cap \D^{j},\F^{i}\cap \F^{j})\rightarrow X(\D^{j},\F^{j})$$
which we denote respectively by $\eta_{ij}$ and $\eta_{ji}$. Put $X_{ij} = \eta_{ij}(X(\D^{i}\cap \D^{j},\F^{i}\cap \F^{j}))$. To see these maps define a gluing, we need to check that the cocycle conditions are satisfied, namely: 
$$\varphi_{ij}(X_{ij}\cap X_{i\ell}) = X_{ji}\cap X_{j\ell} \text{ and } \varphi_{i\ell} = \varphi_{j\ell}\circ \varphi_{ij},$$
where $\varphi_{ij}$ is the composition $\eta_{ji}\circ\eta_{ij}^{-1}$. 
Since the maps $\varphi_{ij}$ are inclusions of open subsets in the scheme  $\gsc(\XX)$,
it suffices to show that 
$$G\cdot (X_{0}(\D^{i}\cap \D^{j}, \F^{i}\cap \F^{j})\cap X_{0}(\D^{i}\cap \D^{\ell},\F^{i}\cap \F^{\ell}))$$ 
$$= X(\D^{i}\cap \D^{j}\cap \D^{\ell},\F^{i}\cap \F^{j}\cap \F^{\ell}),$$
which follows from Corollary \ref{c:tec-loc}.
Hence the open subset $X(\EE)$ is an integral scheme of finite type over $k$ in which the open subsets $X(\D\cap \D',\F\cap \F')$
are identified with the intersections $$X(\D,\F)\cap X(\D',\F')\text{ for all }(\D,\F), (\D',\F')\in \EE.$$ 

We now show that $X(\EE)$ is separated over $k$ by using Condition (iii) of Definition \ref{d:divfan} and Lemma \ref{l:testval}. We follow the argument of the proof of \cite[Proposition 7.5]{AHS08}.
Let $\nu = [s, p, \ell]$ be a $G$-valuation on $k(\XX)$ having centers the schematic points $\xi$ and $\xi'$ in $X(\EE)$. We may assume that $\ell\neq 0$.
Then $\xi, \xi'$ belong respectively to some dense open $G$-stable subsets $X(\D,\F)$ and $X(\D',\F')$, where $(\D,\F), (\D',\F')\in\EE$.
By Lemma \ref{l:valcenter}, the restriction $s= \nu_{|k(\Gamma)}$
has a unique center in $\Gamma$ and $$p\in s(\D)\cap s(\D')\cap\Vc = s(\D\cap \D')\cap\Vc.$$ This implies by Lemma \ref{l:valcenter} that $\nu$
has a center in $X(\D\cap \D',\F\cap\F')$. As $X(\D,\F)$ and $X(\D',\F')$ are separated over $k$, we obtain that $\xi = \xi'$. By Lemma \ref{l:testval}, 
we conclude that the subscheme $X(\EE)$ is a $G$-model of $\XX$. 

Conversely, let us consider $X$ a $G$-model of $\XX$. By the Sumihiro theorem 
(see \cite[Theorem 1 and Lemma 8]{Sum74}, \cite[Theorem 1.3]{Kno91}) there exists a $G$-stable open covering $(X_{i})_{i\in I}$ of $X$ by simple $G$-varieties, where $I$ is a finite set and $X_{i}\subseteq \gsc(\XX)$
for any $i\in I$. By Theorem \ref{t-ppcol}, each $X_{i}$ is described by a colored polyhedral divisor $(\D^{i}, \F^{i})\in \CP(\Gamma_{i}, \ss)$ on a normal semi-projective variety $\Gamma_{i}\subseteq \sc(\YY)$.  
We follow the argument of the proof of \cite[Theorem 5.6]{AHS08}. Let $\bar{\Gamma}_{i}$ be a projective compactification of $\Gamma_{i}$ such that 
the complement $\bar{\Gamma}_{i}\setminus \Gamma_{i}$ is the support of a semi-ample divisor. In this way, any colored polyhedral divisor $(\D^{i},\F^{i})$
is defined on $\bar{\Gamma}_{i}$ by adding empty coefficients if necessary. Moreover the inclusions $X_{i}\cap X_{j}\subseteq X_{i}$ induce birational maps 
between $\bar{\Gamma}_{i}$ and $\bar{\Gamma}_{j}$. By resolving the indeterminacies and using the Hironaka theorem, we obtain a smooth projective variety $\Gamma\subseteq \sc(\YY)$ which dominates all the $\bar{\Gamma}_{i}$'s and is compatible with the initial rational maps.
Then using Lemma \ref{l:reduce}, Corollary \ref{c:tec-loc} and Proposition \ref{p:pullbackpp}, we may choose the colored polyhedral divisors $(\D^{i},\F^{i})$'s such that their pull-back to $\Gamma$ forms a set $\EE$ satisfying Conditions (i), (ii) of Definition \ref{d:divfan}. 

It remains to show that $\EE$ verifies the Condition \ref{d:divfan} (iii).
Let us assume that this condition does not hold for $\EE$. Then there exist
$(\D,\F), (\D',\F')\in \EE$ and a geometric valuation $\mu$ on $k(S)$ such that
$$\mu(\D\cap \D')\cap \Vc\subsetneq\mu(\D)\cap\mu(\D')\cap\Vc.$$ Let $p\in \mu(\D\cap \D')\cap \Vc$ and assume that $p$ does not belong to $\mu(\D)\cap\mu(\D')\cap\Vc$.
Then by Lemma \ref{l:valcenter}, the $G$-valuation $\nu=[\nu, p, 1]$ has no center in $X(\D\cap\D',\F\cap\F')$ but has center $\xi, \xi'$ in $X(\D,\F)$ and $X(\D',\F')$, respectively. Since by the preceding steps we have 
$$X(\D,\F)\cap X(\D',\F') = X(\D\cap\D',\F\cap\F'),$$
we conclude that $\xi\neq \xi'$, which contradicts the separateness of $X$ and 
completes the proof of the theorem.
\end{proof}
Our next task is to characterize the completeness property among the $G$-models of $\XX$.
For this purpose, we introduce the appropriate notion for colored divisorial fans.
\begin{definition}
Let $\EE$ be a colored divisorial fan on $(\Gamma, \ss)$. Recall that $\Vc$ is the valuation cone of the spherical homogeneous $G$-space and the variety $\Gamma$ is always assume to be smooth and projective. We say $\EE$ is \emph{complete} if for any geometric valuation $s$ on the function field $k(\Gamma)$ the condition $\bigcup_{(\D,\F)\in\EE}s(\D)\cap\Vc = \Vc$ holds. In particular, this implies that 
$\loc(\EE) = \Gamma$.
\end{definition}
\begin{proposition}\label{p:completef}
Let $\EE$ be colored divisorial fan defining a $G$-model $X$ of $\XX$. Then the scheme $X$ is proper over $k$ if and only if $\EE$ is complete.
\end{proposition}
\begin{proof}
We will use the criterion of Lemma \ref{l:testval}. From this result, the fact that the completeness of $\EE$ implies that the $k$-scheme $X$ is proper is straightforward.
Assume that $X$ is complete. If
$$E:= \bigcup_{(\D,\F)\in\EE}s(\D)\cap\Vc \subsetneq \Vc,$$
then choosing a vector $p\in \Vc\setminus E$, the $G$-valuation $[s, p, 1]\in \Vh$ on $k(\XX)$
has a center in $X$. This gives a contradiction and completes the proof of the proposition.
\end{proof}
\subsection{Explicit construction}\label{sec-expli}
This subsection aims to
construct explicitly the simple $G$-variety associated with a colored polyhedral divisor $(\D, \F)$ via an embedding into a projective space. We will follow the idea of the proof of \cite[Theorem 3.1]{Kno91}.
\\

We start by taking homogeneous generators
$h_{i}= f_{i}\otimes \chi^{m_{i}}$  $(1\leq i\leq r)$
of the $M$-graded algebra $A(\Gamma,\D)$, where every $f_{i}$ is in $k(\Gamma)^{\star}$. The functions $\chi^{m_{i}}$ considered as $B$-eigenfunctions on $k(\Omega)$ have their poles contained in the subset
$$Z_{0} = \bigcup_{D\in \F_{\Omega}\setminus \F} D\subseteq \Omega,$$
where $\F_{\Omega}$ is the set of colors of $\Omega$. 
Hence using that $G$ is factorial (since simply-connected),
we may choose a function $\xi_{0}\in k[G]^{(B\times H)}$ with zero locus 
equal to $\pi^{-1}_{\Omega}(Z_{0})$ and such that $$\xi_{i}:= \xi_{0}\cdot h_{i}\in k(\Gamma)\otimes_{k}k[G]\text{ for }1\leq i\leq r,$$ where $\pi_{\Omega}:G\rightarrow \Omega$ is the natural projection. 

Let $V$ be the $G$-module  generated by $\xi_{0}, \xi_{1}, \ldots , \xi_{r}$
in $k(\Gamma)\otimes_{k}k[G]$. We finally obtain a natural $G$-equivariant rational map
$$\iota: \Gamma\times \Omega \dashrightarrow \P(V^{\vee}).$$ 
In the next theorem, we may assume that the colored polyhedral divisor $(\D,\F)$ gives rise to a $B$-chart $X_{0}$ in which every color $D\in \F$ contains a $G$-orbit of $G\cdot X_{0}$. From Lemmata \ref{l:tecval2} and \ref{l:torusorbit12}, this is equivalent to ask that there is one fiber polyhedron $\D_{y}$ (see Definition \ref{d-loca}) whose relative interior meets $\Vc$. 
\begin{theorem}
\label{t-explicit}
Let $X_{0} :=\bar{X}\cap\{\xi_{0}\neq 0\}$ where $\bar{X}$ is the closure of the image of the map $\iota$ and let $X= G\cdot X_{0}\subseteq \P(V^{\vee})$. The $G$-variety $X$ is $G$-isomorphic to $X(\D,\F)$ and the $B$-chart $X_{0}(\D,\F)$ is identified with $X_{0}$.
\end{theorem}
\begin{proof}
Let $Z$ be the affine cone arising from $\bar{X}\hookrightarrow \P(V^{\vee})$. The grading on
$k[Z]$ is given by $\bigoplus_{d\in \NN}R_{d},$
where $R_{d}$ is the subvector space
generated by monomials in elements of $V$ of degree $d$. Moreover,
we have 
$$k[X_{0}] = \{f/\xi_{0}^{d}\,|\, f\in R_{d}, \, d\in\NN\}.$$
From this we deduce that $k[X_{0}]^{U} = A(\Gamma, \D)$ (adapt the argument of \cite[Theorem 3.1]{Kno91}) and by \cite[\S 3, Lemma 1]{Tim00} that $X_{0}$ (and so $X$) is normal. Moreover, the rational map $\iota$ is induced by the inclusion $$k[X_{0}]\subseteq k(\Gamma)\otimes_{k}k(\Omega).$$ Since $\D$ is proper, the map $\iota$ is birational. Hence, $X$ is described by a colored polyhedral divisor $(\D,\F')$ and it remains to show that $\F=\F'$. Note that, by construction, the irreducible components of the hyperplane 
$\xi_{0}= 0$ are exactly the elements of $\F_{\Omega}\setminus \F$. Therefore by \cite[Proposition 13.7 (1)]{Tim11} the inclusion
$$k[X_{0}]\subseteq \bigcap_{D\in\F}\mathcal{O}_{v_{D}}\text{ implies that }\F\subseteq \F'.$$
By assumption, one can find a $G$-valuation $\nu = [s, p,\ell]$ with $\ell\neq 0$ and $p/\ell$ in the relative interior of some fiber $\D_{y}$. Then by Lemma \ref{l:valcenter}, $\nu$ has a center $\xi\in X$.
Assume that $\xi\not\in D$ for some $D\in\F'$. By \cite[Lemma 19.12]{Tim11} and
using that $X_{0}\cap D\neq \emptyset$, one can find $f\in k[X_{0}]^{(B)}$ such that $v(f) = 0$ and $v_{D}(f)>0$. But 
the fact that $p/\ell$ is in the interior relative of $\D_{y}$ implies that $v_{D'}(f)= 0$ for all $D'\in \F'$, which gives 
a contradiction and establishes $\F = \F'$. 
\end{proof}

\begin{example}
In this example, $G= \SL_{2}\times \SL_{2}$. We consider the spherical homogeneous $G$-space $\Omega = \SL_{2}/U\times \SL_{2}/U$, where $U$ is a maximal unipotent subgroup of $\SL_{2}$. Here $M = N = \ZZ^{2}$. We exactly have two colors $D_{1}, D_{2}$ in $\Omega$ which are respectively  sent on the first and second vectors of the canonical basis.  We denote by $(\D, \F)$ the colored polyhedral divisor defined by the equalities
$\D = \D_{0}\cdot [0] + \D_{1}\cdot [1] + \D_{\infty}\cdot [\infty]$ and $\F = \{D_{1}\}$. The locus of $\D$ is $\P^{1}$ and 
$$\D_{0} =  [(1,0), (1,1)]+\sigma,\,\,\D_{1} = \left(-\frac{1}{3},0\right)+\sigma,\,\,\D_{\infty} =\left(-\frac{5}{8},  0\right)+\sigma,\,\,$$ 
with $\sigma = \QQ_{\geq 0}(1,0) + \QQ_{\geq 0}(1,24)$ (see Figure 2).
Using the natural $G$-action on $\AA^{2}\times \AA^{2}$, the irreducible  representations inside $k[\Omega]$ are of the form $V(\lambda, \nu) = E(\lambda)\otimes_{k} F(\nu)$ ($\lambda, \nu\in\ZZ_{\geq 0}$) with
$$E(\lambda) = \bigoplus_{i + j = \lambda, i,j\geq 0}kx_{0}^{i}x_{1}^{j}\text{ and }F(\nu) = \bigoplus_{i + j = \nu, i,j\geq 0}ky_{0}^{i}y_{1}^{j}.$$
Then by \cite[Example 2.6]{Lan13} we have
$A(\P^{1}, \D) \simeq k[u,v,w,t]/(uv - w^{8} + t^{3}), $
$$\text{ where }u\mapsto y_{0},\, v\mapsto \frac{(1-z)^{8}}{z^{23}}{x_{0}}^{24}{y_{0}}^{-1},\, w\mapsto  \frac{1-z}{z^{3}}{x_{0}}^{3} \text{ and } t\mapsto \frac{(1-z)^{3}}{z^{8}}{x_{0}}^{8}.$$
Here $z$ is a local coordinate of $\P^{1}$, i.e., $k(\P^{1}) = k(z)$.  The subvariety $y_{0} = 0$ in $\Omega$ corresponds to the color $D_{2}$. 
Denote by $V$ the $G$-submodule generated by $y_{0}^{2}, y_{0}u,y_{0}v, y_{0}w, y_{0}t$ in $k(z)\otimes_{k}k(\Omega)$. Then $\bar{X}$ is the Zariski closure of the image of the morphism
$$\AA^{1}\setminus\{0\}\times \Omega\rightarrow \P(V^{\vee}),\,\, (z, [M_{1}],[M_{2}]) \mapsto [\phi_{z, [M_{1}], [M_{2}]}],$$
where the linear form $\phi_{z, [M_{1}], [M_{2}]}$ is the usual evaluation function on the triple $(z, [M_{1}], [M_{2}])$. Finally, by considering $y_{0}$
as element of the bidual $V \simeq V^{\vee\vee}$, the complement of $\{y_{0} = 0\}$ in $\bar{X}$ corresponds to the chart $X_{0}(\D,\F)$. 
\end{example}
\begin{figure}
		\centering
		\begin{tikzpicture}[scale=1/2]
			\shade[lower right=white, upper left=white, lower left=black, upper right=white,fill opacity=0.6] (1,0) -- (7,0) -- (7,6) -- (11/6,6) -- (1,1) -- (1,0) ;
			\shade[lower right=white, upper left=white, lower left=black, upper right=white,fill opacity=0.6] (9.5,0) -- (17,0) -- (17,6) -- (10.5,6) -- (9.5,0);
			\shade[lower right=white, upper left=white, lower left=black, upper right=white,fill opacity=0.6] (59/3,0) -- (27,0) -- (27,6) -- (62/3,6) -- (59/3,0);
			\draw[very thick,->]	(-1,0) -- (7.3,0);
			\draw[very thick,->]	(0,-1) -- (0,6);
			\draw[thick]	(1,0) -- (7,0)
						(1,0) -- (1,1)
						(1,1) -- (11/6,6);
			\draw[very thick,->]	(9,0) -- (17.3,0);
			\draw[very thick,->]	(10,-1) -- (10,6);
			\draw[thick]	(9.5,0) -- (17,0)
						(9.5,0) -- (10.5,6);
			\draw[very thick,->]	(19,0) -- (27.3,0);
			\draw[very thick,->]	(20,-1) -- (20,6);
			\draw[thick]	(59/3,0) -- (27,0)
						(59/3,0) -- (62/3,6);
			\draw (4,3) node {$\D_0$};
			\draw (14,3) node {$\D_1$};
			\draw (24,3) node {$\D_\infty$};
			\node[draw,fill,circle,inner sep=1.5pt] at (1,0) {};
			\node[draw,fill,circle,inner sep=1.5pt] at (0,1) {};
			\node[draw,fill,circle,inner sep=1.5pt] at (9.5,0) {};
			\node[draw,fill,circle,inner sep=1.5pt] at (59/3,0) {};
			\draw[thick,dashed] (0,1) -- (1,1);
			\draw (1,-0.7) node {$1$};
			\draw (-0.6,1) node {$1$};
			\draw (9.5,-0.9) node {$\frac{1}{3}$};
			\draw (59/3,-0.9) node {$\frac{5}{8}$};
		\end{tikzpicture}
		\caption{}
	\end{figure}
\section{Classification}\label{sec-3class}
In this section we present results to classify normal $G$-varieties with spherical orbits. As explained in the introduction, this classification breaks into two pieces:
(1) a birational part and (2) a biregular part. In Section \ref{s-birtype}, we explain how we can describe the equivariant birational type of a $G$-variety with spherical orbits and how we can go back to the trivial equivariant birational type case treated in Section \ref{s-combi}. Finally, our main result 
is formulated in Section \ref{s-classif} where we give a construction of any normal $G$-variety with spherical orbits in terms of colored divisorial fans. 
\subsection{Equivariant birational type}\label{s-birtype}
We start by collecting results on the birational type of a $G$-variety with spherical orbits. 
Our starting point is the following theorem due to Alexeev and Brion.
\begin{theorem}\cite[Theorem 3.1]{AB05} 
\label{t-AB}
Let $X$ be a $G$-variety with spherical orbits. Then there exist a closed spherical subgroup $H\subseteq G$  and a $G$-stable dense open subset $X_{1}\subseteq X$ such that any isotropy group of
a point of $X_{1}$ is conjugate to $H$.
\end{theorem}
The homogeneous space $\Omega =G/H$ in the above statement will be called 
\emph{the general orbit} of $X$ and $H$ the \emph{stabilizer in general position}. The reader is referred to \cite{Ric72} for various results on the existence of a unique general orbit for reductive group actions.
We will also use the following lemma latter.
\begin{lemme}
\label{l-finite group}
Let $F$ be a finite group acting on an integral scheme $\tilde{X}$ of finite type over $k$. Then the following assertions hold.
\begin{itemize}
\item[(i)] The $F$-action is faithfull if and only if it is generically free, i.e.,
for a general point $x\in \tilde{X}$, the stabilizer $F_{x}$ is trivial.
\end{itemize}
For the next points, assume that any $F$-orbit of $\tilde{X}$
is contained in an affine open subset. For instance, this applies for the case where $\tilde{X}$ is covered by $F$-stable quasi-projective open subsets.
\begin{itemize}
\item[(ii)] The $F$-scheme $\tilde{X}$ admits a good categorical quotient $\gamma:\tilde{X}\rightarrow X$, where $X= \tilde{X}/F$ is an integral scheme of finite type over $k$.
\item[(iii)] If the $F$-action is free, then the quotient map $\gamma: \tilde{X}\rightarrow X$
is an \'etale morphism.
\item[(iv)] The field extension $k(\tilde{X})/k(X)$ is Galois with Galois group $F$ if and only if
the $F$-action on $\tilde{X}$ is generically free.
\end{itemize}
\end{lemme}
\begin{proof}
Assertion (i) follows from a classical argument: the complement of the subset in which
the $F$-action is free is the finite union of closed subsets $\bigcup_{F_{1}}\tilde{X}^{F_{1}}$,
where $F_{1}$ runs over all subgroups of $F$ of cardinality $>1$.
Assertion (ii) is a consequence of \cite[Page 111, III, Theorem 1 (A)]{Mum70}.
For Assertion (iii), the fact that $\gamma: \tilde{X}\rightarrow X$ is a finite flat morphism
is explained in  \cite[Page 112, III, Theorem 1 (B)]{Mum70}. Finally, the morphism $\gamma$ is unramified since we work over a base field of characteristic zero. Let us show Assertion (iv). Since the extension $k(\tilde{X})/k(X)$
is separable, the number of points of a general fiber of $$\tilde{X}\rightarrow X\text{ is }[k(\tilde{X}):k(X)].$$ Hence the $F$-action is generically free if and only if the cardinality of $F$ is $[k(\tilde{X}):k(X)]$.
This finishes the proof of the lemma. 
\end{proof}

The next lemma is a straightforward observation but useful for the sequel. It determines the set of $H$-fixed points of the homogeneous space $G/H$. The proof is left to the reader. 
\begin{lemma}\label{l:pointfixed}
Let $\Omega$ be a homogeneous space $G/H$. Then we have the equality
$$\Omega^{H} = \{gH\,|\, g\in N_{G}(H)\}= N_{G}(H)/H$$
and for any $x\in \Omega^{H}$ the isotropy group $G_{x}$ is equal to $H$.
\end{lemma}
By a \emph{Galois covering} with Galois group $F$ we mean a dominant finite morphism 
$\gamma: \tilde{X}\rightarrow X$ such that the field extension $k(\tilde{X})/k(X)$ is Galois with Galois group $F$. 
The following classical theorem describes the birational type of any $G$-variety having a general stabilizer in terms of its rational quotient 
and its general orbit (see \cite[Theorem 2.13]{CKPR11}). It was inspired by the ideas of Popov and Vinberg developed in \cite[Section 2]{PV89}. Note that the theorem was also proved by Alexeev and Brion in \cite[Section 3.1]{AB06} for certain families of affine spherical varieties.
For reading convenience we give a detailed proof.

\begin{theorem}\label{t-CKPR}\cite[Theorem 2.13]{CKPR11}
Let $\XX$ be a $G$-variety having a general $G$-orbit $\Omega = G/H$ (possibly not spherical). Then there exist a variety $S$ and a $G$-equivariant rational map
$$\gamma:\tilde{\XX} := S\times \Omega\dashrightarrow \XX$$
which is  a finite Galois covering on a $G$-stable dense open subset. After shrinking $S$, the map $\gamma$ is constructed in a such way that it induces 
a Galois covering $S\rightarrow S'$
giving rise to a $G$-equivariant isomorphim between $k(\tilde{\XX})$ and the fraction field of $k(S)\otimes_{k(S')}k(\XX)$, where $k(\XX)^{G} = k(S')$.
\end{theorem}
\begin{proof}
By the Rosenlicht theorem \cite{Ros63}, \cite[Satz 2.2]{Spr89}, there exist a $G$-stable dense
open subset $V\subseteq \XX$ such that every $G$-orbit is $G$-isomorphic to $\Omega$ and a global geometric quotient $\pi_{0}:V\rightarrow S'$ where $S'$ is a $d$-dimensional variety. Note that the fibers of $\pi_{0}$ are exactly the $G$-orbits of $V$ and they intersect the fixed point subscheme $V^{H}$. Using the equality 
$V^{H} = \bigcup_{x\in V}(G\cdot x)^{H}$ and Lemma \ref{l:pointfixed}, we observe that all the isotropy groups of points of  $V^{H}$ for the $G$-action on $V$ are equal to $H$. 
Let us take a $d$-dimensional irreducible reduced closed subset $S_{1}$ of $V^{H}$ sending dominantly on $S'$ by the morphism $\pi_{0}$; 
this latter being equivalent to find a closed schematic point in the fiber
of the generic point $\eta\in S'$ of $\pi_{0}|_{V'}$. 
By shrinking $S_{1}$ and $S'$ if necessary, one can find a 
variety $S$ (via a Galois extension of $k(S_{1})$) and a commutative diagram of finite morphisms
$$
\xymatrix{
S \ar[r]^{\pi_{2}} \ar[rd]_{\pi_{1}} &
S_{1} \ar[d]^{\pi_{0}} \\
& S'}
$$
such that $\pi_{1}$ is a Galois covering. Therefore, we get a new commutative diagram
$$
\xymatrix{
S\times\Omega \ar[r]^{\gamma_{1}} \ar[rd]_{\gamma} &
S\times_{S'}V \ar[d]^{{\rm proj_{2}}} \\
& V}
$$
Here $\gamma_{1}$ is defined by the formula $\gamma_{1}(s, gH) = (s, g\cdot \pi_{2}(s))$
for all $s\in S$ and $g\in G$. By changing $V$ by a $G$-stable open subset, we may assume that $\gamma$ (and therefore $\gamma_{1}$) is surjective. 
Indeed, $S_{1}$ is chosen in a such way that it intersects a $G$-orbit of $V$ taken in general position.

The fact that 
$\gamma_{1}$ is injective is clear. Therefore applying the Zariski Main Theorem,
the map $\gamma_{1}$ is a $G$-equivariant birational morphism. Finally, the morphism
${\rm proj_{2}}$ is a finite morphism since it is obtained by base change of a finite one. It is hence clear that $\gamma$ is a Galois covering. This concludes the proof of the theorem.
\end{proof}
\begin{corollary}\label{t-Galois} Let $\XX$ be a $G$-variety with a general orbit
$\Omega = G/H$ (possibly not spherical) and consider the rational map
$$\gamma: \tilde{\XX}= S\times \Omega\dashrightarrow \XX$$
obtained from Theorem \ref{t-CKPR}.
Let $F$ be the corresponding Galois group acting by $G$-equivariant birational transformations on
$\tilde{\XX}$.
Then for any $G$-model $X$
of $\XX$, there exists an $F$-stable $G$-model $\tilde{X}$ of 
$\tilde{\XX}$ with a regular $F$-action such that $X = \tilde{X}/F$. 
\end{corollary}
\begin{proof}
Let $K = k(\XX)$ and $L = k(\tilde{\XX})$. We consider a $G$-model $X$ of $\XX$.
We define the variety $\tilde{X}$ as the normalization of $X$ with respect to the finite field extension $L/K$. One can construct $\tilde{X}$ as follows: denote by
$\mathcal{A}^{L}$ the sheaf of $\mathcal{O}_{X}$-algebras associated with
the presheaf
$$U\mapsto \overline{\mathcal{O}}_{X}(U)\subseteq L,$$
where $\overline{\mathcal{O}}_{X}(U)$ is the integral closure of $\mathcal{O}_{X}(U)$
in the field extension $L$ for any dense open subset $U\subseteq X$. 
Then $\tilde{X}$ is the relative spectrum $\spec_{X}\mathcal{A}^{L}$. In particular, this normal scheme is finite over $X$, by \cite[Section 1, Proposition 1.2.4]{EGAII}, it is separated over $k$ and is therefore a model of $\tilde{\XX}$.
Moreover, $F$ acts regularly on $\tilde{X}$, has an open covering by $F$-stable affine subsets and $X$ is identified with the quotient $\tilde{X}/F$. 

Let us show that the open subscheme $\tilde{X}\subseteq \sc(\tilde{\XX})$ is a $G$-model of $\tilde{\XX}$.
We may suppose that $X$ is simple. Let $X_{0}$ be a $B$-chart of $X$ intersecting any $G$-orbit. 
Its coordinate ring can be regarded as
$$k[X_{0}] = \bigcap_{v\in \hs'}\mathcal{O}_{v}\cap
\bigcap_{D\in \F'\cup D(\XX)}\mathcal{O}_{v_{D}},$$
where $(\hs', \F')$ is an admissible pair as in \cite[Theorem 3]{Tim00}
and $D(\XX)$ is the set of all prime divisors on any $G$-model of $\XX$ 
that are not $B$-stable. Note that $\hs'$ is a set of $G$-valuations of $K$ and
$\F'$ is a set of colors.
Using the map $\gamma$, we deduce that the integral closure of 
$$\bigcap_{D\in \F'\cup D(\XX)} \mathcal{O}_{v_{D}}\text{ in } L\text{ is } \bigcap_{D\in \F \cup D(\tilde{\XX})}\mathcal{O}_{v_{D}},$$
where $\F$ consists of the irreducible components of $\gamma^{-1}(D)$ for $D$ running through $\F'$. Moreover, a discrete valuation $v$ on $K$ is $G$-invariant if and
only if any of its extensions on $L$ is $G$-invariant. This follows from \cite[Corollary 19.6]{Tim11} and the fact that the Galois group $F$ acts transitively on the set of valuations rings of $L$ dominating $\mathcal{O}_{v}$ (compare \cite[Chapter 4, Exercise 12.1]{Mat89}). 
Hence by considering the set $\hs''$ of valuations extending those of
$\hs'$, the integral closure of $k[X_{0}]$ in $L$ is
$$A= \bigcap_{v\in \hs''}\mathcal{O}_{v}\cap
\bigcap_{D\in \F\cup D(\tilde{\XX})}\mathcal{O}_{v_{D}}.$$

Clearly, the pair $(\hs'', \F)$ satisfies Conditions $(i)$ and $(ii)$ of Lemma \ref{l-chartt}. 
We conclude that $\tilde{X}_{0} := \spec\, A$ is a $B$-chart of $\gsc(\tilde{\XX})$.
As $\gamma$ is a $G$-equivariant map, the subset $g\cdot \tilde{X}_{0}$ coincides with the preimage of $g\cdot X_{0}$ under the natural map $\tilde{X}\rightarrow X$ for any $g\in G$. Thus $\tilde{X} = G\cdot \tilde{X}_{0}$ is a $G$-model of $\tilde{\XX}$. This completes the proof 
of the theorem.
\end{proof}

The next step is to give an interpretation of the equivariant 
birational classes of $G$-varieties with spherical orbits in terms
of Galois cohomology.
Let $S'$ be a variety with function field $E= k(S')$. Let us fix an algebraic closure $\overline{E}$ of the field $E$.
The $G$-isomorphism classes of forms of the $G$-algebra $E\otimes_{k}k(\Omega)$ over $E$ are parameterized by the first pointed 
set of Galois cohomology 
$$\mathfrak{H} :=H^{1}(\overline{E}/E, \Aut_{G}(\overline{E}\otimes_{k}k(\Omega)))$$
with coefficients in the $G$-equivariant automorphism group of the $\overline{E}$-algebra
$\overline{E}\otimes_{k} k(\Omega)$ (see \cite[Chapter III, \S1]{Ser97}). Writing $K$ for $N_{G}(H)/H$, we have 
$$ \Aut_{G}(\overline{E}\otimes_{k}k(\Omega)) = \Aut_{G(\overline{E})}(G(\overline{E})/H(\overline{E})) = N_{G}(H)(\overline{E})/H(\overline{E}) = K(\overline{E}).$$
We note that according to \cite[p 283]{BP87} the algebraic group $K$ is a split diagonalizable group.
Hence decomposing $K = K_{\rm tor}\times K_{0}$ into a direct product with its torsion group $K_{\rm tor}$ and its neutral component $K_{0}$, we get the following identifications
$$ \mathfrak{H}\simeq H^{1}(\overline{E}/E, K_{0}(\overline{E}))\oplus H^{1}(\overline{E}/E, K_{\rm tor}(\overline{E}))\simeq H^{1}(\overline{E}/E, K_{\rm tor}(\overline{E})),$$
where the last isomorphism comes from the Hilbert Theorem 90.

Finally, from Theorem \ref{t-CKPR} we observe that $\mathfrak{H}$ classifies the birational type of any $G$-variety $X$ with spherical orbits such that $E = k(X)^{G}$ and  $\overline{E}\otimes_{E}k(X)\simeq \overline{E}\otimes_{k} k(\Omega)$. With this in hand, we obtain the following corollary.
\begin{corollaire}\label{c:cohoclass}
The $G$-equivariant birational class of a $G$-variety with general spherical orbit $\Omega =G/H$ and rational quotient $S'$ determine an element of the set $\mathfrak{H}$, and vice-versa. Moreover,
if $N_{G}(H)$ is connected, then $\mathfrak{H}$ is a singleton. 
\end{corollaire}
\begin{proof}
The first claim is a direct consequence of Corollary \ref{t-Galois} and the above discussion.
For the last claim, the connectedness of $N_{G}(H)$ implies that $H^{1}(\overline{E}/E, K_{\rm tor}(\overline{E})) =\{1\}$ and so $\mathfrak{H}$ is a singleton.
\end{proof}
We can reformulate the last corollary for the case of $G$-varieties with horospherical orbits. In particular, we recover \cite[Satz 2]{Kno90}.
\begin{corollaire}
Let $\XX$ be a $G$-variety with general horospherical orbit $\Omega = G/H$ and
geometric quotient $S$ on a $G$-stable dense open subset. Then $\XX$ is $G$-equivariantly birational to $S\times \Omega$, where $G$ acts on $S\times \Omega$ with 
the trivial action on $S$ and the natural one on $\Omega$. 
\end{corollaire}
\begin{proof}
Since $N_{G}(H)$ is a parabolic subgroup (hence connected), we conclude that $\mathfrak{H}$ is a singleton by Corollary \ref{c:cohoclass}.
\end{proof}
\begin{remark}\label{rem-GaloisAction}
Denote by $G_{\overline{E}/ E}$ the absolute Galois group of $E$.
A Galois cohomology class in $\mathfrak{H}$ corresponds to an equivalence class of semi-linear actions 
of the group $G_{\overline{E}/E}$. Indeed, if 
$$\alpha: G_{\overline{E}/E}\rightarrow \Aut_{G}(\overline{E}\otimes_{k}k(\Omega)), \,\, g\mapsto \alpha_{g}$$ is a continuous $1$-cocycle, then the corresponding semi-linear action is given by the formula
$$g\cdot (\lambda\otimes f) := \alpha_{g}\circ g^{\star} (\lambda\otimes f),\text{ for all }
g\in G_{\overline{E}/E}, \lambda\in \overline{E}, \text{ and } f\in k(\Omega).$$
Here the transformation $g^{\star}:\overline{E}\otimes_{k}k(\Omega)\rightarrow \overline{E}\otimes_{k}k(\Omega)$ is induced by the usual Galois action on the field 
extension $\overline{E}$. 
\end{remark}
The next proposition gives an explicit description of the $F$-action on the product
$S\times\Omega$ considered in Theorem \ref{t-CKPR}.
\begin{proposition}\label{prop-GaloisAction}
The $G_{\overline{E}/E}$-action on $\overline{E}\otimes_{k}k(\Omega)$ is induced
by the usual Galois action on $\overline{E}$ and by a  continuous $1$-cocycle with values in the equivariant automorphism group of $\Omega$ (up to change with an equivalent semi-linear action). Regarding the notation of Theorem \ref{t-CKPR}, 
this geometrically means that the $F$-action on the product $S\times \Omega$, assumed for simplicity regular, is induced by a generically free $F$-action on $S$ and by a $G$-equivariant one 
on $\Omega$.
\end{proposition}
\begin{proof}
We will keep the same notation as in Remark \ref{rem-GaloisAction}.
Set $\Omega_{\overline{E}} := \Omega\times_{\spec\, k}\spec\, \overline{E}.$
For a $1$-cocycle $\alpha$  representing a cohomology class in $\mathfrak{H}$, the action of $\alpha_{g}$ on $\overline{E}(\Omega_{\overline{E}})$ is determined by the action on the $B$-eigenfunctions $\chi^{m}\in k(\Omega)^{(B)}$ (see \cite[Theorem 21.5]{Tim11}). Moreover, 
$\alpha_{g}(\chi^{m}) = \omega_{g}(m)\chi^{m}$ for all $m\in M$
and for some group homomorphism $\omega_{g}:M\rightarrow \G_{m}(\overline{E})$.
 From the relations
$$\alpha_{g}\circ g^{\star}\circ \alpha_{g'}\circ (g')^{\star}(\chi^{m}) = \alpha_{gg'} \circ (gg')^{\star}(\chi^{m}) = \alpha_{gg'}(\chi^{m}) \text{ for all } g,g'\in  G_{\overline{E}/E},$$
we remark that
$$\omega_{gg'}(m) = \omega_{g}(m)g^{\star}(\omega_{g'}(m))$$
and so $g\mapsto \omega_{g}$ is a $1$-cocycle with values in $K(\overline{E})$.
We get therefore an automorphism $\mathfrak{H}\rightarrow \mathfrak{H}$ sending the class of $\alpha$ to the class of $\omega$. In the sequel we will identify these two objects. 

Now using the isomorphism $$\phi_{\star}:\mathfrak{H}
\rightarrow H^{1}(\overline{E}/E, K_{\rm tor}(\overline{E})),\,\, [g\mapsto \omega_{g}]\mapsto [g\mapsto \phi(\omega_{g})],$$
where $\phi: K\rightarrow K_{\rm tor}$ is the natural projection, we may change the $1$-cocycle $\omega$ into a cohomologous one where its values are in $K_{\rm tor}(\overline{E})$. 
Doing this change and using that $k$ is algebraically closed, we obtain that 
$$\omega_{g} \in K_{\rm tor}(\overline{E}) = K_{\rm tor}(k)\text{ for all }g\in G_{\overline{E}/E},$$
and thus the $\overline{E}$-automorphism $\alpha_{g}$ is the scalar extension of an automorphism of $k(\Omega)$.
\end{proof}

Let us end this section with some examples.

\begin{example}{\em Diagonalizable matrices.} In general, a $G$-variety with a unique general orbit has non trivial equivariant birational type. A basic example (see \cite[Page 23]{Bri96}) is to look at the action by conjugacy of $G= {\rm GL}_{n}$ on the space of $n\times n$-matrices $X$. The subset of diagonalizable matrices with distinct eigenvalues forms an open subset $X_{0}\subseteq X$ where all the stabilizers are conjugate
to a maximal torus $T$. In addition, the $T$-fixed point set of $X_{0}$ is irreducible  while the $T$-fixed point set of $S\times G/T$ for any variety $S$ is isomorphic to the product $S\times W$, where $W$ is the Weyl group of $G$ (see Lemma \ref{l:pointfixed}). This latter is the symmetric group with $n$ letters, where $n = \dim\, T$. Hence the equivariant birational type of $X$ cannot be trivial for $n>1$.
\end{example}
\begin{example}[\cite{Arz97c}] \emph{${\rm SL}_{2}$-threefolds.}
Let $G = {\rm SL}_{2}$ and let $H\subseteq G$ be the normalizer of a maximal torus.
We will consider normal affine threefolds that have a unique ${\rm SL }_{2}$-fixed point, an affine line as algebraic quotient, and general stabilizer $H$. 
It follows from
Corollary \ref{c:cohoclass} that such $\SL_{2}$-varieties have trivial equivariant birational type.  Arzhantsev gave in \cite{Arz97c} a concrete classification of these threefolds in terms of  a pair $(a,b)$ of coprime non-negative integers. Actually, the way of passing from our description to the one in  \cite[Section 4]{Arz97c} is to associate to the marked pair $(a, b)$ the polyhedral divisor $$(\D, \F) = \left((-b/a + \QQ_{\geq 0})\cdot [0], \F\right)\in\CP(\AA^{1}, \ss),$$ where $\ss$ is the homogeneous spherical datum of $\Omega = \SL_{2}/H$ and $\F$ consists to the unique color. 

Let us start with an explicit example to illustrate this correspondence.
We first consider the case of a single marked point $(\ell, 1)$ on the affine line $\AA^{1}$.  Let us denote by $x_{1},x_{2},  x_{3}$ the coordinates of the affine space $S^{2}V\simeq \AA^{3}$, where $V$ is a two-dimensional vector space. The matrix $A =\left(\begin{smallmatrix} a_{11}   & a_{12} \\ a_{21}   &   a_{22} \end{smallmatrix}\right)\in \SL_{2}$ acts on $S^{2}V$ via the formulae
$$ A\cdot x_{1} = a^{2}_{11}x_{1}  + 2a_{11}a_{12}x_{2}  + a_{12}^{2}x_{3};$$ $$A\cdot x_{2} = a_{11}a_{21}x_{1} +(a_{11}a_{22}+a_{12}a_{21})x_{2} +a_{12}a_{22}x_{3};$$ $$A\cdot x_{3} = a_{21}^{2}x_{1} + 2a_{21}a_{22}x_{2} + a_{22}^{2}x_{3}.$$  
Note that the stabilizer at any point $(0,\lambda, 0)\in S^{2}V$ for $\lambda\in k^{\star}$ is conjugate to $H$. Also, if $U$ is the upper unipotent matrix subgroup of $\SL_{2}$, then
$$k[S^{2}V]^{U} =k[x:= x_{3}, z:= x_{2}^{2} - x_{1}x_{3}].$$
The degrees of $z$ and $x$ are respectively $0$ and $1$. So $S^{2}V$ is given by $(\D_{1},\F)$, where $\D_{1}$ is the $\QQ_{\geq 0}$-polyhedral divisor over $\AA^{1}$
satisfying $\D_{1}(1) = - [0]$. 
This can be seen geometrically by regarding the fibers of the quotient map
$$S^{2} V\rightarrow \AA^{1}, \,\, (x_{1}, x_{2}, x_{3})\mapsto x_{2}^{2} - x_{1}x_{3},$$
where the fiber in $0$ has a unique fixed point $(0,0,0)$ and the other fibers are all ${\rm SL}_{2}$-isomorphic to ${\rm SL}_{2}/H$. The unique non-closed orbit of the fiber in $0$ is ${\rm SL}_{2}$-isomorphic to ${\rm SL}_{2}/U_{2}$, where $U_{2} = \left \{\left(\begin{smallmatrix} \pm 1   & \lambda  \\ 0   &   \pm 1 \end{smallmatrix}\right) ; \lambda \in k\right\}$. The valuation hypercone 
$$\Vh = \{[v_{z}, a, b]\,|\, a\in \QQ_{\leq 0}, b\in \QQ_{\geq 0}, z\in \P^{1}\}$$
of $k(S^{2}V)$ intersects in its relative the Cayley cone $$C_{0}(\D_{1}) = \{[v_{0}, a, b ]\, |\, (a, b)\in \QQ_{\geq 0} (1, 0) + \QQ_{\geq 0}(-1, 1)\}$$ at $0$ 
but this does not hold for all the other Cayley cones.

The group $\mu_{2\ell}(k)$ acts on $S^{2}V$ via the usual multiplication on the coordinates. In this way, it is easy to see that the general stabilizer of  $S^{2}V/\!\!/\mu_{2\ell}(k)$ is $H$. Letting $\chi^{1} = x^{2}/z$ and $t = z^{\ell}$, we further observe that 
$$k[ S^{2}V/\!\!/\mu_{2\ell}(k)]^{U} = k[S^{2}V/\!\!/U]^{\mu_{2\ell}(k)} =k[x^{2\ell}, x^{2\ell-2}z, \dots, z^{\ell}] = k[t,t\chi^{1}, \ldots, t\chi^{\ell}]\simeq A(\AA_{1}, \D_{\ell}),$$
where $\D_{\ell}$ is the $\QQ_{\geq 0}$-polyhedral divisor over $\AA^{1}$ determined by the $\QQ$-divisor $\D_{\ell}(1) = -\frac{1}{\ell}\cdot [0]$.

Let $s\in\ZZ_{>0}$ be an integer coprime to $\ell$. Now for a single 
marked pair $(\ell, s)$ on $\AA^{1}$, the last step is to modify our variety via the cartesian commutative diagram
$$\xymatrix{
\AA^{1}\times_{\AA^{1}}S^{2}V/\!\!/\mu_{2\ell}(k) \ar[d] \ar[r] & S^{2}V/\!\!/\mu_{2\ell}(k)\ar[d]\\
\AA^{1} \ar[r]^{t\mapsto t^{s}}  &\AA^{1}.}$$
Denote by $\mathcal{X}_{\ell, s}$ the normalization of $\AA^{1}\times_{\AA^{1}}S^{2}V/\!\!/\mu_{2\ell}(k)$.
Then the algebra $k[\mathcal{X}_{\ell,s}]^{U}$ is identified with the integral closure
of $A(\AA^{1}, \D_{\ell})[t^{1/s}]$ in the function field $k(\chi^{1}, t^{1/s})$ (compare with \cite[Section D, Lemma D.6]{Tim11}). It is therefore isomorphic 
to $A(\AA^{1}, \D_{\ell}^{s}),$ where $\D_{\ell}^{s}$ is the $\QQ_{\geq 0}$-polyhedral divisor over $\AA^{1}$ defined by  $\D_{\ell}^{s}(1) = -\frac{s}{\ell}\cdot [0]$ (apply \cite[Theorem 2.4]{Lan13}). We conclude that $\mathcal{X}_{\ell,s}$ is described by the colored polyhedral divisor $(\D_{\ell}^{s},\F)$. 
\end{example}
\subsection{Classification of $G$-varieties with spherical orbits}
\label{s-classif}
We now pass to the biregular classification
of normal $G$-varieties with spherical orbits. Before stating our main result (see Theorem \ref{t-claf}),
we start by introducing the notion of splitting associated to a $G$-variety with spherical orbits $\XX$. 
\\

\begin{definition}\label{def-split}
Let $\Omega$ be the general orbit of $\XX$. By a \emph{splitting} of $\XX$ we mean the datum $\gamma$ of
a $G$-equivariant generically free action of a finite abelian group $F$ on a product $\tilde{\XX} := S \times \Omega$
such that 
\begin{itemize}
\item $S$ is a quasi-projective variety;
\item $\gamma$ corresponds to a non-trivial cohomology class (see \ref{c:cohoclass}) if $\gamma$ is not the trivial action;  
\item  $F$ acts on $\tilde{\XX}$ via a generically free $F$-action on $S$ and an equivariant one on $\Omega$;
\item The quotient $\tilde{\XX}/F$ is $G$-birational to $\XX$. 
\end{itemize}
We denote by the same letter the invariant rational map 
$\gamma:\tilde{\XX}\dashrightarrow \XX.$ 
\end{definition}
\begin{remark}
Note that $\XX$ admits always a splitting. Indeed by the results of Section \ref{s-birtype}, there exists a finite group $F_{0}$ acting on a product $S_{0}\times \Omega$ such that
$\tilde{\XX}/F_{0}$ is $G$-birational to $\XX$. The $F_{0}$-action satisfies the first and third conditions of Definition \ref{def-split}. Since the equivariant automorphism group of $\Omega$ is abelian, the $F_{0}$-action on $\Omega$ factorises into an $F_{0}/[F_{0},F_{0}]$-action,
where $[F_{0},F_{0}]$ is the derived subgroup of $F_{0}$. Therefore, we may change $F_{0}$ by $F_{0}/[F_{0},F_{0}]$ and $S_{0}$ by $S_{0}/[F_{0},F_{0}]$ and this yields the existence of a splitting for $\XX$. 
\end{remark}
\begin{remark}
Conversely, if  $S$ is a quasi-projective variety with a generically free $F$-action and $F$ acts by $G$-equivariant automorphisms on $\Omega$, then the quotient
$\XX_{0}:= (S \times \Omega)/F$ on the product exists (since  $S \times \Omega$ is quasi-projective) and $\XX_{0}$ is a $G$-variety having $\Omega$ as general orbit. 
\end{remark}
Until now $F$ is the Galois group of the splitting $\gamma$ of $\XX$.
In order to study the $G$-models of $\XX$ we need to consider the models of $S$ admitting a natural $F$-action. This leads
us to state the following lemma where the proof, left to the reader, is similar as in \cite[Proposition 12.2]{Tim11}. 
\begin{lemme}
The subset $$\fsc(S) :=\left\{\xi\in\sc(S)\,\,|\,\,\alpha^{\star}(\mathcal{O}_{\xi, \sc(S)})\subseteq \mathcal{O}_{(1,\xi), F\times \sc(S)}\right\}$$
is a dense open subscheme of $\sc(S)$, where $\alpha$ is the comorphism of the rational $F$-action.   
\end{lemme}
The lattice of $B$-weights $M_{\gamma}$ of $k(\XX)$ is in general different from the 
one of $k(\tilde{\XX})$. More precisely,
$$M_{\gamma} = \{m\in M\,|\, m\,\text{  weight of }f\in k(\tilde{\XX})^{(B)}\cap k(\tilde{\XX})^{F}\}.$$ 
However, $M_{\gamma}$ is a sublattice of $M$ of finite index and so $M_{\QQ}\simeq \QQ\otimes_{\ZZ}M_{\gamma}$. If further $N_{\gamma} = \Hom(M_{\gamma}, \ZZ)$ is the dual lattice,
then $N_{\QQ} \simeq \QQ\otimes_{\ZZ}N_{\gamma}$. Thus we may regard colored polyhedral divisors with respect to $N_{\QQ}$ for describing the $G$-models of $\XX$ (see Definition \ref{d: F-divfan} later on). Let $\F_{\Omega}$ be the set of colors of $\Omega$. The group $F$ naturally acts on $\F_{\Omega}$ by translation. Denoting by $\F_{\Omega}/F$ the set
of $F$-orbits,  we observe that $\F_{\Omega}/F$ is in bijection with
the set of colors of $\gsc(\XX)$. For any $a\in k(\XX)$ and any $D\in \F_{\Omega}$, it follows that
$$v_{gD}(a) = v_{D}(g\cdot a) = v_{D}(a)\text{ for any } g\in F,$$
where $a$ is seen as an $F$-invariant rational function on $\tilde{\XX}$ and $D$ is a
color of $\gsc(\tilde{\XX})$. This defines a \emph{coloration map} $\varrho_{\gamma}: \F_{\Omega}/F\rightarrow N_{\gamma}$. This map will play a role in the computation
of the divisor class group of a $G$-model of $\XX$ in Theorem \ref{t-divisor1}.

With the notation of Section \ref{s-chi}, let $\xi = f\otimes \chi^{m}\in A_{M}\subseteq k(\tilde{\XX})$ be a $B$-eigenfunction invariant by the $F$-action (here $f\in k(S)^{\star}$ and $m\in M_{\gamma}$). If $g\in F$,
then the action of $g$ on $\xi$ is given by 
$$g\cdot \xi = g\cdot f\otimes \omega_{g}(m)\cdot \chi^{m},$$
where $\omega_{g}:M\rightarrow \G_{m} = k^{\star}$ is a group homomorphism (see the proof of Proposition \ref{prop-GaloisAction}).
Looking at the value of $v(g\cdot \xi)$ for any $G$-valuations $v$ of $k(\tilde{\XX})$, we obtain an action of $F$ on the set $\Vh$ of $G$-valuations of $k(\tilde{\XX})$. It is determined by the relations
$g\cdot v = [g\cdot s, a, b]$ for any $g\in F$ and any $v = [s, a, b]\in \Vh.$
\\

Let us introduce the notion of colored divisorial fans with respect to a splitting $\gamma$ of $\XX$. We let $\Gamma$ be a smooth projective $F$-model of $S$.
\begin{definition}\label{d: F-divfan}
A \emph{colored divisorial fan} on $(\Gamma, \ss, \gamma)$ is a usual colored divisorial fan $\EE$ on  $(\Gamma, \ss)$ (see Definition \ref{d:divfan}) with the following properties.
\begin{itemize}
\item[(i)] For every element $(\D,\F)\in \EE$, the subset $C(\D)\cap \Vh$ is $F$-stable for the $F$-action on the valuation set $\Vh$.
\item[(ii)] For every element $(\D,\F)\in \EE$, the set of colors $\F$ is $F$-stable.
In particular, $\varrho:\F\rightarrow N$ factors through the map $\varrho_{\gamma}:\F/F\rightarrow N_{\gamma}.$ 
\end{itemize}
\end{definition}
Before stating our result on the classification of $G$-models of $\XX$, we start with the following lemma which gives conditions for the regularity of the birational $F$-actions on the simple $G$-models of $\tilde{\XX}$. 
\begin{lemma}
\label{l:Fcolpp}
Let $\tilde{X} =X(\D,\F)$ be a simple $G$-model of $\tilde{\XX}= S\times\Omega$, where $(\D,\F)\in \CP(\Gamma, \ss)$ is a colored polyhedral divisor. If $(\D,\F)$ satisfies Conditions $(i),(ii)$ of Definition \ref{d: F-divfan}, then the birational $F$-action on $\tilde{X}$ is regular.
\end{lemma}
\begin{proof}
Assume that $(\D,\F)$ satisfies Definition \ref{d: F-divfan}. Since $\tilde{X} = G\cdot X_{0}$ with $X_{0} := X_{0}(\D, \F)$, it suffices
to show that $F$ acts regularly on $X_{0}$. But this follows from the previous
discussion on how $F$ acts on colors and $G$-valuations and the description 
of $B$-charts in Theorem \ref{t-ppcol} in terms of intersection of valuation rings. 
\end{proof}
The following theorem completes the construction of normal $G$-varieties with spherical orbits. Indeed, each normal $G$-variety $X$ with spherical orbits has a general orbit $\Omega$ according to Alexeev and Brion (see \cite[Theorem 3.1]{AB05}).
The spherical homogeneous space $\Omega$ is entirely described by its homogeneous spherical datum $\ss$ (see \cite{Lun01, Los09, BP16}). Moreover, the $G$-equivariant birational type of $X$ can be explicitly constructed via a Galois covering with total space a trivial family $S\times \Omega$ (see Corollary \ref{t-Galois} and Proposition \ref{prop-GaloisAction}) and it is determined by a Galois cohomology class (see Corollary \ref{c:cohoclass}).
Finally as stated thereafter, one can construct $X$ by a geometric and combinatorial object $\EE$ depending only on the $G$-equivariant birational type of $X$.    
\begin{theorem}\label{t-claf}
Let $\XX$ be an arbitrary $G$-variety with spherical orbits. Let 
$\gamma: \tilde{\XX}\dashrightarrow \XX$ be a splitting (see Definition \ref{def-split})  with finite abelian Galois group $F$. Denote by $\ss$ the homogeneous spherical datum of the general orbit of $\XX$.  Let $\EE$
be a colored divisorial fan on $(\Gamma, \ss, \gamma)$. Then every local chart
$X(\D, \F)$ corresponding to $(\D,\F)\in\EE$ admits a $G$-equivariant regular $F$-action coming from the rational action on $\tilde{\XX}$. In addition, 
$$X(\EE, \gamma) := \bigcup_{(\D,\F)\in \EE} X(\D,\F, \gamma)\subseteq \gsc(\XX)$$
is a $G$-model of $\XX$, where $X(\D,\F, \gamma) = X(\D,\F)/F$. The subset $X(\D\cap \D',\F\cap \F',\gamma)$
is identified with the intersection $$X(\D,\F,\gamma)\cap X(\D',\F',\gamma)\text{ for all }(\D,\F), (\D',\F')\in \EE.$$
Conversely, any $G$-model of $\XX$ arises in this way. The variety $X(\EE, \gamma)$ is complete if and only if $\EE$ is complete.   
\end{theorem}  
\begin{proof}
Let $\EE$ be a colored divisorial fan on $(\Gamma, \ss, \gamma)$. Note that each $F$-variety $X(\D,\F)$ is quasi-projective (see \cite[Section 5, Lemma 2(1)] {Tim00}) and therefore the quotient $X(\D,\F,\gamma)$ is well-defined. Hence by combining Theorem \ref{t:class1} and Lemmata \ref{l-finite group}(ii), \ref{l:Fcolpp},
we construct the quotient space $X(\EE,\gamma) = X(\EE)/F$ as a normal $G$-scheme of finite type over $k$ which is $G$-birational to $\XX$. Moreover,
the separateness of $X(\EE,\gamma)$ is equivalent to the one of $X(\EE)$ (see \cite[Expos\'e V, Corollaire 1.5]{SGA1}). We conclude that $X(\EE,\gamma)$ is a
$G$-model of the $G$-variety $\XX$.

Conversely, let us consider a $G$-model $X$ of $\XX$ and again denote by the same letter $\gamma: \tilde{X}\rightarrow X$ the quotient map obtained from Corollary \ref{t-Galois}.
Since the morphism $\gamma$ is affine and $G$-equivariant, the $G$-variety $\tilde{X}$ admits a finite open covering of simple $G$-models $(\tilde{X}_{i})_{i\in I}$ provided by $\tilde{X}_{i} = \gamma^{-1}(X_{i})$, where $X_{i}$ is a simple $G$-stable dense open subset of $X$. Now each open subset $\tilde{X}_{i}$ is described by a colored polyhedral divisor $(\D^{i},\F^{i})\in \CP(\Gamma_{i},\ss)$, where $\Gamma_{i}$ is a normal semi-projective variety which is a model of $S$ (see Theorem \ref{t-ppcol}). Using Proposition \ref{p:pullbackpp}, we may assume that each $\D^{i}$ is minimal in the sense of \cite[Definition 8.7]{AH06}. 
Since $k[X_{0}(\D^{i},\F^{i})]$ is $F$-stable and the group $F$ acts as well on $k[X_{0}(\D^{i},\F^{i})]^{U} = A(\Gamma,\D^{i})$, by \cite[Theorem 8.8]{AH06} we deduce that Conditions (i), (ii) of Definition \ref{d: F-divfan} are satisfied for each $(\D^{i},\F^{i})$. 

Again by resolving the indeterminacy in an $F$-equivariant way, we may construct a smooth projective $F$-model $\Gamma$ in $\fsc(S)$ that dominates
all the $\Gamma_{i}$ as in the argument of the proof of \cite[Theorem 5.8]{AHS08}. By pulling back the colored polyhedral divisors $(\D^{i},\F^{i})$ on $\Gamma$ (see Proposition \ref{p:pullbackpp}), we obtain a colored divisorial fan $\EE$ on $(\Gamma, \ss, \gamma)$ such that $\tilde{X} = X(\EE)$. Therefore the quotient space gives $X = X(\EE,\gamma)$. The last claim is a direct consequence of Proposition \ref{p:completef} and the fact that the quotient map $\gamma: X(\EE)\rightarrow X(\EE,\gamma)$ is a proper morphism. This finishes the proof of the theorem.
\end{proof}

\section{Invariant Weil divisors}\label{sec-Weil}
For a normal $G$-variety $X$ with spherical orbits,
our next task is to give a parameterization of the $B$-divisors of $X$ in terms 
of its defining colored fan $\EE$ (see Theorem \ref{t-divisor1} and see \cite[Corollary 3.15]{PS11} for case of $\TT$-varieties). The reason for this
is that any prime divisor of $X$ is linearly equivalent to a $B$-stable one 
according to \cite{FMSS95}. This allows us, in terms of our parameterization, to give 
a presentation by generators and relations of the divisor class group of $X$.
In the case of a toric $\TT$-variety $V$ with a defining fan $\EE_{V}$, the $\TT$-divisors of $X$ are naturally in bijection with the one-dimensional cones of $\EE_{V}$. This description is useful in practice and we hope for similar applications in our context.

In this section, we fix a normal $G$-variety with spherical orbits $\XX$.
We consider a splitting $\gamma: \tilde{\XX}\dashrightarrow \XX$ with Galois group $F$ and we let $\EE$
be a colored divisorial fan on $(\Gamma, \ss, \gamma)$, where $\Gamma$ is a smooth
projective $F$-variety such that $k(\Gamma) = k(\tilde{\XX})^{G}$. We denote by $\Omega$ the general spherical orbit of $\XX$.
\\

As a tool to study the geometry of normal $G$-varieties with spherical orbits, we introduce the \emph{contraction morphism.} 
Let $(\D,\F)\in \EE$.
We consider an affine $F$-stable open covering $(U_{i})_{i\in I}$ of the algebraic variety $\Gamma$. Let $\EE_\cont$ be the colored divisorial fan
 generated by $\{(\D_{|U_{i}},\F)\,|\,i\in I\}.$ Note that
the inclusions $C(\D_{|U_{i}})\subseteq C(\D)$ induce a natural $G$-morphism 
$$\pi_{\cont}:X(\EE_{\cont},\gamma)\rightarrow X = X(\D,\F,\gamma).$$
We will also denote by $X(\EE_{\cont})$ the $G$-model of $\tilde{\XX}$ obtained as the
lift of $X(\EE_{\cont},\gamma)$.
The next result collects some properties about the contraction map $\pi_{\cont}$. The reader is referred to \cite[Theorem 3.1]{AH06} for the case of torus actions.
\begin{proposition}\label{p-presol}
The map $\pi_{\cont}$ is a $G$-equivariant proper birational morphism and does not depend on the 
choice of the open covering $(U_{i})_{i\in I}$. The $G$-variety $X(\EE_{\cont},\gamma)$ admits a global quotient and we 
have a commutative diagram
$$\xymatrix  { 
    X(\EE_{\cont}, \gamma)  \ar[rr]^{\pi_{\cont}} \ar[rd] && X \ar@{-->}[ld] \\ & \Gamma/F }
$$     
where $X\dashrightarrow \Gamma/F$ is the rational quotient induced by the inclusion
of function fields $k(X)^{G}\subseteq k(X)$.
\end{proposition}
\begin{proof}
Using the local structure theorem (see \ref{p-loc}), we observe that $\pi_{\cont}$ 
is locally described by the morphism $q$ introduced in Lemma \ref{l-prop-lc}. We conclude
by remarking that the properness is a local condition.
\end{proof}
We now introduce the set of $G$-valuations for describing the $G$-divisors of a normal $G$-variety with spherical orbits. We recall that $M_{\gamma}$ is the lattice of $B$-weights of $k(\XX)$ and $N_{\gamma} = {\rm Hom}(M_{\gamma}, \ZZ)$ is the dual.
\\

{\em Vertical valuations.} 
Let us start with a single colored polyhedral divisor $(\D, \F)\in\EE$. We denote by $\Vert(\D)$ the set of pairs $([Y], v)$, where $[Y]$ is an $F$-orbit of a prime divisor of $\loc(\D)$ and $v$ is a vertex of $\D_{Y}$
such that the following conditions hold. The subset $\Gamma(\loc(\D), \mathcal{O}(-Y)\cdot \mathcal{A})$ is the ideal of a prime divisor of $\spec\, A(\loc(\D),\D)$,
where we recall that $$\mathcal{A} = \bigoplus_{m\in\tail(\D)^{\vee}\cap M}\mathcal{O}(\D(m)).$$  
Note that this notion does not depend on the choice of a representative $Y\in [Y]$. For an element of  $([Y], v)\in \Vert(\D)$, the center of the $G$-valuation $[v_{[Y]}, \mu(v)v, \mu(v)]$ of $k(\XX)$ is the generic point of a $G$-cycle $D_{[Y], v}\subseteq X(\D,\F,\gamma)$ (see Lemma \ref{l:valcenter} and the discussion in Section \ref{s-classif} for the Galois action on the set of $G$-valuations), where $\mu(v)$ is the smallest integer $\ell\in\ZZ_{>0}$ such that $\ell v\in N_{\gamma}$.
\\

{\em Horizontal valuations.}
For simplicity we denote by the same letter a ray (i.e. a one-dimensional face) of a polyhedral cone of $N_{\QQ}$ and its primitive lattice generator in $N_{\gamma}$. We will confuse these two notions when it is needed. We denote by $\Ray(\D,\F)$ the set of rays $\rho$ of $\sigma = \tail(\D)$ such that $\varrho(\F)\cap \rho = \emptyset$ and such that $\D(m)$ is a big Cartier $\QQ$-divisor for any $m$ in the relative interior of $\sigma^{\vee}\cap\rho^{\perp}$. Similarly, for an element of  $\rho\in \Ray(\D,\F)$, the center of the $G$-valuation $[\cdot, \rho, 0]$ of $k(\XX)$ is the generic point of a $G$-cycle $D_{\rho}\subseteq X(\D,\F,\gamma)$. We define more generally the sets
$$\Vert(\EE) = \bigcup_{(\D,\F)\in \EE}\Vert(\D) \text{ and }\Ray(\EE) = \bigcup_{(\D,\F)\in \EE}\Ray(\D,\F)$$
for the colored divisorial fan $\EE$.
\\

The following is the main upshot of this section. We refer to \cite[Theorem 4.22]{FZ03}, \cite[Corollary 3.15]{PS11}, \cite[Corollary 2.12]{LT16} for former special cases where the theorem was proven.
 
\begin{theorem}
\label{t-divisor1}
Let $X$ be a normal $G$-variety with spherical orbits. Let $\EE$ be a colored divisorial fan  on $(\Gamma, \ss, \gamma)$ describing $X$.
If $\Div(\EE)$ denotes the set of $G$-divisors of  $X$,
then the map
$$\phi:\Vert(\EE)\,\bigsqcup\, \Ray(\EE)\rightarrow \Div(\EE),\,\, ([Y],v)\mapsto D_{[Y], v},\, \rho\mapsto D_{\rho}.$$
is well-defined and bijective. Moreover, 
the divisor class group $\Cl(X)$ is isomorphic to the abelian group
$$ \Cl(\loc(\EE)/F)\oplus \bigoplus_{([Y], v)\in \Vert(\EE)}\ZZ D_{[Y], v}\oplus \bigoplus_{\Ray(\EE)} \ZZ D_{\rho}\oplus \bigoplus_{D\in \F_{\Omega}/F}\ZZ D,$$
modulo the relations
$$[Y] = \sum_{v\in N_{\QQ}, \,([Y], v)\in \Vert(\EE)}\mu(v)D_{[Y], v}\text{ and }$$
$$\sum_{([Y], v)\in \Vert(\EE)}\mu(v)\langle m , v\rangle D_{[Y], v} + 
\sum_{\rho\in \Ray(\EE)}\langle m, \rho\rangle D_{\rho} + \sum_{D\in\F_{\Omega}/F}\langle m, \varrho_{\gamma}(D)\rangle D = 0,$$
where $m\in M_{\gamma}$ and $Y\subseteq\loc(\EE)$ is a prime divisor.
See Section \ref{sec-divfan} for the definition of $\loc(\EE)$.
\end{theorem}
\begin{proof}
Without loss of generality, we may assume that $X  =  X(\D,\F,\gamma)$ is simple. 
Note that the map $\phi$ is clearly injective if it is well-defined. Let $Z$ be a $G$-divisor of $X$. Then $Z$ is the image under $\pi_{\cont}$ of a $G$-divisor $\hat{Z}$ of $\hat{X} := X(\EE_{\cont},\gamma)$ represented by the same valuation $\nu$.

{\em Case 1.} Assume that $\hat{Z}$ is dominantly sent on $\loc(\EE)/F$, or equivalently that $\nu$ is central, i.e., $\nu$ is trivial on $k(S)^{F}$ and
$\nu =[\cdot, \rho, 0]$ for some $\rho\in N_{\gamma}\cap \Vc$.
Hence $\hat{Z}$ lifts uniquely into a $G$-divisor on $X(\EE_{\cont})$ that we denote by the same letter. Let $S_{0}$ be the complement in $\loc(\D)$ of all prime divisors where the polyhedral coefficients of $\D$ are non-trivial. Remark that $S_{0}\times X_{\sigma,\F}\subseteq X(\EE_{\cont})$ is a $G$-stable dense open subset, where $X_{\sigma,\F}$ is is the general fiber of the quotient map $\pi: X(\EE_{\cont})\rightarrow \loc(\D)$. The pair $(\sigma,\F)$ is the colored cone of the spherical $G$-variety $X_{\sigma,\F}$ and $\sigma$ is the tail of $\D$. In addition, $\hat{Z}$ restricted to the open subset $S_{0}\times X_{\sigma,\F}$
is a product of $S_{0}$ and a $G$-divisor on $X_{\sigma,\F}$ (compare with \cite[Lemma 3]{FMSS95}). Thus by using \cite[Lemma 2.4]{Kno91} and Lemma \ref{l:valcenter}, we conclude that $\rho\in \Ray(\EE_{\cont})$.

Again, we write the letter $Z$ for the lift of $Z$ under the quotient map by $F$.
Let us consider the sheaf of $\mathcal{O}_{\loc(\D)}$-algebras
$$\mathcal{A}_{\rho}:= \bigoplus_{\lambda\in\sigma^{\vee}\cap\rho^{\perp}\cap M}\mathcal{O}(\D(\lambda))\otimes V_{\lambda}$$
such that $Z\cap X_{0} = P_{u}\times \spec\, \Gamma(\loc(\D), \mathcal{A}_{\rho})$ for some $B$-chart $X_{0}\subseteq X$ intersecting $Z$ (see Lemma \ref{l-L-orbits}). 
Now the evaluations of $\D$ are big on the relative interior of $\sigma^{\vee}\cap\rho^{\perp}\cap M$ if and only if $$\dim \,Z = \dim\,\Gamma(\loc(\D), \mathcal{A}_{\rho}) = \dim\,\spec_{\loc(\D)}\mathcal{A}_{\rho} = \dim\, \hat{Z}$$ if and only if   
$\dim\, X - \dim\, Z = \dim\, \hat{X} - \hat{Z} =  1$. We finally conclude that $\rho\in\Ray(\EE)$ and $Z  = D_{\rho}$.
Our analysis shows that the assignment $\rho \mapsto D_{\rho}$ is a well-defined map.

{\em Case 2.} We now pass to the case where the image of $\hat{Z}$ by $\pi$ is not dense. Since $\hat{Z}$ is a divisor, the closure $Y$ of $\pi(\hat{Z})$ in $\loc(\D)$ is also a divisor.  In addition, $\nu  = [v_{Y}, p, \ell]\in \Vh$ with $\ell\neq 0$.
Let $Z_{1}\subseteq X(\EE_{\cont})$ be the $G$-stable subvariety obtained as the union 
of $G$-orbits contained in $$\bigcup_{D\in\F_{\Omega}}D\text{ and set }X_{1} = X(\EE_{\cont})\setminus Z_{1}.$$ Since the codimension of each irreducible component of $Z_{1}$
is at least $2$, we have $X_{1}\cap \hat{Z}\neq \emptyset$. Then $X_{1}$ is described by an uncolored colored divisorial fan on $(\Gamma, \ss)$ that we can suppose to be a singleton $\{(\D_{1},\emptyset)\}$. We see $(\D_{1},\emptyset)$ as an element of  $\EE_{\cont}$. Now by virtue of Theorem \ref{p-loc}, the $B$-chart $X_{0}(\D_{1},\emptyset)$
is expressed as a product $P_{u}'\times \spec\, A(\loc(\D_{1}), \D_{1})$. 
Denoting by $$\pi_{1}: \spec\, A(\loc(\D_{1}), \D_{1})\rightarrow \loc(\D_{1})$$ the quotient map, the $G$-divisor $\hat{Z}$ is equal to the closure of $P_{u}'\times Z_{2}$, where $Z_{2}$ is an irreducible component of $\pi_{1}^{-1}(\pi(\hat{Z})\cap  \loc(\D_{1})).$ According to \cite[Proposition 3.13]{PS11}
we deduce that  $(p, \ell)= (\mu(v)v, \mu(v))$ for some vertex $v$ of a polyhedral 
coefficient of $\D_{1}$ and $([Y], v)\in \Vert(\EE_{\cont})$.   

We denote by the same letter $Z$ a lift of $Z$ under the quotient map by $F$.
Let us consider the sheaf of $\mathcal{O}_{\loc(\D)}$-algebras
$$\mathcal{A}_{Y, v}:= \bigoplus_{\lambda\in\sigma^{\vee}\cap M}\mathcal{O}(\lfloor\D(\lambda)\rfloor - Y)\otimes V_{\lambda}$$
such that $Z\cap X_{0}(\D,\F) = P_{u}\times \spec\, \Gamma(\loc(\D), \mathcal{A}_{Y, v})$ (see Lemma \ref{l-L-orbits}). Now $(v, [Y])\in \Vert(\EE)$ (adapt arguments of the proof of \cite[Proposition 4.11]{HS10}) if and only if 
 $$\dim \,Z = \dim\,\Gamma(\loc(\D), \mathcal{A}_{v, Y}) = \dim\,\spec_{\loc(\D)}\mathcal{A}_{v, Y} = \dim\, \hat{Z},$$
if and only if $Z$ is of codimension one.
Our analysis shows that the assignment $\phi$ is a bijective map.

For the presentation of the divisor class group ${\rm Cl}(X)$ by generators and relations, the proof is based on the same argument as in \cite[Corollary 3.15]{PS11}, \cite[Corollary 2.12]{LT16}. We recall the key argument.
By \cite{FMSS95} every divisor of $X$ is linearly equivalent to a $B$-stable divisor. Hence $\Cl(X)$ is the quotient of the free abelian group of the $B$-stable divisors modulo the subgroup of principal divisors associated with the $B$-eigenfunctions $f\otimes\chi^{m}$ of $k(X)$, where $f\in k(S)^{\star}$ and $\chi^{m}$
is as in \ref{s-chi}. The calculation of ${\rm Cl}(X)$  follows from the expression of $\div(f\otimes\chi^{m})$ in terms of $B$-valuations of $k(X)$ corresponding to $B$-divisors. This finishes the proof of the theorem.
\end{proof}

\section{Canonical class}\label{s-canclass}
In this section, we investigate the canonical class of a normal
$G$-variety with spherical orbits. We will write $\sim$ for the linear equivalence relation between Weil divisors. 
Our main result can be stated as follows (see \cite[Theorem 3.21]{PS11} for the special case of normal varieties with a torus action).
\begin{theorem}
\label{t-canonical}
Let $K_{X}$ be a canonical divisor of a normal $G$-variety $X$ with spherical orbits defined by a colored divisorial fan $\EE$ on $(\Gamma, \ss, \gamma)$.
Then we have 
$$\sharp\, F\cdot K_{X} \sim  -\sharp\, F\left(\sum_{\rho\in \Ray(\EE)} D_{\rho} +\sum_{D\in\F_{X}}a_{D}D - \sum_{([Y],v)\in\Vert(\EE)}\left[\frac{\mu(v)}{r_{[Y],v}}(b_{[Y]} + 1) -1\right]D_{[Y], v}\right),$$
where $K_{\loc(\EE)} = \sum_{Y\subseteq \loc(\EE)}b_{Y}\cdot Y$ is a canonical divisor of the variety $\loc(\EE)$, the number $b_{[Y]}$ stands for $\frac{1}{\sharp [Y]}\sum_{Y\in [Y]}b_{Y}$, $\F_{X}$ is the set of colors of $X$ and $\sharp F a_{D}\in \ZZ_{\geq 1}$ for any $D\in\F_{X}$. Moreover, the ramification index $r_{[Y],v}$ is defined as
$$r_{[Y],v} = \sharp \{g\in F\,|\, g\cdot D_{Y,v}\subseteq D_{Y,v}\text{ and }g\cdot x = x\},$$
where $D_{Y,v}\subseteq \gamma^{-1}(D_{[Y],v})$ is an irreducible component and $x$ is a general closed point of $D_{Y,v}$. This number does not depend on the choice of the prime divisor $D_{Y,v}$. 
\end{theorem}
\begin{proof}
We separate the proof into two parts. In the first part, we determine the canonical class of $X$ in the case where $X =\tilde{X}$ is a $G$-model of $\tilde{\XX} = S\times\Omega$. In the second part, we deduce the general case to the first step by applying the Riemann-Hurwitz formula for the quotient map $\gamma: X(\EE)\rightarrow X(\EE,\gamma)$.

{\em Case 1.} Assume that $X$ is a $G$-model of $\tilde{\XX}$. We first make some reduction by taking a local chart and removing closed subsets of codimension $\geq 2$. In this way, we may suppose that $X$ is smooth and
$X$ is determined by an uncolored colored polyhedral divisor $(\D,\emptyset)$
with smooth affine locus. We will consider the two following open subsets of the $G$-variety $X$:
(1) the subset $X_{1} = S_{0}\times X_{\sigma, \emptyset},$ where we remove the special fibers of the quotient map by $G$ and $X_{\sigma, \emptyset}$ is the generic fiber. Here $X_{\sigma, \emptyset}$ is $G$-isomorphic to the spherical embedding of $\Omega$ corresponding to the colored cone $(\sigma, \emptyset)$ (see Example \ref{ex-sphere}) and $S_{0}$ is an open subset of $\loc(\D)$;
(2) The $B$-chart $X_{0}\subseteq X$ associated with $(\D,\emptyset)$ which is
the complement of the union of the colors of $X$. By Theorem \ref{p-loc},
the $B$-chart $X_{0}$ is identified with a product $P_{u}\times \spec\, A(\loc(\D), \D)$, where $P_{u}$ is an affine space. Note that
the complement of $X_{0}\cup X_{1}$ in $X$ is a closed subset with irreducible 
components of codimension at least $2$.

Let $\alpha$ be the exterior product of a basis of the module of differential forms of $P_{u}$ and let $\delta$ be a global section of the canonical bundle of $\loc(\D)$. For a basis $e_{1},\ldots, e_{n}$ of $M$, let $\chi^{e_{1}},\ldots, \chi^{e_{n}}$ be the associated Laurent monomials. Then by the argument
of the proofs of \cite[Section 4.1, Proposition 4.1]{Bri97} and \cite[Theorem 3.21]{PS11}, the differential form
$$\omega = \alpha\wedge \delta\wedge\frac{d\chi^{e_1}}{\chi^{e_{1}}}\wedge\ldots \wedge \frac{d\chi^{e_n}}{\chi^{e_n}}$$
restricts on $X_{0}$, $X_{1}$ and $X_{0}\cap X_{1}$ to generators of the canonical sheaves on $X_{0}$, $X_{1}$ and $X_{0}\cap X_{1}$. Consequently, the differential form $\omega$ is a global section of the canonical bundle of $X$.
Hence for computing a canonical divisor $K_{X}$, it is sufficient
to determine the order $v_D(\omega)$ of $\omega$ along any prime divisor $D$ of $X$. Restricting on the chart $X_{1}$, we have $v_D(\omega)\in\ZZ_{<0}$ and $v_{D_\rho}(\omega) = -1$ for all $\rho\in \Ray(\EE)$
and $D\in \F_{X}$ (compare with \cite[Section 4.1, Proposition 4.1]{Bri97}). Restricting on $X_{0}$, the computation of the order of $\omega$ at $D_{Y,v}$ for $(Y,v)\in \Vert(\EE)$ follows from the argument of the proof of \cite[Theorem 3.21]{PS11}, where the notation $\Vert(\EE)$ is considered for the trivial $F$-action.
Furthermore, the order of $\omega$ along a prime divisor which is not $B$-stable is equal to $0$. We finally obtain the formula $$K_{X} = -\sum_{\rho\in \Ray(\EE)} D_{\rho} + \sum_{(Y,v)\in\Vert(\EE)}(\mu(v)(b_{Y} + 1) -1)D_{Y, v} - \sum_{D\in\F_{X}}a_{D}D,$$
where $K_{\loc(\D)} = \sum_{Y\subseteq \loc(\D)}b_{Y}\cdot Y$ is a canonical divisor of $\loc(\D)$ and $a_{D}\in \ZZ_{\geq 1}$ for any $D\in\F_{X}$.

{\em Case 2.}
Let us assume that $X$ is a $G$-model of $\XX =\tilde{\XX}/F$ with colored divisorial fan $\EE$ defined on $(\Gamma, \ss, \gamma)$. Consider the quotient map
$\gamma: X(\EE)\rightarrow X$ by $F$. By the Riemann-Hurwitz formula, we have
$K_{X(\EE)} \sim \gamma^{\star}K_{X} + R$, where $R = \sum_{i\in I}(r_{i}-1)R_{i}$
is a divisor supported on the ramification locus of $\gamma$ and $r_{i}$ is precisely the ramification index attached to the prime divisor $R_{i}$.
We recall that the ramification locus is the smallest closed subset $Z\subseteq X(\EE)$ such that $\gamma$ is \'etale on $X(\EE)\setminus Z$ and so
$R$ is $G$-stable. 
 Now using \cite[Theorem 2.18, page 271]{Liu02} we obtain that
$$\sharp\, F\cdot K_{X} = \gamma_{\star}\gamma^{\star}K_{X} \sim \gamma_{\star}K_{X(\EE)} - \gamma_{\star}R.$$

{\em Claim. The $F$-action is free on the general points of the divisors $D_{\rho}$ and $D\in\F_{X(\EE)}$.}
Indeed, by making the same reduction as in Case 1, we may find a $G\times F$-stable dense open subset of the form $S_{0}\times X_{\sigma, \emptyset}$
such that the $F$-action is free. We conclude the proof of the claim by remarking that the open subset $S_{0}\times X_{\sigma, \emptyset}$ intersects any divisor $D_{\rho}$ and any color $D\in\F_{X(\EE)}$.

So the ramification divisor $R$ is determined by a finite number of elements of $\Vert(\EE)$. By a direct computation we have 
$$\gamma_{\star}R = \sum_{([Y],v)\in\Vert(\EE)}\left(\sum_{D_{Y,v}\subseteq\gamma^{-1}(D_{[Y], v})} [k(D_{Y,v}):\gamma^{\star}k(D_{[Y], v})](r_{Y,v}-1)\right)D_{[Y], v},$$
where  $r_{Y,v}$ is the ramification index of $D_{Y, v}$ and
$k(D_{[Y], v})$, $k(D_{Y,v})$ are the residue fields of the generic points of $D_{[Y], v}$ and $D_{Y,v}$, respectively. Since the $F$-action transitively permutes the irreducible components of $\gamma^{-1}(D_{[Y], v})$ for any $([Y], v)\in \Vert(\EE)$, we have $r_{Y,v} = r_{Y',v}$ and 
$$[k(D_{Y,v}):\gamma^{\star}k(D_{[Y], v})] = [k(D_{Y',v}):\gamma^{\star}k(D_{[Y], v})]
\text{ for all prime divisors } D_{Y,v}, D_{Y',v} \text{ in }\gamma^{-1}(D_{[Y], v}).$$  
Moreover, the formula involving ramification indices and inertial degrees gives
$$\sharp\, F = [k(X(\EE)):\gamma^{\star}k(X)] = \sum_{D_{Y,v}\subseteq\gamma^{-1}(D_{[Y], v})} [k(D_{Y,v}):\gamma^{\star}k(D_{[Y], v})]r_{Y,v}.$$
Let us translate this formula in terms of the $F$-action on $X(\EE)$. We denote by $F_{Y,v}$ the subgroup of $F$ which preserves $D_{Y,v}$. Picking a general closed point $x$ in $D_{Y,v}$, we write ${\rm Stab\,}D_{Y,v}$ for the stabilizer at $x$ of the $F_{Y,v}$-action on $D_{Y,v}$. Then using Lemma \ref{l-finite group} (iv) we have 
$$[k(D_{Y,v}):\gamma^{\star}k(D_{[Y], v})] = [k(D_{Y,v}):k(D_{Y,v})^{F_{Y,v}/{\rm Stab\,}D_{Y,v}}] = \frac{\sharp F_{Y,v}}{\sharp{\rm Stab\,}D_{Y,v}}\text{ and } \sharp [Y] = \frac{\sharp F}{\sharp F_{Y,v}}.$$
Hence the equality $\sharp F = \sharp [Y] \cdot [k(D_{Y,v}):\gamma^{\star}k(D_{[Y], v})] \cdot r_{Y,v}$ by the argument above yields $r_{Y,v} = \sharp{\rm Stab\,}D_{Y,v}$, as required. Letting $r_{[Y],v} = r_{Y,v}$, we finally arrive at 
$$\gamma_{\star}R = \sum_{([Y],v)\in\Vert(\EE)} (\sharp\, F - \sharp\,F/r_{[Y],v} )D_{[Y], v}. $$
By Case 1, it follows that $\gamma_{\star}K_{X(\EE)}$ is equal to 
$$-\sharp\, F\left(\sum_{\rho\in \Ray(\EE)} D_{\rho} +\sum_{D\in\F_{X}}a_{D}D\right)  + \sum_{([Y],v)\in\Vert(\EE)}\left(\mu(v)\frac{\sharp\, F}{r_{[Y],v}}\left(b_{[Y]} + 1\right) -\frac{\sharp\, F}{r_{[Y],v}}\right)D_{[Y], v},$$
where for any $D\in \F_{X}$ we let $a_{D} = \frac{1}{\sharp \gamma^{-1}(D)} \sum_{D'\subseteq \gamma^{-1}(D)}a_{D'}$. 
The difference $\gamma_{\star}K_{X(\EE)} - \gamma_{\star}R$ and the preceding computations give the desired formula.
\end{proof}
\begin{remark}
The coefficient $a_{D}$ in Theorem \ref{t-canonical} can be explicitly determined in terms
of the homogeneous spherical datum $\ss$ of $\Omega$. We refer to \cite[Theorem 4.2]{Bri97}
for more details.
\end{remark}
The next proposition gives an example where there are no ramification divisors.
\begin{proposition}
With the same notation as in Theorem \ref{t-canonical}, if the group $F$ acts freely on $\Gamma$, then $r_{[Y], v} = 1$ for all $([Y], v)\in \Vert(\EE)$.
\end{proposition}
\begin{proof}
Note that under our assumption the $F$-action on the variety $X(\EE_{\cont})$
is free. Indeed, if $g\in F$ belongs to the stabilizer of a point $x\in X(\EE_{\cont})$, then we have $$g\cdot \pi(x) = \pi(g\cdot x) = \pi(x),$$ where $\pi: X(\EE_{\cont})\rightarrow \Gamma$ is the quotient map by $G$.  Since the $F$-action on $\Gamma$ is free, we deduce that $g = 1$. Now there exists a $(G\times F)$-stable dense open subset of $X(\EE)$ with complement of codimension $\geq 2$ and $(G\times F)$-isomorphic to the one of $X(\EE_{\cont})$. 
This implies that the quotient map $\gamma: X(\EE)\rightarrow X$ has no ramification divisor.
\end{proof}

\begin{footnotesize}

\end{footnotesize}

\end{document}